\documentclass[10pt]{article}
\usepackage{etex}
\usepackage{mathtools}
\usepackage{enumitem}

\usepackage{amsmath,amsthm}
\usepackage{csquotes}
\usepackage{todonotes}

\usepackage{anyfontsize}

\makeatletter
\let\orgdescriptionlabel\descriptionlabel
\renewcommand*{\descriptionlabel}[1]{
	\let\orglabel\label
	\let\label\@gobble
	\phantomsection
	\edef\@currentlabel{#1} 	
	\let\label\orglabel
	\orgdescriptionlabel{#1}
}
 											
\def\th@plain{
	\thm@notefont{}
	\itshape
}
\def\th@definition{
	\thm@notefont{}
	\normalfont
}

\g@addto@macro\th@remark{\thm@headpunct{}}
\g@addto@macro\th@definition{\thm@headpunct{}}
\g@addto@macro\th@plain{\thm@headpunct{}}
\makeatother

\usepackage{amssymb}

\usepackage{graphicx}

\usepackage{subfig}
\usepackage[final]{showkeys}

\usepackage{etoolbox}

\usepackage{wasysym}

\usepackage{mathrsfs}

\usepackage{mathpazo}

\usepackage[titletoc]{appendix}
\usepackage[doc]{optional}

\usepackage{soul}

\usepackage{colortbl,booktabs,sectsty,multirow}
\usepackage{xcolor}

\usepackage{cancel}

\usepackage{empheq}

\definecolor{myblue}{rgb}{.8, .8, 1}
  \newcommand*\mybluebox[1]{
    \colorbox{myblue}{\hspace{1em}#1\hspace{1em}}}

\usepackage[obeyspaces,hyphens,spaces]{url}

\usepackage{hyperref}
\hypersetup{
    colorlinks=true,
    linkcolor=blue,
    citecolor=blue,
    filecolor=magenta,
    urlcolor=cyan
}

\usepackage[
  open,
  openlevel=2,
  atend,
  numbered
]{bookmark}

\usepackage[capitalize, nameinlink, noabbrev]{cleveref}
\crefname{equation}{}{}
\crefname{chapter}{Chapter}{Chapters}
\crefname{item}{item}{items}
\crefname{figure}{Figure}{Figures}
\crefname{theorem}{Theorem}{Theorems}
\crefname{lemma}{Lemma}{Lemmas}
\crefname{proposition}{Proposition}{Propositions}
\crefname{corollary}{Corollary}{Corollarys}
\crefname{definition}{Definition}{Definitions}
\crefname{fact}{Fact}{Facts}
\crefname{example}{Example}{Examples}
\crefname{algorithm}{Algorithm}{Algorithms}
\crefname{remark}{Remark}{Remarks}
\crefname{note}{Note}{Notes}
\crefname{notation}{Notation}{Notations}
\crefname{case}{Case}{Cases}
\crefname{exercise}{Exercise}{Exercises}
\crefname{question}{Question}{Questions}
\crefname{claim}{Claim}{Claims}
\crefname{enumi}{}{}

\usepackage[top= 2cm, bottom = 2 cm, left = 2.2 cm, right= 2.2 cm]{geometry}

\usepackage{float}

\usepackage{pgf}

\usepackage{array}
\usepackage{tabu}

\allowdisplaybreaks

\numberwithin{equation}{section}

\theoremstyle{plain}
\newtheorem{theorem}{Theorem}[section]

\newtheorem{fact}[theorem]{Fact}
\newtheorem{lemma}[theorem]{Lemma}
\newtheorem{proposition}[theorem]{Proposition}

\theoremstyle{definition}
\newtheorem{definition}[theorem]{Definition}

\newtheorem{remark}[theorem]{Remark}

\newcommand{\inte}{\ensuremath{\operatorname{int}}}

\newcommand{\cont}{\ensuremath{\operatorname{cont}}}

\newcommand{\lev}{\ensuremath{\operatorname{lev}}}
\newcommand{\spn}{\ensuremath{{\operatorname{span} \,}}}

\newcommand{\dom}{\ensuremath{\operatorname{dom}}}

\newcommand{\Fix}{\ensuremath{\operatorname{Fix}}}
\newcommand{\Id}{\ensuremath{\operatorname{Id}}}

\newcommand{\Pro}{\ensuremath{\operatorname{P}}}

\newcommand{\I}{\ensuremath{\operatorname{I}}}
\newcommand{\J}{\ensuremath{\operatorname{J}}}

\newcommand{\N}{\ensuremath{\operatorname{N}}}
\newcommand{\sri}{\ensuremath{\operatorname{sri}}}

\newcommand{\cone}{\ensuremath{\operatorname{cone}}}
\newcommand{\minimize}{\ensuremath{\operatorname{minimize}}}
\newcommand{\subject}{\ensuremath{\operatorname{subject~ to}}}

\providecommand{\norm}[1]{\lVert#1\rVert}

\providecommand{\innp}[1]{\langle#1\rangle}
\providecommand{\Innp}[1]{\Big\langle#1\Big\rangle}

\newcommand\scalemath[2]{\scalebox{#1}{\mbox{\ensuremath{\displaystyle #2}}}}

\begin{document}

\title{ \sffamily Projecting onto intersections of  halfspaces and hyperplanes}

\author{
         Hui\ Ouyang\thanks{
                 Mathematics, University of British Columbia, Kelowna, B.C.\ V1V~1V7, Canada.
                 E-mail: \href{mailto:hui.ouyang@alumni.ubc.ca}{\texttt{hui.ouyang@alumni.ubc.ca}}.}
                 }

\date{September 21, 2020}

\maketitle

\begin{abstract}
\noindent
It is well-known that the sequence of iterations of the composition of projections onto closed affine subspaces  converges linearly to the projection onto the intersection of the affine subspaces 
when the sum of   the corresponding   linear subspaces is closed. 
Inspired by this,  in this work, we systematically study 
the relation between  the projection onto intersection of halfspaces and hyperplanes, and the composition of projections onto halfspaces and hyperplanes.
 In addition, as  by-products, we provide the   Karush-Kuhn-Tucker   conditions for characterizing the optimal solution of convex optimization with finitely many equality and inequality constraints in Hilbert spaces and construct an explicit formula for the projection onto the intersection of hyperplane and halfspace.
\end{abstract}

{\small

}

\section{Introduction} \label{sec:Introduction}
Throughout this paper, we assume that
\begin{empheq}[box = \mybluebox]{equation*}
\text{$\mathcal{H}$ is a real Hilbert space},
\end{empheq}
with inner product $\innp{\cdot,\cdot}$ and induced norm $\|\cdot\|$.

Throughout the paper, we use the convention that $\mathbb{N}:=\{0,1,2,\ldots\}$. Let $m \in \mathbb{N} \smallsetminus \{0\}$.  For every $ i \in \{1,\ldots, m\}$,  let $u_{i}$  be in $\mathcal{H}$ and let $\eta_{i}$ be in $\mathbb{R}$.  Set 
\begin{empheq}[box=\mybluebox]{equation*} 
W_{i} := \{x \in \mathcal{H} ~:~ \innp{x,u_{i}} \leq \eta_{i} \} \quad \text{and} \quad H_{i} := \{x \in \mathcal{H} ~:~ \innp{x,u_{i}} = \eta_{i} \}.
\end{empheq}

According to Deutsch's \cite[Theorems~9.8 and 9.35]{D2012} and an easy translation argument,  the sequence of  iterations of the  composition of projections onto closed affine subspaces converges linearly to the projection onto the intersection of the affine subspaces when the sum of the linear subspaces which are  parallel to the  affine subspaces is closed. Moreover, intersections of halfspaces and hyperplanes are frequently seen in  constraints of road design problems \cite[Section~2]{BK2015}, least norm problems \cite[Section~5.2]{BCS2019}, constrained regression problems \cite[Page~153]{BV2004}, and least square problems \cite[Page~226]{Meyer2000}.

Note that although the Dykstra's algorithm allows one to compute best approximations from an intersection of finitely many closed convex sets (see, e.g., \cite[Page~207]{D2012}), the practical manipulation of the Dykstra's algorithm is very complicated. Moreover, although there exist  explicit formulae for  $\Pro_{ W_{1} \cap W_{2}}$ and $\Pro_{ H_{1} \cap W_{2}  }$ (see, \cref{fact:W1capW2:u1u2LD}, \cref{fact:Projec:Intersect:halfspace}, \cref{theor:u1u2LD:PBAM}, and \cref{them:Formula:H1capW2} below), given a point $x \in \mathcal{H}$, the formulae of $\Pro_{ W_{1} \cap W_{2}}x$ and $\Pro_{ H_{1} \cap W_{2}  }x$  depend on the linear dependence relation of $u_{1}$ and $u_{2}$, and on the \enquote{region} where the $x$ is located in.  Hence, it is worthwhile to explore other easy   ways to find the best approximation from the intersection of halfspaces and hyperplanes or feasibility point in that intersection.

Inspired by \cite[Theorem~9.8]{D2012}, in this work, our goal is \emph{to study the relation between  the projection onto intersection of halfspaces and hyperplanes, and the composition of projections onto halfspaces and hyperplanes.}

The main results in this work are the following:
\begin{enumerate}
	\item[\textbf{R1:}] \cref{theorem:Halfspaces} summarizes the relation between $\Pro_{W_{1} \cap W_{2}}$ and $ \Pro_{W_{2}}\Pro_{W_{1}} $. In particular, if $u_{1}$ and $u_{2}$ are linear dependent or orthogonal, then $ \Pro_{W_{2}}\Pro_{W_{1}} = \Pro_{W_{1} \cap W_{2}}$. Otherwise, for $\gamma :=\frac{|\innp{u_{1}, u_{2}}| }{\norm{u_{1}} \norm{u_{2}}} \in \left[0,1\right[\,$,   if $\innp{u_{1},u_{2}} <0 $, then $(\forall x \in \mathcal{H})$ $(\forall k \in \mathbb{N})$  $\norm{ (\Pro_{W_{2}}\Pro_{W_{1}})^{k}x - \Pro_{W_{1} \cap W_{2}}x} \leq \gamma^{k} \norm{x - \Pro_{W_{1} \cap W_{2}}x}$; if $\innp{u_{1},u_{2}} >0$, then $(\forall x \in \mathcal{H})$ $\Pro_{W_{2}}\Pro_{W_{1}}x \in W_{1} \cap W_{2}$: particularly, if $ x \in W_{1} \cup W_{2}$, then  $\Pro_{W_{2}}\Pro_{W_{1}}x =\Pro_{W_{1} \cap W_{2}}x $; if $ x \in W^{c}_{1} \cap W^{c}_{2}$ with $ \Pro_{H_{1}}x \in W_{2} $, then  $\Pro_{W_{2}}\Pro_{W_{1}}x=\Pro_{H_{1}}x  =\Pro_{W_{1} \cap W_{2}}x $; if $ x \in W^{c}_{1} \cap W^{c}_{2}$ with $ \Pro_{H_{1}}x \notin W_{2} $, then $\Pro_{W_{2}}\Pro_{W_{1}}x \in W_{1} \cap W_{2} \smallsetminus \{ \Pro_{W_{1} \cap W_{2}}x  \}$.

	\item[\textbf{R2:}] \cref{them:KKT:EQINEQ:Hilbert} states  the KKT conditions associated with convex optimization with finitely many equality and inequality constraints in Hilbert spaces.

	\item[\textbf{R3:}] \cref{them:Formula:H1capW2} shows an explicit formula for the projection onto intersection of hyperplane and halfspace. 
	
	\item[\textbf{R4:}]  \cref{theom:HyperplaneHalfspace:BAM} concludes the relations of $\Pro_{H_{1} \cap W_{2}}$ with $\Pro_{W_{2}}\Pro_{H_{1}}$ and  $\Pro_{H_{1}} \Pro_{W_{2}}$. In particular, if
	 $u_{1}$ and $u_{2}$ are linearly dependent or orthogonal, then $\Pro_{W_{2}}\Pro_{H_{1}} =\Pro_{H_{1} \cap W_{2}} =\Pro_{H_{1}} \Pro_{W_{2}}$; Otherwise, for $\gamma :=\frac{|\innp{u_{1}, u_{2}}| }{\norm{u_{1}} \norm{u_{2}}} \in \left[0,1\right[\,$, $(\forall x \in \mathcal{H})$ $(\forall k \in \mathbb{N})$
	 $\norm{ (\Pro_{W_{2}}\Pro_{H_{1}})^{k}x - \Pro_{H_{1} \cap W_{2}}x} \leq \gamma^{k} \norm{x -\Pro_{H_{1} \cap W_{2}}x }$, and 
	  if only $  \Pro_{H_{1}} \Pro_{W_{2}}x \notin   W_{2}$, then 	$(\forall x \in W_{2})$ $(\forall k \in \mathbb{N})$ $\norm{ \Pro_{W_{2}} (\Pro_{H_{1}}\Pro_{W_{2}} )^{k} x-  \Pro_{H_{1} \cap W_{2}}x   } \leq \gamma^{k} \norm{x -\Pro_{H_{1} \cap W_{2}}x  }$, and $ (\forall x \in W^{c}_{2}) $ $(\forall k \in \mathbb{N})$ $\norm{  (\Pro_{H_{1}}\Pro_{W_{2}} )^{k} - \Pro_{H_{1} \cap W_{2}}x }  \leq \gamma^{k} \norm{x -\Pro_{H_{1} \cap W_{2}}x  }$.				
\end{enumerate}	
	
	In view of  our results mentioned above, if $u_{1}$ and $u_{2}$ are linearly dependent  or  orthogonal, then $\Pro_{ W_{2} }\Pro_{ W_{1} }=\Pro_{ W_{1} \cap W_{2}  }$, and  $\Pro_{ W_{2} }\Pro_{ H_{1} }=\Pro_{ H_{1} }\Pro_{ W_{2} }=\Pro_{ H_{1} \cap W_{2}  }$. Let $x \in \mathcal{H}$. 
	Moreover, the sequence $\left((\Pro_{W_{2}}\Pro_{W_{1}})^{k}x \right)_{k \in \mathbb{N}}$ converges linearly or in one step for finding the desired best approximation point $\Pro_{ H_{1} \cap W_{2}  }x$ or converges in one step for finding the feasibility point in $W_{1} \cap W_{2}$. In addition, the sequence $\left((\Pro_{W_{2}}\Pro_{H_{1}})^{k}x \right)_{k \in \mathbb{N}}$ always converges linearly to the desired best approximation point $\Pro_{ H_{1} \cap W_{2}}x$.

	Although  the result on KKT conditions  is classical and well-known, some KKT conditions are only shown in finite-dimensional spaces  (see, e.g., \cite[Page~244]{BV2004} and \cite[Theorem~28.3]{R1970}), some include only inequality constraints without equality constraints (see, e.g., \cite[Proposition~27.21]{BC2017}, \cite[Theorem~3.78]{AmirBeck},  \cite[Page~249]{Luenberger1969}, and \cite[Pages~94]{Schirotzek2007}), and some act as only necessary optimality conditions (see, e.g., \cite[Theorem~3.78]{AmirBeck}, \cite[Page~249]{Luenberger1969}, and \cite[Pages~94 and 274]{Schirotzek2007}).
The result on KKT conditions presented in \cref{them:KKT:EQINEQ:Hilbert} characterizes the optimal solution of the  convex optimization with finitely many equality and inequality constraints, which  is  a generalization of  the version presented in  \cite[Page~244]{BV2004} from finite-dimensional spaces to Hilbert spaces and from differentiable function to subdifferentiable function, and of the version shown in \cite[Proposition~27.21]{BC2017} from inequality constrains to inequality and equality constants.

	The organization of the rest of the paper is the following.
	We display some auxiliary results  in \cref{sec:Auxiliary}. 	
In \cref{sec:hyperplane}, we collect the fact on the linear convergence of the composition of finitely many projections onto hyperplanes and construct an explicit formula of projection onto finitely many hyperplanes.  
In 	\cref{sec:Halfspaces}, we systematically study the relation between $\Pro_{W_{2}} \Pro_{W_{1}}$ and $\Pro_{W_{1} \cap W_{2}}$.
In	\cref{sec:Projection:HyperpnaeHalfspace}, we aim to  construct   explicit formulae for the projection onto the intersection of hyperplane and halfspace, which plays a critical role for us to investigate the relations of $\Pro_{H_{1} \cap W_{2}}$ with  $\Pro_{W_{2}}\Pro_{H_{1}}$ and  $\Pro_{H_{1}} \Pro_{W_{2}}$   in \cref{sec:com:hyperplane:halfspace}. To this end,
as a by-product, in \cref{sec:Projection:HyperpnaeHalfspace}, we also provide the  KKT conditions associated with convex optimization with finitely many equality and inequality constraints in Hilbert spaces.

We now turn to the notation used in this paper.
Let $D$ be a subset of $\mathcal{H}$. $D^{c}:=\mathcal{H} \smallsetminus D$ is the \emph{complementary set} of $D$. The \emph{interior}  of $D$ is the largest open set that is contained in $D$; it is denoted by $\inte D$.
The \emph{orthogonal complement} of $D$ is the set $ D^{\perp} :=\{x \in \mathcal{H}~:~ (\forall y \in D)~ \innp{x,y}=0 \}$. 
$D$ is an \emph{affine subspace} of $\mathcal{H}$ if   $D \neq \varnothing$ and $(\forall \rho\in\mathbb{R})$ $\rho D +(1-\rho  )D=D$.
In addition, $D$   is a \emph{cone} if $D =\mathbb{R}_{++}D $.  The polar cone of $D$ is $D^{\ominus} := \{ u \in \mathcal{H}~:~ \sup  \innp{D , u} \leq 0  \}$.  The \emph{conical hull} of $D$ is the intersection of all the cones in $\mathcal{H}$ containing $D$, i.e., the smallest cone in $\mathcal{H}$ containing $D$. It is denoted by $\cone D$. 
Let  $C$ be a nonempty convex subset of $\mathcal{H}$. 
	Let $x \in \mathcal{H}$. The \emph{normal cone} to $C$ at $x$ is 
	\begin{align*}
	\N_{C} x :=\begin{cases}
	\{u \in \mathcal{H} ~:~ \sup \innp{C-x, u} \leq 0 \}, \quad &\text{if} ~ x \in C;\\
	\varnothing, \quad & \text{otherwise}.
	\end{cases}
	\end{align*} 
The \emph{strong relative interior} of $C$ is
$\sri C :=\{ x \in C ~:~ \cone (C-x) = \overline{\spn} (C -x)\}$.
Suppose that  $C$ is a nonempty closed convex subset of $\mathcal{H}$. The \emph{projector} (or \emph{projection operator}) onto $C$ is the operator, denoted by
$\Pro_{C}$,  that maps every point in $\mathcal{H}$ to its unique projection onto $C$.

Let $\mathcal{K}$ be a real Hilbert space. Denote by $\mathcal{B} (\mathcal{H}, \mathcal{K}) := \{ T: \mathcal{H} \rightarrow \mathcal{K} ~:~ T ~\text{is  linear and bounded} \}$. Let $T \in \mathcal{B} (\mathcal{H}, \mathcal{K}) $. The adjoint of $T$ is the unique operator $T^{*} \in \mathcal{B} (\mathcal{K},\mathcal{H}) $ that satisfies $(\forall x \in \mathcal{H})$ $(\forall y \in \mathcal{K})$ $\innp{Tx,y} =\innp{x,T^{*}y}$.
 Let $f : \mathcal{H} \to \left]-\infty, +\infty\right]$ be \emph{proper}, that is, $\dom f :=\{x \in \mathcal{H} ~:~ f(x) < +\infty\} \neq \varnothing$. Denote the \emph{domain of continuity} of $f$ by $\cont f:= \{x \in \mathcal{H} ~:~ f(x) \in \mathbb{R}~\text{and}~f~\text{is continuous at} ~x  \}$.
The  \emph{subdifferential} of $f$ is the set-valued operator 
\begin{align*}
\partial f: \mathcal{H} \to 2^{\mathcal{H}} : x \mapsto \{ u \in \mathcal{H} ~:~ (\forall y \in \mathcal{H})  \innp{y-x, u} +f(x) \leq f(y) \}.
\end{align*}
The indicator function of a subset $A$ of $\mathcal{H}$ is the function $\iota_{A}:\mathcal{H} \to \left]-\infty, +\infty\right] : x \mapsto \begin{cases}
0, \quad \text{if } x \in A;\\
+\infty, \quad \text{otherwise}.  
\end{cases} $
Denote by $\Gamma_{0} (\mathcal{H})$ the set of all proper lower semicontinuous  convex functions from $\mathcal{H}$ to $\left]-\infty, +\infty\right]$.
Let $u$ be in $\mathcal{H}$. Then $\ker u := \{x \in
\mathcal{H} ~:~ \innp{x,u} =0 \}$ is the \emph{kernel of $u$}. The \emph{set of fixed
	points of the operator $T: \mathcal{H} \to \mathcal{H}$} is denoted by $\Fix T$, i.e., $\Fix T := \{x \in \mathcal{H} ~:~ Tx=x\}$.

For other notation not explicitly defined here, we refer the reader to \cite{BC2017}.

\section{Auxiliary results} \label{sec:Auxiliary}
In this section, we provide some results to be used in the sequel.

\subsection*{Linearly independent vectors}
The following well-known results will be used frequently in our proofs later. For completeness, we attach the easy proof below as well.

\begin{fact} \label{lem:basic:u1u2:LinerDep}
	Let $u_{1}$ and $u_{2}$ be in $\mathcal{H}$. The following statements hold.
	\begin{enumerate}
		\item \label{item:lem:basic:u1u2:LinerDep:LD:u1neq0}  Suppose  $u_{1} \neq 0$. Then $ u_{1}, u_{2} $ are linearly dependent $\Leftrightarrow $ $u_{2} =\frac{\innp{u_{2},u_{1}}}{ \innp{u_{1}, u_{1}}} u_{1}$  $\Leftrightarrow $  $\norm{u_{1}} \norm{u_{2}} = | \innp{u_{1}, u_{2}}|$   $\Leftrightarrow $ 
		either $u_{2} = \frac{\norm{u_{2}}}{\norm{u_{1}}} u_{1}$ or $u_{2} = -\frac{\norm{u_{2}}}{\norm{u_{1}}} u_{1}$. 
		\item \label{item:lem:basic:u1u2:LinerDep:LD} $ u_{1}, u_{2} $ are linearly dependent $\Leftrightarrow$ $\norm{u_{1}} \norm{u_{2}} = | \innp{u_{1}, u_{2}}| $. 
		\item \label{item:lem:basic:u1u2:LinerDep:LID} $\norm{u_{1}} \norm{u_{2}}  > | \innp{u_{1}, u_{2}}| $ if and only if $ u_{1}, u_{2} $ are  linearly independent.
	\end{enumerate}
\end{fact}

\begin{proof}
	
	\cref{item:lem:basic:u1u2:LinerDep:LD:u1neq0}: Because $u_{1} \neq 0$,  we know that $ u_{1}, u_{2} $ are linearly dependent if and only if $u_{2} =\beta u_{1}$ for some $\beta \in \mathbb{R} $. Take inner product with $u_{1}$ for both sides of $u_{2} =\beta u_{1}$ to obtain that $\innp{u_{1},u_{2}} =\beta \innp{u_{1}, u_{1}}$, which implies that $\beta=\frac{\innp{u_{1},u_{2}}}{ \innp{u_{1}, u_{1}}}$. Hence, $ u_{1}, u_{2} $ are linearly dependent if and only if $u_{2} =\frac{\innp{u_{1},u_{2}}}{ \innp{u_{1}, u_{1}}} u_{1}$. 
	
	On the other hand, $0 \leq \Innp{u_{2} -\frac{\innp{u_{1},u_{2}}}{ \innp{u_{1}, u_{1}}}  u_{1}, u_{2} -\frac{\innp{u_{1},u_{2}}}{ \innp{u_{1}, u_{1}}}  u_{1}} = \innp{u_{2}, u_{2}} - \frac{\innp{u_{1},u_{2}}^{2}}{ \innp{u_{1}, u_{1}}} =\norm{u_{2}}^{2} -\frac{ \innp{u_{1},u_{2}}^{2} }{\norm{u_{1}}^{2}}$. Hence, 
	$\norm{u_{1}}^{2}\norm{u_{2}}^{2} = | \innp{u_{1}, u_{2}}|^{2}$ if and only if $u_{2} =\frac{\innp{u_{2},u_{1}}}{ \innp{u_{1}, u_{1}}} u_{1}$. 	Altogether, \cref{item:lem:basic:u1u2:LinerDep:LD:u1neq0} is true.

	\cref{item:lem:basic:u1u2:LinerDep:LD}: If $u_{1}=0$, then   $\norm{u_{1}} \norm{u_{2}}  =0= | \innp{u_{1}, u_{2}}| $ and $ u_{1}, u_{2} $ are linearly dependent.  So,  \cref{item:lem:basic:u1u2:LinerDep:LD} follows from  \cref{item:lem:basic:u1u2:LinerDep:LD:u1neq0}.

	\cref{item:lem:basic:u1u2:LinerDep:LID}: By Cauchy-Schwartz inequality,  $\norm{u_{1}} \norm{u_{2}} \geq  | \innp{u_{1}, u_{2}}| $. Hence,   \cref{item:lem:basic:u1u2:LinerDep:LD} implies 	\cref{item:lem:basic:u1u2:LinerDep:LID}.
\end{proof}

\begin{fact} {\rm \cite[Theorem~6.5-1]{Kreyszig1989}} \label{fact:Gram:inver}  
	Let $ a_{1}, \ldots, a_{m} $ be points in $\mathcal{H}$. Then  $a_{1}, \ldots, a_{m}$ are linearly independent  if and only if the  Gram matrix
	\begin{align} \label{eq:GramMatrix}
	G(a_{1}, \ldots, a_{m}) :=
	\begin{pmatrix} 
	\norm{a_{1}}^{2} &\innp{a_{1},a_{2}} & \cdots & \innp{a_{1}, a_{m}}  \\ 
	\innp{a_{2},a_{1}} & \norm{a_{2}}^{2} & \cdots & \innp{a_{2},a_{m}} \\
	\vdots & \vdots & ~~& \vdots \\
	\innp{a_{m},a_{1}} & \innp{a_{m},a_{2}} & \cdots & \norm{a_{m}}^{2} \\
	\end{pmatrix} 
	\end{align}
	is invertible.
\end{fact}

\begin{lemma} \label{lem:capMi}
	Let $a_{1}, \ldots, a_{m} $ be linearly independent points in $\mathcal{H}$ and let  $\xi_{1}, \ldots, \xi_{m} $ be in $\mathbb{R}$. For every $( i \in \{1, \ldots, m\})$, set   $M_{i} :=\{x \in \mathcal{H} ~:~ \innp{x, a_{i}} = \xi_{i} \}$. Then $\cap^{m}_{i=1} M_{i}  \cap \spn \{a_{1}, \ldots, a_{m} \} $ is a singleton. Consequently, $\cap^{m}_{i=1} M_{i} \neq \varnothing$. 
\end{lemma}

\begin{proof}
	According to \cref{fact:Gram:inver},
	the  Gram matrix $G(a_{1}, \ldots, a_{m})$ defined as \cref{eq:GramMatrix}
	is invertible. 
	Let $\beta_{1}, \ldots, \beta_{m} \in \mathbb{R}^{m}$.  Then,
	\begin{align*}
	\sum^{m}_{i=1} \beta_{i} a_{i}  \in \cap^{m}_{i=1} M_{i}  \cap	\spn \{a_{1}, \ldots, a_{m} \} 
	& \Leftrightarrow  \begin{cases}
	\innp{ \sum^{m}_{i=1} \beta_{i} a_{i} , a_{1}} =\xi_{1}\\
	\quad \quad \quad \vdots \\
	\innp{ \sum^{m}_{i=1} \beta_{i} a_{i} , a_{m}} =\xi_{m}\\
	\end{cases} \\
& \Leftrightarrow 
	\begin{pmatrix} 
	\norm{a_{1}}^{2} &\innp{a_{1},a_{2}} & \cdots & \innp{a_{1}, a_{m}}  \\ 
	\vdots & \vdots & ~~& \vdots \\
	\innp{a_{m},a_{1}} & \innp{a_{m},a_{2}} & \cdots & \norm{a_{m}}^{2} \\
	\end{pmatrix}  
	\begin{pmatrix}
	\beta_{1}\\
	\vdots\\
	\beta_{m}
	\end{pmatrix} = 
	\begin{pmatrix}
	\xi_{1}\\
	\vdots\\
	\xi_{m}
	\end{pmatrix} \\
& \Leftrightarrow  \begin{pmatrix}
	\beta_{1}\\
	\vdots\\
	\beta_{m}
	\end{pmatrix} =  	G(a_{1}, \ldots, a_{m}) ^{-1}
	\begin{pmatrix}
	\xi_{1}\\
	\vdots\\
	\xi_{m}
	\end{pmatrix},
	\end{align*}
	which implies that $( a_{1}, \cdots, a_{m}) G(a_{1}, \ldots, a_{m}) ^{-1}
	(\xi_{1}, \cdots, \xi_{m} )^{\intercal} \in \cap^{m}_{i=1} M_{i} \cap	\spn \{a_{1}, \ldots, a_{m} \}  \neq \varnothing$.
\end{proof}

\subsection*{Best approximation mappings and projections}
\begin{definition} \label{def:BAM}  {\rm \cite[Definition~3.1]{BOyW2020BAM}}
	Let $G: \mathcal{H} \to \mathcal{H}$, and let $\gamma \in \left[0,1\right[\,$.
	Then $G$ is a \emph{best approximation mapping with constant $\gamma$} (for short $\gamma$-BAM), if
	\begin{enumerate}
		\item \label{def:BAM:Fix}  $\Fix G$ is a nonempty closed  convex subset of $\mathcal{H}$,
		\item  \label{def:BAM:eq} $\Pro_{\Fix G}G=\Pro_{\Fix G}$, and
		\item  \label{def:BAM:Ineq} 
		$(\forall x \in \mathcal{H})$ $\norm{Gx -\Pro_{\Fix G}x} \leq \gamma \norm{x - \Pro_{\Fix G}x}$.
	\end{enumerate}	
	In particular, if $\gamma$ is unknown or not necessary to point out, we just say that $G$ is a BAM.
\end{definition}	

The following result plays an important role in the proofs  of some of our main results later. 

\begin{fact} {\rm  \cite[Proposition~3.10]{BOyW2020BAM}} \label{fact:BAM:Properties}
	Let $\gamma \in \left[0,1\right[\,$ and let $G: \mathcal{H} \to \mathcal{H}$. Suppose that $G$ is a $\gamma$-BAM. Then $(\forall x \in \mathcal{H})$ $(\forall k \in \mathbb{N})$  $\norm{G^{k}x -\Pro_{\Fix G}x} \leq \gamma^{k} \norm{x- \Pro_{\Fix G}x}$.
\end{fact}

\begin{fact} {\rm \cite[Example~29.20]{BC2017}} \label{fact:Projec:halfspaces}
	Let $u \in \mathcal{H} \smallsetminus \{0\}$, let $\eta \in \mathbb{R}$, and set $W:=\{x \in \mathcal{H} ~:~ \innp{x,u} \leq \eta\}$.  Then $W \neq \varnothing$ and 
	\begin{align*}
	(\forall x \in \mathcal{H}) \quad \Pro_{W} x =\begin{cases}
	x, \quad  &\text{if } \innp{x,u} \leq \eta;\\
	x+ \frac{\eta - \innp{x,u}}{\norm{u}^{2}}u, \quad &\text{if } \innp{x,u} > \eta.
	\end{cases}
	\end{align*}	
\end{fact}

\begin{fact} {\rm \cite[Fact~1.8]{DH2006II}}  \label{fact:AsubseteqB:Projection}
	Let $A$ and $B$ be two nonempty closed convex subsets of $\mathcal{H}$. Let $A \subseteq B $ and $x \in \mathcal{H}$. Then  $\Pro_{B}x \in A$ if and only if $ \Pro_{B}x  =\Pro_{A}x $.
\end{fact}

\begin{fact} {\rm \cite[Example~29.18]{BC2017}} \label{fact:Projec:Hyperplane}
	Suppose that $u \in \mathcal{H} \smallsetminus \{0\}$, let $\eta \in \mathbb{R}$, and set $H := \{ x \in \mathcal{H} ~:~ \innp{x,u} = \eta \}$. Then 
	\begin{align*}
	(\forall x \in \mathcal{H}) \quad \Pro_{H} x = x+ \frac{\eta - \innp{x,u}}{\norm{u}^{2}}u.
	\end{align*}
\end{fact}	
\begin{remark} \label{remark:FactWFactH}
	Let $u \in \mathcal{H} \smallsetminus \{0\}$ and $\eta \in \mathbb{R}$. Set $W:=\{x \in \mathcal{H} ~:~ \innp{x,u} \leq \eta\}$ and $H := \{ x \in \mathcal{H} ~:~ \innp{x,u} = \eta \}$.  If $\eta >0 $, then $\frac{2 \eta}{ \norm{u}^{2}} u \in W^{c}$. Otherwise, $u \in W^{c}$ and $W^{c} \neq \varnothing$. Moreover, 
	combine \Cref{fact:Projec:halfspaces,fact:Projec:Hyperplane} to see that 
	$(\forall x \in W^{c})$ $\Pro_{W}x =x+ \frac{\eta - \innp{x,u}}{\norm{u}^{2}}u= \Pro_{H}x$. 
\end{remark}

\begin{fact} \label{fact:ExchangeProj}{\rm \cite[Lemma~2.3]{BOyW2020BAM}}
	Let $M$ and $N$ be closed affine subspaces of $\mathcal{H}$ with $M \cap N \neq \varnothing$. Assume  $M \subseteq N$ or $N \subseteq M$. Then $\Pro_{M}\Pro_{N}=\Pro_{N}\Pro_{M} =\Pro_{M \cap N} $.
\end{fact}

Recall that $u_{1}$ and $u_{2}$ are in $\mathcal{H}$ and $\eta_{1}$ and $\eta_{2}$ are in  $\mathbb{R}$, and that 
\begin{subequations}
	\begin{align}
	&W_{1} := \{x \in \mathcal{H} ~:~ \innp{x,u_{1}} \leq \eta_{1} \},   \quad 	W_{2} := \{x \in \mathcal{H} ~:~ \innp{x,u_{2}} \leq \eta_{2} \}, \label{eq:W1W2}\\
&H_{1} := \{x \in \mathcal{H} ~:~ \innp{x,u_{1}}= \eta_{1} \} , \quad   H_{2} := \{x \in \mathcal{H} ~:~ \innp{x,u_{2}} = \eta_{2} \}.	\label{eq:H1H2}
	\end{align}
\end{subequations}
The following result will be used frequently later. 
\begin{lemma} \label{cor:PH2PH1}
	Suppose that  $u_{1} \neq 0$ and $u_{2} \neq 0$. Let $x \in \mathcal{H}$. Then the following hold:
	\begin{enumerate}
		\item \label{cor:PH2PH1:H2H1} We have the identities:
		\begin{subequations}
			\begin{align*}
			\Pro_{H_{2}}\Pro_{H_{1}}x &= x +\frac{\eta_{1} -\innp{x,u_{1}}}{\norm{u_{1}}^{2}} u_{1} +\frac{\eta_{2} -\innp{x,u_{2}}}{\norm{u_{2}}^{2}} u_{2} -\frac{ (\eta_{1} -\innp{x,u_{1}}) \innp{u_{1},u_{2}}}{ \norm{u_{1}}^{2} \norm{u_{2}}^{2}} u_{2}   \\
			&=x +\frac{\eta_{1} -\innp{x,u_{1}}}{\norm{u_{1}}^{2}} u_{1} + \frac{1}{ \norm{u_{1}}^{2}  \norm{u_{2}}^{2} } \left( (\eta_{2} -\innp{x,u_{2}}) \norm{u_{1}}^{2} - (\eta_{1} -\innp{x,u_{1}}) \innp{u_{1},u_{2}}  \right)u_{2}.  
			\end{align*}
		\end{subequations}
	\item \label{cor:PH2PH1:notin:EQ} $\Pro_{H_{1}}x \notin W_{2} \Leftrightarrow \norm{u_{1}}^{2} (\innp{x,u_{2}} - \eta_{2}) > \innp{u_{1}, u_{2}} (\innp{x,u_{1}} -\eta_{1}) $. Moreover, $\Pro_{H_{2}}x \notin W_{1} \Leftrightarrow \norm{u_{2}}^{2} (\innp{x,u_{1}} - \eta_{1}) > \innp{u_{1}, u_{2}} (\innp{x,u_{2}} -\eta_{2}) $.
		\item  \label{cor:PH2PH1:H2H1:=0} Suppose that $\innp{u_{1},u_{2}}=0$. Then $(\forall x \in \mathcal{H})$	$\Pro_{H_{2}}\Pro_{H_{1}}x  = x +\frac{\eta_{1} -\innp{x,u_{1}}}{\norm{u_{1}}^{2}} u_{1} +\frac{\eta_{2} -\innp{x,u_{2}}}{\norm{u_{2}}^{2}} u_{2} $. Moreover, $ \Pro_{H_{2}}\Pro_{H_{1}}  =\Pro_{ H_{1} \cap H_{2}}$.
		\item \label{cor:PH2PH1:H2} Suppose that $u_{1}$ and $u_{2}$ are linearly dependent. Then $\Pro_{H_{2}}\Pro_{H_{1}}=\Pro_{H_{2}} $.
		\item \label{cor:PH2PH1:inW1} Suppose that $\innp{u_{1},u_{2}} >0$ and that $\Pro_{H_{1}}x \notin W_{2}$. Then $\Pro_{W_{2}}\Pro_{H_{1}}x=\Pro_{H_{2}}\Pro_{H_{1}}x \in \inte W_{1} \cap H_{2} $.
	\end{enumerate}
	
\end{lemma}

\begin{proof}
	\cref{cor:PH2PH1:H2H1}:  Apply \cref{fact:Projec:Hyperplane} two times to obtain that
	\begin{align*}
	\Pro_{H_{2}}\Pro_{H_{1}}x & =\Pro_{H_{1}}x + \frac{\eta_{2} - \innp{\Pro_{H_{1}}x,u_{2}}}{\norm{u_{2}}^{2}}u_{2} \\
	& = x +\frac{\eta_{1} -\innp{x,u_{1}}}{\norm{u_{1}}^{2}} u_{1} + \frac{1}{\norm{u_{2}}^{2}} \left( \eta_{2}  -\Innp{x +\frac{\eta_{1} -\innp{x,u_{1}}}{\norm{u_{1}}^{2}} u_{1}, u_{2}}   \right) u_{2}\\
	&= x +\frac{\eta_{1} -\innp{x,u_{1}}}{\norm{u_{1}}^{2}} u_{1} +\frac{\eta_{2} -\innp{x,u_{2}}}{\norm{u_{2}}^{2}} u_{2} -\frac{ (\eta_{1} -\innp{x,u_{1}}) \innp{u_{1},u_{2}}}{ \norm{u_{1}}^{2}  \norm{u_{2}}^{2}  } u_{2}  \\
	&=x +\frac{\eta_{1} -\innp{x,u_{1}}}{\norm{u_{1}}^{2}} u_{1} + \frac{1}{ \norm{u_{1}}^{2}  \norm{u_{2}}^{2}  } \left( (\eta_{2} -\innp{x,u_{2}}) \norm{u_{1}}^{2} - (\eta_{1} -\innp{x,u_{1}}) \innp{u_{1},u_{2}}  \right)u_{2}.
	\end{align*}
	
	\cref{cor:PH2PH1:notin:EQ}: These equivalences are clear from \cref{fact:Projec:Hyperplane} and \cref{eq:W1W2}. 
	
	\cref{cor:PH2PH1:H2H1:=0}: Let $x \in \mathcal{H}$. Using \cref{cor:PH2PH1:H2H1} and  $ \innp{u_{1},u_{2}}=0$, we get   $\Pro_{H_{2}}\Pro_{H_{1}}x  = x +\frac{\eta_{1} -\innp{x,u_{1}}}{\norm{u_{1}}^{2}} u_{1} +\frac{\eta_{2} -\innp{x,u_{2}}}{\norm{u_{2}}^{2}} u_{2} $,  
 $ \innp{\Pro_{H_{2}} \Pro_{H_{1}} x, u_{1}} =\eta_{1}$ and $ \innp{\Pro_{H_{2}} \Pro_{H_{1}} x, u_{2}} =\eta_{2}$, which, by \cref{eq:H1H2}, imply that $\Pro_{H_{2}} \Pro_{H_{1}} x \in H_{1} \cap H_{2}$.  So,   \cref{fact:ExchangeProj} yields $ \Pro_{H_{2}} \Pro_{H_{1}} x=\Pro_{ H_{1} \cap H_{2}}\Pro_{H_{2}} \Pro_{H_{1}} x=\Pro_{ H_{1} \cap H_{2}}x$.
	
	\cref{cor:PH2PH1:H2}:  Recall that  $u_{1}$ and $u_{2}$ are linearly dependent and $u_{1} \neq 0$. Then  \cref{lem:basic:u1u2:LinerDep}\cref{item:lem:basic:u1u2:LinerDep:LD:u1neq0} leads to 
	\begin{align*}
	\innp{u_{1},u_{2}} u_{2} =\pm \norm{u_{1}}\norm{u_{2}} \left( \pm \frac{\norm{u_{2}}}{\norm{u_{1}}} u_{1} \right) =\norm{u_{2}}^{2} u_{1},
	\end{align*}
	which implies that
	\begin{align}   \label{eq:cor:PH2PH1:u2:u1}
	(\forall x\in \mathcal{H})	\quad \frac{ (\eta_{1} -\innp{x,u_{1}}) \innp{u_{1},u_{2}}}{ \norm{u_{1}}^{2} \norm{u_{2}}^{2}} u_{2} = \frac{\eta_{1} -\innp{x,u_{1}}}{\norm{u_{1}}^{2}} u_{1}.
	\end{align}
	Combining \cref{eq:cor:PH2PH1:u2:u1} with the first identity in \cref{cor:PH2PH1:H2H1} and using \cref{fact:Projec:Hyperplane},  we obtain \cref{cor:PH2PH1:H2}.
	
	\cref{cor:PH2PH1:inW1}: It is easy to see $\Pro_{W_{2}}\Pro_{H_{1}}x=\Pro_{H_{2}}\Pro_{H_{1}}x$ from $\Pro_{H_{1}}x \notin W_{2}$ and \cref{remark:FactWFactH}.  
Moreover, using  $\innp{u_{1},u_{2}} >0$ and $\Pro_{H_{1}}x \notin W_{2}$, we obtain that
\begin{align*}
\eta_{1} -\innp{\Pro_{H_{2}}\Pro_{H_{1}}x ,u_{1}} \stackrel{\cref{cor:PH2PH1:H2H1}}{=}-\frac{1}{ \norm{u_{1}}^{2}  \norm{u_{2}}^{2}  } \left( (\eta_{2} -\innp{x,u_{2}}) \norm{u_{1}}^{2} - (\eta_{1} -\innp{x,u_{1}}) \innp{u_{1},u_{2}}  \right)\innp{u_{2},u_{1}} \stackrel{   \cref{cor:PH2PH1:notin:EQ}}{>}0,
\end{align*}
which, by \cref{eq:W1W2}, implies that $\Pro_{H_{2}}\Pro_{H_{1}}x  \in \inte   W_{1}$. Consequently,  $\Pro_{H_{2}}\Pro_{H_{1}} x \in \inte  W_{1} \cap H_{2}$. 
\end{proof}

\begin{lemma} \label{lemma:x:H1:W2:Px}
		Suppose that  $u_{1} \neq 0$ and $u_{2} \neq 0$ and that $\innp{u_{1},u_{2}} \neq 0$. Let $x \in H_{1} \cap W^{c}_{2}$.  Then the following hold:
		\begin{enumerate}
			\item \label{lemma:x:H1:W2:Px:H1} $\Pro_{W_{2}}x=\Pro_{H_{2}}x \notin H_{1}$
			\item  \label{lemma:x:H1:W2:Px:W2} $\Pro_{H_{1}}\Pro_{W_{2}}x=\Pro_{H_{1}}\Pro_{H_{2}}x \notin W_{2}$.
			\item \label{lemma:x:H1:W2:Px:k} $(\forall k \in \mathbb{N} \smallsetminus \{0\})$	$(\Pro_{W_{2}}\Pro_{H_{1}})^{k}x =	(\Pro_{H_{2}}\Pro_{H_{1}})^{k}x \notin H_{1}$ and $ \Pro_{H_{1}}(\Pro_{W_{2}}\Pro_{H_{1}})^{k}x =\Pro_{H_{1}}(\Pro_{H_{2}}\Pro_{H_{1}})^{k}x\notin W_{2}$.
		\end{enumerate}
\end{lemma}

\begin{proof}
	\cref{lemma:x:H1:W2:Px:H1}: As a consequence of $x \in   W^{c}_{2}$ and \cref{remark:FactWFactH}, we have $\Pro_{W_{2}}x=\Pro_{H_{2}}x$. Combine the assumptions with \cref{fact:Projec:Hyperplane} and \cref{eq:H1H2} to see that
	\begin{align*}
	\eta_{1} - \innp{ \Pro_{H_{2}}x , u_{1}} = \eta_{1}  -\Innp{x +\frac{\eta_{2} -\innp{x,u_{2}}}{\norm{u_{2}}^{2}} u_{2}, u_{1}} 
	=-\frac{\eta_{2} -\innp{x,u_{2}}}{\norm{u_{2}}^{2}} \innp{u_{2}, u_{1}} \neq 0,
	\end{align*}
	which, by \cref{eq:H1H2}, implies that $\Pro_{H_{2}}x \notin H_{1}$. 
	
	 \cref{lemma:x:H1:W2:Px:W2}: Clearly, $\Pro_{H_{1}}\Pro_{W_{2}}x=\Pro_{H_{1}}\Pro_{H_{2}}x$ is from \cref{lemma:x:H1:W2:Px:H1}.  Apply \cref{cor:PH2PH1}\cref{cor:PH2PH1:H2H1} with swapping $H_{1}$ and $H_{2}$ to obtain that 
	 \begin{align*}
	 \eta_{2} - \innp{ \Pro_{H_{1}}\Pro_{H_{2}}x ,u_{2} } & = -\frac{1}{ \norm{u_{1}}^{2}  \norm{u_{2}}^{2} } \left( (\eta_{1} -\innp{x,u_{1}}) \norm{u_{2}}^{2} - (\eta_{2} -\innp{x,u_{2}}) \innp{u_{1},u_{2}}  \right) \innp{u_{1}, u_{2} }\\
	  & = \frac{1}{ \norm{u_{1}}^{2}  \norm{u_{2}}^{2} }  (\eta_{2} -\innp{x,u_{2}}) \innp{u_{1},u_{2}}^{2} <0,  \quad (\text{by } x  \in H_{1} \cap W^{c}_{2} \text{ and }\innp{u_{1},u_{2}} \neq 0)
	 \end{align*}
	 which, by \cref{eq:W1W2}, implies that $\Pro_{H_{1}}\Pro_{H_{2}}x \notin W_{2}$.
	 
	 \cref{lemma:x:H1:W2:Px:k}:	We proceed the proof by induction on $k$.  For $k=1$, bearing in mind that $x \in H_{1} \cap W^{c}_{2}$,   \cref{remark:FactWFactH} and \cref{lemma:x:H1:W2:Px:H1}, we have that $\Pro_{W_{2}}\Pro_{H_{1}} x =	\Pro_{H_{2}} x \notin H_{1}$. Combine this with \cref{lemma:x:H1:W2:Px:W2} to see that  $ \Pro_{H_{1}}\Pro_{W_{2}}\Pro_{H_{1}} x= \Pro_{H_{1}}\Pro_{H_{2}}x \notin W_{2}$. Hence, the base case is true. Suppose that for some $ k \in \mathbb{N} \smallsetminus \{0\}$,	
	 \begin{align} \label{eq:lemma:x:H1:W2:Px:k:IH}
 (\Pro_{W_{2}}\Pro_{H_{1}})^{k}x =	(\Pro_{H_{2}}\Pro_{H_{1}})^{k}x \notin H_{1} \quad \text{and} \quad   \Pro_{H_{1}}(\Pro_{W_{2}}\Pro_{H_{1}})^{k}x \notin W_{2}.
	 \end{align}
	  Denote by $y:=  \Pro_{H_{1}}(\Pro_{W_{2}}\Pro_{H_{1}})^{k}x$. Then $ y \in  H_{1} \cap W^{c}_{2}$ follows from the induction hypothesis and \cref{eq:lemma:x:H1:W2:Px:k:IH}. Moreover, 
	  \begin{subequations}
	  	\begin{align}
	  	& (\Pro_{W_{2}}\Pro_{H_{1}})^{k+1}x =\Pro_{W_{2}}\Pro_{H_{1}} (\Pro_{W_{2}}\Pro_{H_{1}})^{k}x 
	  	=\Pro_{W_{2}}y \stackrel{\text{\cref{lemma:x:H1:W2:Px:H1} }}{=}\Pro_{H_{2}}y =(\Pro_{H_{2}}\Pro_{H_{1}})^{k+1}x,  \quad \Pro_{W_{2}}y 
	  	\stackrel{\text{\cref{lemma:x:H1:W2:Px:H1} }}{\notin} H_{1}, \label{lemma:x:H1:W2:Px:k:H1}\\
	  	&\Pro_{H_{1}}(\Pro_{W_{2}}\Pro_{H_{1}})^{k+1}x \stackrel{\cref{lemma:x:H1:W2:Px:k:H1}}{=} \Pro_{H_{1}}\Pro_{H_{2}}y \stackrel{\text{\cref{lemma:x:H1:W2:Px:W2}}}{\notin} W_{2}. \label{lemma:x:H1:W2:Px:k:W2}
	  	\end{align}
	  \end{subequations}
Hence,  \cref{lemma:x:H1:W2:Px:k:H1} and \cref{lemma:x:H1:W2:Px:k:W2} yield that \cref{eq:lemma:x:H1:W2:Px:k:IH} holds for $k+1$. Therefore, \cref{lemma:x:H1:W2:Px:k} holds by induction.  
\end{proof}

\section{Projection onto intersection of hyperplanes} \label{sec:hyperplane}
In this section, we consider  the projection onto intersection of hyperplanes and  the composition of projections onto hyperplanes.

Recall that $\I:=\{1,2,\ldots, m\}$ and that for every $i \in \I$,  $u_{i}$ is in $\mathcal{H}$, and $\eta_{i}$ is in $\mathbb{R}$. Moreover, 
\begin{empheq}[box=\mybluebox]{equation*}
(\forall i \in \I) \quad H_{i} := \{x \in \mathcal{H} ~:~ \innp{x,u_{i}}= \eta_{i} \}.
\end{empheq}

\begin{remark} \label{rem:J:I}
	Suppose that $m\geq 2$ and that $u_{1}, \ldots, u_{m}$ are linearly dependent. Then without loss of generality, assume that $u_{1}, \ldots, u_{t}$ with $t \in \I \smallsetminus \{m\}$ are linearly independent and that $(\forall i \in  \{t+1, \ldots, m\})$ $u_{1}, \ldots, u_{t},u_{i}$ are linearly dependent.  
	Let $  i \in  \{t+1, \ldots, m\}$. Assume that  $u_{i} =\sum^{t}_{j=1} \gamma_{j} u_{j}$ for some $(\gamma_{1}, \ldots, \gamma_{t})^{\intercal} \in \mathbb{R}^{t} $. Then if $\eta_{i} =\sum^{t}_{j=1} \gamma_{j} \eta_{j}$, then $\cap^{t}_{j=1} H_{j} = (\cap^{t}_{j=1} H_{j}) \cap H_{i}$. Otherwise,  $ (\cap^{t}_{j=1} H_{j}) \cap H_{i} =\varnothing$. 
	Set $\J$ as the maximally subset of $\I$ such that  $u_{i}$, for all $i \in \J$  are linearly independent. Let $x\in \mathcal{H}$.
	Therefore, if only $\cap_{i \in \I} H_{i} \neq \varnothing$, to deduce  $\Pro_{\cap_{i \in \I} H_{i}} x$, we only need to first find $\J$, then  $\Pro_{\cap_{i \in \I} H_{i}} x=\Pro_{\cap_{i \in \J} H_{i}} x$. Therefore, in the following \cref{prop:proj:inters:hyperplanes}, we care only the case in which $u_{1}, \ldots, u_{m}$ are linearly independent.
\end{remark}

\begin{proposition} \label{prop:proj:inters:hyperplanes}
	Suppose that $u_{1}, \ldots, u_{m}$ are linearly independent. Denote by 
	\begin{align*}
	(\beta_{1}, \ldots, \beta_{m})^{\intercal} := G(u_{1}, \ldots, u_{m})^{-1} \left(\innp{u_{1}, x}-\eta_{1}, \ldots, \innp{u_{m},x}-\eta_{m} \right)^{\intercal},
	\end{align*} 
	where $G(u_{1}, \ldots, u_{m})$ is the Gram matrix defined in \cref{fact:Gram:inver}. 
	Then 
	\begin{align*}
	(\forall x \in \mathcal{H}) \quad \Pro_{\cap_{i \in \I} H_{i}}x= x-\sum_{ i \in \I} \beta_{i}u_{i}.
	\end{align*}
\end{proposition}

\begin{proof}
	Define $L:\mathcal{H} \to \mathbb{R}^{m}$ by 
	\begin{align*}
	(\forall x \in \mathcal{H}) \quad Lx:=\left(\innp{u_{1}, x}, \ldots, \innp{u_{m},x} \right)^{\intercal}.
	\end{align*}
	Clearly, $L\in \mathcal{B} ( \mathcal{H}, \mathbb{R}^{m})$. Then according to the definition of adjoint operator, $L^{*}: \mathbb{R}^{m} \to \mathcal{H}$ is defined by 
	\begin{align*}
	\left(\forall \alpha:=(\alpha_{1}, \ldots, \alpha_{m} )^{\intercal} \in \mathbb{R}^{m} \right) \quad L^{*}\alpha=\sum_{ i \in \I} \alpha_{i} u_{i}.
	\end{align*}
	It is easy to see that $LL^{*} : \mathbb{R}^{m} \to \mathbb{R}^{m}$ satisfies that 
	\begin{align*}
	(\forall \alpha \in \mathbb{R}^{m} ) \quad LL^{*}\alpha =\left(\Innp{u_{1}, \sum_{i \in \I} \alpha_{i}u_{i}}, \ldots, \Innp{u_{m},\sum_{i \in \I} \alpha_{i}u_{i} }  \right)^{\intercal}=G(u_{1}, \ldots, u_{m})\alpha. 
	\end{align*}
	Combine this with  \cref{fact:Gram:inver} and the linear independence of $u_{1}, \ldots, u_{m}$ to see  that $ LL^{*} $  is invertible, which, connecting with  \cite[Example~29.17(iii)]{BC2017}, implies that 
	\begin{align*}
	(\forall x \in \mathcal{H}) \quad \Pro_{\cap_{i \in \I} H_{i}}x=x-L^{*}(LL^{*})^{-1}(Lx-\eta) =x-\sum_{ i \in \I} \beta_{i}u_{i},
	\end{align*}
	where $\eta:=(\eta_{1}, \ldots, \eta_{m} )^{\intercal} $.
\end{proof}

\begin{remark}
	With \cref{rem:J:I} and \cref{prop:proj:inters:hyperplanes}, we are able to solve the least norm problem presented in \cite[Section~5.2]{BCS2019} without the requirement that the related matrix is full rank.
\end{remark}

\begin{lemma} \label{lem:H1H2Ineq}
	Suppose that $u_{1}$ and $u_{2}$ are linearly independent. Denote by $\gamma :=\frac{|\innp{u_{1}, u_{2}}| }{\norm{u_{1}} \norm{u_{2}}}$. Then $\gamma \in \left[0,1\right[\,$ and 
	\begin{align*}
	(\forall x \in \mathcal{H}) (\forall k \in \mathbb{N}) \quad \norm{ (\Pro_{H_{2}}  \Pro_{H_{1}})^{k} x-  \Pro_{H_{1} \cap H_{2}}x } \leq \gamma^{k} \norm{x -\Pro_{H_{1} \cap H_{2}}x }.
	\end{align*}
\end{lemma}

\begin{proof}
	It is easy to see the desired results from \cite[Example~9.40]{D2012}.
\end{proof}

\section{Compositions of projections onto halfspaces} \label{sec:Halfspaces}

In this section, we study the projection onto intersection of halfspaces and the composition of projections onto halfspaces.  

Recall that $\I:=\{1,2,\ldots, m\}$, that for every $i \in \I$, $u_{i}$ is in $\mathcal{H}$, and  $\eta_{i}$ is in $\mathbb{R}$, and that
\begin{empheq}[box=\mybluebox]{equation} \label{eq:W12}
(\forall i \in \I) \quad W_{i} := \{x \in \mathcal{H} ~:~ \innp{x,u_{i}} \leq  \eta_{i} \}, \quad \text{and} \quad H_{i} := \{x \in \mathcal{H} ~:~ \innp{x,u_{i}}= \eta_{i} \}.
\end{empheq}

For every $k \in \mathbb{N}$, let $[k]$ denote \enquote{$k \mod m$}. Let $x \in \mathcal{H}$. The sequence of iterations of Dykstra's algorithm is: $x_{0}:=x, e_{-(m-1)}=\cdots=e_{-1}=e_{0}=0$, 
\begin{align*}
(\forall k \in\mathbb{N} \smallsetminus \{0\}) \quad x_{k}:=\Pro_{W_{[k]}}(x_{k-1} +e_{k-m}), \text{ and } e_{k}:=x_{k-1} +e_{k-m}-x_{k}.
\end{align*}
\begin{fact} {\rm \cite[Example~9.41]{D2012}} \label{eq:fact:Dykstra:halfspaces}
	Suppose that $(\forall i \in \I)$ $u_{i} \neq 0$ and $\cap_{i \in \I} W_{i} \neq \varnothing$. Let $x \in \mathcal{H}$. Then the  sequence $(x_{k})_{k\in \mathbb{N}}$ according to the Dykstra's algorithm with $x_{0}=x$ converges to $ \Pro_{\cap_{i \in \I} W_{i}} x$.
\end{fact}

Unfortunately, as we mentioned before, the practical manipulation of the Dykstra's algorithm is not easy. In the remaining of this section, we consider the case in which $m=2$ and systematically investigate the relation between $\Pro_{W_{1} \cap W_{2}}$ and $\Pro_{W_{2}}\Pro_{W_{1} }$. 

\subsection*{$u_{1}$ and $u_{2}$ are linearly dependent }
In the whole subsection, we assume that 
\begin{empheq}[box=\mybluebox]{equation*}
u_{1} \text{ and } u_{2} \text{ are linearly dependent}.
\end{empheq}  
\begin{fact} {\rm \cite[Proposition~29.22]{BC2017} } \label{fact:W1capW2:u1u2LD}
Exactly one of the following cases occurs:
	\begin{enumerate}
		\item \label{item:fact:W1capW2:u1u2LD:i} $u_{1}=u_{2}=0$ and $0 \leq \min\{\eta_{1}, \eta_{2} \}$. Then $W_{1} \cap W_{2} =\mathcal{H}$ and $\Pro_{W_{1} \cap W_{2}} =\Id$.
		\item \label{item:fact:W1capW2:u1u2LD:ii} $u_{1}=u_{2}=0$ and $\min\{\eta_{1}, \eta_{2} \} <0$. Then $W_{1} \cap W_{2} =\varnothing$.
		\item \label{item:fact:W1capW2:u1u2LD:iii} $u_{1} \neq 0$, $u_{2}=0$, and $0 \leq \eta_{2}$. Then $W_{1} \cap W_{2} = W_{1} $ and $\Pro_{W_{1} \cap W_{2}}  =  \Pro_{W_{1}} $.
		\item \label{item:fact:W1capW2:u1u2LD:iv} $u_{1} \neq 0$, $u_{2}=0$, and  $\eta_{2} <0$.  Then $W_{1} \cap W_{2} = \varnothing$.
		\item \label{item:fact:W1capW2:u1u2LD:v} $u_{1} = 0$, $u_{2} \neq0$, and  $0 \leq \eta_{1}$.
		Then $W_{1} \cap W_{2} = W_{2}  $ and $\Pro_{W_{1} \cap W_{2}} =\Pro_{ W_{2}} $.
		\item \label{item:fact:W1capW2:u1u2LD:vi} $u_{1} = 0$, $u_{2} \neq 0$, and  and  $\eta_{1} <0$.  Then $W_{1} \cap W_{2} = \varnothing$.
		\item \label{item:fact:W1capW2:u1u2LD:vii} $u_{1} \neq 0$, $u_{2} \neq0$, and  $\innp{u_{1},u_{2}} >0$. Then $W_{1} \cap W_{2} = \{x \in \mathcal{H} ~:~ \innp{x, u }\leq \eta \} $ where $u=\norm{u_{2}}u_{1}$ and $\eta = \min \{ \eta_{1}\norm{u_{2}}, \eta_{2} \norm{u_{1}} \}$, and 
		\begin{align*}
		(\forall x \in \mathcal{H}) \quad \Pro_{W_{1} \cap W_{2}} x= \begin{cases}
		x, \quad  &\text{if } \innp{x,u} \leq \eta;\\
		x+ \frac{\eta - \innp{x, u }}{\norm{u }^{2}} u, \quad &\text{if } \innp{x,u} > \eta.
		\end{cases}
		\end{align*}
		\item \label{item:fact:W1capW2:u1u2LD:viii} $u_{1} \neq 0$, $u_{2} \neq 0$,   $\innp{u_{1},u_{2}} <0$, and  $\eta_{1} \norm{u_{2}} +\eta_{2} \norm{u_{1} }<0$.    Then $W_{1} \cap W_{2} = \varnothing$.
		\item \label{item:fact:W1capW2:u1u2LD:ix} $u_{1} \neq 0$, $u_{2} \neq 0$,   $\innp{u_{1},u_{2}} <0$, and $\eta_{1} \norm{u_{2}} +\eta_{2} \norm{u_{1} } \geq 0$.    Then 
		$W_{1} \cap W_{2} = \{x \in \mathcal{H} ~:~ \gamma_{1} \leq \innp{x, u }\leq \gamma_{2} \} \neq \varnothing $ where $u=\norm{u_{2}}u_{1}$,  $\gamma_{1}=-\eta_{2} \norm{u_{1}}$,   and $ \gamma_{2} =\eta_{1} \norm{u_{2}}$, and
		\begin{align*}
		(\forall x \in \mathcal{H}) \quad \Pro_{W_{1} \cap W_{2}} x= \begin{cases}
		x -  \frac{\innp{x, u } -\gamma_{1}}{\norm{u }^{2}} u, \quad &\text{if } \innp{x,u} < \gamma_{1};\\
		x, \quad  &\text{if }  \gamma_{1} \leq \innp{x,u} \leq \gamma_{2};\\
		x-  \frac{ \innp{x, u } -\gamma_{2}}{\norm{u }^{2}} u, \quad &\text{if } \innp{x,u} > \gamma_{2}.
		\end{cases}
		\end{align*}
	\end{enumerate}
\end{fact}

\begin{lemma} \label{lemma:u1u2LD:Cases}
	Suppose that  $u_{1} \neq 0$ and $u_{2} \neq 0$. 
	\begin{enumerate}
		\item \label{lemma:u1u2LD:Cases:>} Suppose that $\innp{u_{1},u_{2}} >0$.
		\begin{enumerate}
			\item \label{lemma:u1u2LD:Cases:>:leq} If $\eta_{1} \norm{u_{2}} \leq \eta_{2} \norm{u_{1}} $, then $W_{1} \cap W_{2}=W_{1}$ and $H_{1} \subseteq W_{2}$.
			\item \label{lemma:u1u2LD:Cases:>:>} If $\eta_{1} \norm{u_{2}} > \eta_{2} \norm{u_{1}} $, then $W_{1} \cap W_{2}=W_{2}$ and $H_{1} \subseteq W^{c}_{2}$.
		\end{enumerate}
		\item \label{lemma:u1u2LD:Cases:<} Suppose that $\innp{u_{1},u_{2}} <0$ and that $\eta_{1} \norm{u_{2}} +\eta_{2} \norm{u_{1} } \geq 0$.    Then $H_{1} \subseteq W_{2}$ and $H_{2} \subseteq W_{1}$.
	\end{enumerate}
\end{lemma}

\begin{proof}
	\cref{lemma:u1u2LD:Cases:>}:  Set $u:=\norm{u_{2}}u_{1}$.
	Then the assumptions and  \cref{lem:basic:u1u2:LinerDep}\cref{item:lem:basic:u1u2:LinerDep:LD:u1neq0} imply that
	\begin{align}  \label{eq:lemma:u1u2LD:Cases:u}
	u=\norm{u_{2}}u_{1} = \norm{u_{1}}u_{2}.  
	\end{align}
	As a consequence of \cref{eq:W12},  we see that for every $y \in \mathcal{H}$,
	\begin{subequations} \label{eq:lemma:W1capW2:u1u2LD}
		\begin{align} 
		& y \in W_{1} \Leftrightarrow  \innp{y, u_{1}} \leq \eta_{1} \Leftrightarrow \norm{u_{2}}  \innp{y, u_{1}} \leq \norm{u_{2}} \eta_{1} \stackrel{\cref{eq:lemma:u1u2LD:Cases:u}}{\Leftrightarrow} \innp{y,u} \leq \eta_{1} \norm{u_{2}}, \label{eq:lemma:u1u2LD:Cases:W1} \\
		& y \in W_{2} \Leftrightarrow  \innp{y, u_{2}} \leq \eta_{2} \Leftrightarrow \norm{u_{1}} \innp{y, u_{2}} \leq \norm{u_{1}}  \eta_{2} \stackrel{\cref{eq:lemma:u1u2LD:Cases:u}}{\Leftrightarrow} \innp{y,u} \leq \eta_{2} \norm{u_{1}}, \label{eq:lemma:u1u2LD:Cases:W2}
		\end{align} 
	\end{subequations}

	\cref{lemma:u1u2LD:Cases:>:leq}:  Suppose that $\eta_{1} \norm{u_{2}} \leq \eta_{2} \norm{u_{1}} $.  Then the required results follow from  \cref{eq:W12} and \cref{eq:lemma:W1capW2:u1u2LD}.

	\cref{lemma:u1u2LD:Cases:>:>}: Suppose   that $\eta_{1} \norm{u_{2}} > \eta_{2} \norm{u_{1}} $. Then it is easy to see $W_{2} \subseteq W_{1}$ from \cref{eq:lemma:W1capW2:u1u2LD}.  Let $x \in \mathcal{H}$. Now,
	\begin{align*}
	x \in H_{1} \Leftrightarrow \innp{x, u_{1}} = \eta_{1} \stackrel{\cref{eq:lemma:u1u2LD:Cases:u}}{\Leftrightarrow} \innp{x,u} = \eta_{1} \norm{u_{2}} >  \eta_{2} \norm{u_{1}}  \stackrel{\cref{eq:lemma:u1u2LD:Cases:W2}}{\Rightarrow} x \in W^{c}_{2}.
	\end{align*}
	Hence, $H_{1}  \subseteq W^{c}_{2}$.
	
	\cref{lemma:u1u2LD:Cases:<}: Suppose that $\innp{u_{1},u_{2}} <0$ and that $\eta_{1} \norm{u_{2}} +\eta_{2} \norm{u_{1} } \geq 0$. 
	By assumptions and  \cref{lem:basic:u1u2:LinerDep}\cref{item:lem:basic:u1u2:LinerDep:LD:u1neq0}, we know that 
	\begin{align}  \label{eq:lemma:u1u2LD:Cases:u:-}
	\norm{u_{2}}u_{1} = -\norm{u_{1}}u_{2}.  
	\end{align}
	
Clearly $\eta_{1} \norm{u_{2}} +\eta_{2} \norm{u_{1} } \geq 0$ implies 
	\begin{align} \label{eq:PW1capW2:LD:Case5.3:assump}
	-\frac{\norm{u_{2}}}{ \norm{u_{1}}}\eta_{1}  \leq \eta_{2} \quad \text{and} \quad -\frac{\norm{u_{1}}}{ \norm{u_{2}}}\eta_{2}  \leq \eta_{1}.
	\end{align}
	Let $x \in \mathcal{H}$. Now,
	\begin{subequations}
		\begin{align*}
		x\in H_{1} \Leftrightarrow	\innp{x, u_{1}} =\eta_{1}  \Leftrightarrow  -\frac{\norm{u_{2}}}{ \norm{u_{1}}}\innp{x, u_{1}} = -\frac{\norm{u_{2}}}{ \norm{u_{1}}}\eta_{1}  \stackrel{\cref{eq:lemma:u1u2LD:Cases:u:-}}{\Leftrightarrow} \innp{x, u_{2}} = -\frac{\norm{u_{2}}}{ \norm{u_{1}}}\eta_{1}  \stackrel{\cref{eq:PW1capW2:LD:Case5.3:assump}}{\leq} \eta_{2} \Leftrightarrow x \in W_{2},\\
		x\in H_{2} \Leftrightarrow	\innp{x, u_{2}} =\eta_{2}  \Leftrightarrow  -\frac{\norm{u_{1}}}{ \norm{u_{2}}}\innp{x, u_{2}} = -\frac{\norm{u_{1}}}{ \norm{u_{2}}}\eta_{2}  \stackrel{\cref{eq:lemma:u1u2LD:Cases:u:-}}{\Leftrightarrow} \innp{x, u_{1}} = -\frac{\norm{u_{1}}}{ \norm{u_{2}}}\eta_{2}  \stackrel{\cref{eq:PW1capW2:LD:Case5.3:assump}}{\leq}  \eta_{1} \Leftrightarrow x \in W_{1},
		\end{align*}
	\end{subequations}
	which imply that $H_{1} \subseteq W_{2}$ and $H_{2} \subseteq W_{1}$. 
\end{proof}

\begin{lemma}  \label{lemma:W1capW2:u1u2LD:vii}
	Suppose that  $u_{1} \neq 0$, $u_{2} \neq0$, and  $\innp{u_{1},u_{2}} >0$. Let $x \in \mathcal{H}$.  Then exactly one of the following cases occurs:
	\begin{enumerate}
		\item \label{item:lemma:W1capW2:u1u2LD:vii:i} $x \in W_{1} \cap W_{2}$. Then $\Pro_{W_{1} \cap W_{2}} x=x$.
		\item \label{item:lemma:W1capW2:u1u2LD:vii:ii} $x \in W_{1} \cap W^{c}_{2}$.  Then $\Pro_{W_{1} \cap W_{2}} x= \Pro_{H_{2}}x$.
		\item \label{item:lemma:W1capW2:u1u2LD:vii:iii} $x \in W^{c}_{1} \cap W_{2}$.   Then $\Pro_{W_{1} \cap W_{2}} x= \Pro_{H_{1}}x$.
		\item  \label{item:lemma:W1capW2:u1u2LD:vii:iv} $x \in W^{c}_{1} \cap W^{c}_{2}$.   Then 
		\begin{align*}
		\Pro_{W_{1} \cap W_{2}} x= 
		\begin{cases}
		\Pro_{H_{1}}x, \quad & \text{if } \eta_{1} \norm{u_{2}} \leq \eta_{2} \norm{u_{1}};\\
		\Pro_{H_{2}}x, \quad & \text{if } \eta_{1} \norm{u_{2}} > \eta_{2} \norm{u_{1}}.
		\end{cases}
		\end{align*}
	\end{enumerate}
\end{lemma}

\begin{proof}
		\cref{item:lemma:W1capW2:u1u2LD:vii:i}: It is trivial by the definition of projection.
	
	\cref{item:lemma:W1capW2:u1u2LD:vii:ii}: The assumption shows that $W_{1} \cap W^{c}_{2} \neq \varnothing$, which yields $W_{1} \not \subseteq W_{2}$. Hence, by \cref{lemma:u1u2LD:Cases}\cref{lemma:u1u2LD:Cases:>:leq},   we have that  $\eta_{1} \norm{u_{2}} > \eta_{2} \norm{u_{1}} $ and  $W_{1} \cap W_{2}=W_{2}$.  Therefore, by \cref{remark:FactWFactH}, $\Pro_{W_{1} \cap W_{2}} x= \Pro_{W_{2}}x= \Pro_{H_{2}}x$.

	\cref{item:lemma:W1capW2:u1u2LD:vii:iii}:  Switch $W_{1}$ and $W_{2}$ in \cref{item:lemma:W1capW2:u1u2LD:vii:ii} to obtain	\cref{item:lemma:W1capW2:u1u2LD:vii:iii}.

	\cref{item:lemma:W1capW2:u1u2LD:vii:iv}:  The desired result follows from the assumptions,  \cref{lemma:u1u2LD:Cases}\cref{lemma:u1u2LD:Cases:>}  and  \cref{fact:Projec:Hyperplane}.
\end{proof}

\begin{lemma}  \label{lemma:W1capW2:u1u2LD:ix}
	Suppose that  $u_{1} \neq 0$, $u_{2} \neq 0$,   $\innp{u_{1},u_{2}} <0$, and $\eta_{1} \norm{u_{2}} +\eta_{2} \norm{u_{1} } \geq 0$. Let $x \in \mathcal{H}$.   Then the following statements hold:
	\begin{enumerate}
		\item \label{lemma:W1capW2:u1u2LD:ix:i} If $x \in W_{1} \cap W_{2}$, then $\Pro_{W_{1} \cap W_{2}} x=x$.
		\item \label{lemma:W1capW2:u1u2LD:ix:ii}  If $x \in W_{1} \cap W^{c}_{2}$, then $\Pro_{W_{1} \cap W_{2}} x= \Pro_{H_{2}}x$.
		\item \label{lemma:W1capW2:u1u2LD:ix:iii}  If $x \in W^{c}_{1} \cap W_{2}$, then $\Pro_{W_{1} \cap W_{2}} x= \Pro_{H_{1}}x$.
		\item \label{lemma:W1capW2:u1u2LD:ix:iv} $W^{c}_{1} \cap W^{c}_{2} =\varnothing$. 
	\end{enumerate}
\end{lemma}

\begin{proof} 
	Set \begin{align} \label{eq:lemma:W1capW2:u1u2LD:ix:Notation}
	u:=\norm{u_{2}}u_{1}, \quad  \gamma_{1}:=-\eta_{2} \norm{u_{1}}, \quad \text{and} \quad \gamma_{2} :=\eta_{1} \norm{u_{2}}.
	\end{align}
		From  assumptions and \cref{lem:basic:u1u2:LinerDep}\cref{item:lem:basic:u1u2:LinerDep:LD:u1neq0}, we know that 
	\begin{align}  \label{eq:lemma:W1capW2:u1u2LD:ix:u}
	u=\norm{u_{2}}u_{1} = -\norm{u_{1}}u_{2}.  
	\end{align}
	As a consequence of  \cref{fact:W1capW2:u1u2LD}\cref{item:fact:W1capW2:u1u2LD:ix}, we see that $W_{1} \cap W_{2} = \{x \in \mathcal{H} ~:~ \gamma_{1} \leq \innp{x, u }\leq \gamma_{2} \} \neq \varnothing $
	and 
	\begin{align} \label{eq:lemma:W1capW2:u1u2LD:ix:PW1capW2}
	(\forall x \in \mathcal{H}) \quad \Pro_{W_{1} \cap W_{2}} x= \begin{cases}
	x -  \frac{\innp{x, u } -\gamma_{1}}{\norm{u }^{2}} u, \quad &\text{if } \innp{x,u} < \gamma_{1};\\
	x, \quad  &\text{if }  \gamma_{1} \leq \innp{x,u} \leq \gamma_{2};\\
	x-  \frac{ \innp{x, u } -\gamma_{2}}{\norm{u }^{2}} u, \quad &\text{if } \innp{x,u} > \gamma_{2}.
	\end{cases}
	\end{align}
	Hence, according to \cref{eq:W12},  for every $y \in \mathcal{H}$,
	\begin{subequations}\label{eq:lemma:W1capW2:u1u2LD:ix:W1W2}
		\begin{align} 
		&y \in W_{1} \Leftrightarrow  \innp{y, u_{1}} \leq \eta_{1} \Leftrightarrow \norm{u_{2}} \innp{y, u_{1}} \leq \norm{u_{2}} \eta_{1} \stackrel{\cref{eq:lemma:W1capW2:u1u2LD:ix:Notation}}{\Leftrightarrow}  \innp{y,u} \leq \eta_{1} \norm{u_{2}}  \stackrel{\cref{eq:lemma:W1capW2:u1u2LD:ix:Notation}}{\Leftrightarrow} \innp{y,u} \leq \gamma_{2}, \label{eq:lemma:W1capW2:u1u2LD:ix:W1}\\
		&y \in W_{2} \Leftrightarrow  \innp{y, u_{2}} \leq \eta_{2} \Leftrightarrow \norm{u_{1}} \innp{y, u_{2}} \leq \norm{u_{1}}\eta_{2} \Leftrightarrow - \eta_{2} \norm{u_{1}} \leq \innp{y,- \norm{u_{1}}u_{2}} \stackrel{\cref{eq:lemma:W1capW2:u1u2LD:ix:u}}{\Leftrightarrow} \gamma_{1} \leq \innp{y, u}. \label{eq:lemma:W1capW2:u1u2LD:ix:W2}
		\end{align} 
	\end{subequations}

	\cref{lemma:W1capW2:u1u2LD:ix:i}: This is trivial.
	
	\cref{lemma:W1capW2:u1u2LD:ix:ii}: Assume that $x \in W_{1} \cap W^{c}_{2}$. Then  \cref{eq:lemma:W1capW2:u1u2LD:ix:W1W2} leads to $ \innp{x, u} \leq \gamma_{2}$ and $ \innp{x, u} < \gamma_{1}$. Hence, 
using  \cref{eq:lemma:W1capW2:u1u2LD:ix:PW1capW2}, \cref{eq:lemma:W1capW2:u1u2LD:ix:Notation},  and \cref{eq:lemma:W1capW2:u1u2LD:ix:u}, we have that
	\begin{align*}
	\Pro_{W_{1} \cap W_{2}} x  = x-  \frac{\innp{x, u } -\gamma_{1}}{\norm{u }^{2}} u  
	 =  x-  \frac{ \innp{x, -\norm{u_{1}}u_{2} } + \eta_{2} \norm{u_{1}} }{\norm{-\norm{u_{1}}u_{2} }^{2}} (-\norm{u_{1}}u_{2})   
	 = x + \frac{\eta_{2} - \innp{x, u_{2}}}{\norm{u_{2}}^{2}} u_{2} =\Pro_{H_{2}}x.
	\end{align*}
	
	\cref{lemma:W1capW2:u1u2LD:ix:iii}: The proof is similar to the proof of \cref{lemma:W1capW2:u1u2LD:ix:ii}.

	\cref{lemma:W1capW2:u1u2LD:ix:iv}:  Assume to the contrary that there exists $z \in W^{c}_{1} \cap W^{c}_{2}$. Then by 
	\cref{eq:W12}, 
	\begin{align*}
	\innp{z, u_{1}} > \eta_{1} ~\text{and} ~  \innp{z, u_{2}} > \eta_{2} & \stackrel{\cref{eq:lemma:W1capW2:u1u2LD:ix:u}}{\Leftrightarrow } \innp{z, u_{1}} > \eta_{1} ~\text{and} ~  \left< z, -\frac{\norm{u_{2}}}{\norm{u_{1}}}u_{1} \right>  > \eta_{2}  \\
	& ~\Leftrightarrow~  -\frac{\norm{u_{2}}}{\norm{u_{1}}} \eta_{1} >-\frac{\norm{u_{2}}}{\norm{u_{1}}} \innp{z, u_{1}} > \eta_{2} \\
	&  ~\Rightarrow~  \norm{u_{1}} \eta_{2} +\norm{u_{2}} \eta_{1} <0,
	\end{align*}
	which contradicts to the assumption that $\norm{u_{1}} \eta_{2} +\norm{u_{2}} \eta_{1} \geq 0$.
\end{proof}

\begin{theorem} \label{theorem:W1capW2:u1u2LD}
	Suppose that   $W_{1} \cap W_{2} \neq \varnothing$. Then $\Pro_{W_{2}} \Pro_{W_{1}} =\Pro_{W_{1} \cap W_{2}}$.
\end{theorem}

\begin{proof}
Using $W_{1} \cap W_{2} \neq \varnothing$,   
\cref{fact:W1capW2:u1u2LD}, and the linear dependence of $u_{1}$ and $u_{2}$,    we have exactly the following cases. 
	
	\textbf{Case~1:}  $u_{1}=u_{2}=0$ and $0 \leq \min\{\eta_{1}, \eta_{2} \}$.
	Then, due to \cref{fact:W1capW2:u1u2LD}\cref{item:fact:W1capW2:u1u2LD:i},  $W_{1} \cap W_{2} =\mathcal{H}$ and $\Pro_{W_{1} \cap W_{2}} =\Id$. Moreover, $W_{1} =\mathcal{H}$ and $W_{2} =\mathcal{H}$. Hence, $ \Pro_{W_{2} }\Pro_{W_{1} }= \Id =\Pro_{W_{1} \cap W_{2}} $.
	
	\textbf{Case~2:} $u_{1} \neq 0$, $u_{2}=0$, and $0 \leq \eta_{2}$. Then   \cref{fact:W1capW2:u1u2LD}\cref{item:fact:W1capW2:u1u2LD:iii} implies that $W_{1} \cap W_{2} =W_{1}$ and $\Pro_{W_{1} \cap W_{2}}  =\Pro_{W_{1}} $. Moreover, 
	$W_{2} =\mathcal{H}$, and $\Pro_{W_{2}} =\Id$. Hence, $ \Pro_{W_{2}} \Pro_{W_{1}} =\Id \Pro_{W_{1}} =\Pro_{W_{1}} =\Pro_{W_{1} \cap W_{2}} $.

	\textbf{Case~3:}  $u_{1} = 0$, $u_{2} \neq0$, and  $0 \leq \eta_{1}$.
	Then, by \cref{fact:W1capW2:u1u2LD}\cref{item:fact:W1capW2:u1u2LD:v}, $W_{1} \cap W_{2}=W_{2}$, and $\Pro_{W_{1} \cap W_{2}}  =\Pro_{W_{2}} $. Moreover,
	$W_{1} =\mathcal{H}$, and $\Pro_{W_{1}} =\Id$. Hence,
	$ \Pro_{W_{2}} \Pro_{W_{1}} = \Pro_{W_{2}} \Id =\Pro_{W_{2}}  = \Pro_{W_{1} \cap W_{2}} $.
	
	\textbf{Case~4:}  $u_{1} \neq 0$, $u_{2} \neq0$, and  $\innp{u_{1},u_{2}} >0$. Let $x \in \mathcal{H}$. Notice that  $\mathcal{H}= (W_{1} \cap W_{2}) \cup (W_{1} \cap W^{c}_{2}) \cup ( W^{c}_{1} \cap W_{2}) \cup (W^{c}_{1} \cap W^{c}_{2})$. We have exactly the following four subcases. 
	
	Case~4.1: $x \in W_{1} \cap W_{2}$.  Then, clearly, $\Pro_{W_{2}} \Pro_{W_{1}}x =x=\Pro_{W_{1} \cap W_{2}} x$.
	
	Case~4.2:  $x \in W_{1} \cap W^{c}_{2}$. Then $\Pro_{W_{1}}x=x $ and, by \cref{remark:FactWFactH},  $\Pro_{W_{2}}x =\Pro_{H_{2}}x $.  Hence,   \cref{lemma:W1capW2:u1u2LD:vii}\cref{item:lemma:W1capW2:u1u2LD:vii:ii} implies  $\Pro_{W_{2}} \Pro_{W_{1}}x = \Pro_{H_{2}}x =\Pro_{W_{1} \cap W_{2}} x$.
	
	Case~4.3:  $x \in W^{c}_{1} \cap W_{2}$. 
Then $W^{c}_{1} \cap W_{2} \neq \varnothing$, which implies that $W_{1} \cap W_{2}  \neq W_{2}$. By \cref{lemma:u1u2LD:Cases}\cref{lemma:u1u2LD:Cases:>:>}, we have that $\eta_{1} \norm{u_{2}} \leq \eta_{2} \norm{u_{1}} $, $W_{1} \cap W_{2}=W_{1}$ and $H_{1} \subseteq W_{2}$.  Combine these results with $x \in W^{c}_{1}$ and \cref{remark:FactWFactH} to obtain that $\Pro_{W_{2}} \Pro_{W_{1}}x = \Pro_{W_{2}} \Pro_{H_{1}}x= \Pro_{H_{1}}x =\Pro_{W_{1}}x=\Pro_{W_{1} \cap W_{2}} x$.
	
	Case~4.4: Assume that $x \in W^{c}_{1} \cap W^{c}_{2}$. Set $u:=\norm{u_{2}}u_{1}$ and $\eta := \min \{ \eta_{1}\norm{u_{2}}, \eta_{2} \norm{u_{1}} \}$. By \cref{lemma:W1capW2:u1u2LD:vii}\cref{item:lemma:W1capW2:u1u2LD:vii:iv}, 
	\begin{align} \label{eq:PW1capW2:LD:Case4.4:PW1capW2}
	\Pro_{W_{1} \cap W_{2}} x= 
	\begin{cases}
	\Pro_{H_{1}}x, \quad & \text{if } \eta_{1} \norm{u_{2}} \leq \eta_{2} \norm{u_{1}};\\
	\Pro_{H_{2}}x, \quad & \text{if } \eta_{1} \norm{u_{2}} > \eta_{2} \norm{u_{1}}.
	\end{cases}
	\end{align}

	Case~4.4.1: Suppose that $\eta_{1} \norm{u_{2}} \leq \eta_{2} \norm{u_{1}}$. Then $x \in W^{c}_{1} \cap W^{c}_{2}$ and
	 \cref{lemma:u1u2LD:Cases}\cref{lemma:u1u2LD:Cases:>:leq} yield  
	\begin{align*}
	\Pro_{W_{2}}\Pro_{W_{1}} x=\Pro_{W_{2}}\Pro_{H_{1}} x = \Pro_{H_{1}} x \stackrel{\cref{eq:PW1capW2:LD:Case4.4:PW1capW2}}{=} \Pro_{W_{1} \cap W_{2}}x.
	\end{align*}
	
	Case~4.4.2: Suppose that $\eta_{1} \norm{u_{2}} > \eta_{2} \norm{u_{1}}$. 
	Recall that $x \in W^{c}_{1} \cap W^{c}_{2}$, that $u_{1}$ and $u_{2}$ are linearly dependent with $u_{1} \neq 0$ and $u_{2} \neq 0$,   and that
\cref{lemma:u1u2LD:Cases}\cref{lemma:u1u2LD:Cases:>:>}, \cref{remark:FactWFactH}, and \cref{cor:PH2PH1}\cref{cor:PH2PH1:H2}. We obtain that
	\begin{align}\label{eq:theorem:W1capW2:u1u2LD:PW2PW1}
	\Pro_{W_{2}}\Pro_{W_{1}} x=\Pro_{W_{2}}\Pro_{H_{1}} x = \Pro_{H_{2}}\Pro_{H_{1}} x=\Pro_{H_{2}} x \stackrel{\cref{eq:PW1capW2:LD:Case4.4:PW1capW2}}{=}\Pro_{W_{1} \cap W_{2}}x.
	\end{align}

	\textbf{Case~5:}  $u_{1} \neq 0$, $u_{2} \neq 0$,   $\innp{u_{1},u_{2}} <0$, and $\eta_{1} \norm{u_{2}} +\eta_{2} \norm{u_{1} } \geq 0$.

Note that $\mathcal{H}= (W_{1} \cap W_{2}) \cup (W_{1} \cap W^{c}_{2}) \cup ( W^{c}_{1} \cap W_{2}) \cup (W^{c}_{1} \cap W^{c}_{2})$. Then  combine  \cref{lemma:W1capW2:u1u2LD:ix}\cref{lemma:W1capW2:u1u2LD:ix:iv} with $W^{c}_{1} \cap W^{c}_{2} =\varnothing$ to know that we have exactly the following three subcases. 
	
	Case~5.1:   $x \in W_{1} \cap W_{2}$. This is trivial.
	
	Case~5.2:  $x \in W_{1} \cap W^{c}_{2}$. Then by  \cref{remark:FactWFactH} and   \cref{lemma:W1capW2:u1u2LD:ix}\cref{lemma:W1capW2:u1u2LD:ix:ii}, $\Pro_{W_{2}} \Pro_{W_{1}}x =\Pro_{W_{2}} x=\Pro_{H_{2}} x = \Pro_{W_{1} \cap W_{2}} x$.
 
	Case~5.3: $x \in W^{c}_{1} \cap W_{2}$. 
	Then the identities  $	\Pro_{W_{2}}\Pro_{W_{1}} x=\Pro_{W_{2}}\Pro_{H_{1}} x= \Pro_{H_{1}}x= \Pro_{W_{1} \cap W_{2}} x$ follow easily from \cref{lemma:u1u2LD:Cases}\cref{lemma:u1u2LD:Cases:<} and  \cref{lemma:W1capW2:u1u2LD:ix}\cref{lemma:W1capW2:u1u2LD:ix:iii}.

	Altogether, $(\forall x \in \mathcal{H})$ $	\Pro_{W_{2}}\Pro_{W_{1}} x= \Pro_{W_{1} \cap W_{2}} x$, which means that the proof is complete. 
\end{proof}

\subsection*{$u_{1}$ and $u_{2}$ are linearly independent }
In the whole subsection, we assume that 
\begin{empheq}[box=\mybluebox]{equation*}
u_{1} \text{ and } u_{2} \text{ are linearly independent}.
\end{empheq}  
 Denote by
 \begin{subequations}
 	\begin{align}
 	& C_{1}:=\left\{x \in \mathcal{H} ~:~ \innp{x,u_{1}} > \eta_{1} ~\text{and}~ \norm{u_{1}}^{2} (\innp{x,u_{2}} - \eta_{2}) \leq \innp{u_{1}, u_{2}} (\innp{x,u_{1}} -\eta_{1})  \right\}, \label{eq:C1}\\
 	& C_{2}:=\left\{ x \in \mathcal{H} ~:~ \innp{x,u_{2}} > \eta_{2} ~\text{and}~ \norm{u_{2}}^{2} (\innp{x,u_{1}} - \eta_{1}) \leq \innp{u_{1}, u_{2}} (\innp{x,u_{2}} -\eta_{2}) \right\}, \label{eq:C2}\\
 	&C_{3}:= \scalemath{0.95}{\left\{x \in \mathcal{H} ~:~ \norm{u_{1}}^{2} (\innp{x,u_{2}} - \eta_{2}) > \innp{u_{1}, u_{2}} (\innp{x,u_{1}} -\eta_{1}), \norm{u_{2}}^{2} (\innp{x,u_{1}} - \eta_{1}) > \innp{u_{1}, u_{2}} (\innp{x,u_{2}} -\eta_{2})\right\}.} \label{eq:C3}
 	\end{align}
 \end{subequations}

\begin{fact} {\rm \cite[Proposition~29.23]{BC2017}} \label{fact:Projec:Intersect:halfspace}
	Let $x \in \mathcal{H}$. Then $W_{1} \cap W_{2} \neq \varnothing$ and 
	\begin{align*}
	\Pro_{W_{1} \cap W_{2}} x = x -\gamma_{1}u_{1} -\gamma_{2}u_{2},
	\end{align*}
	where exactly one of the following holds:
	\begin{enumerate}
		\item \label{fact:Projec:Intersect:halfspace:i} If $x \in W_{1} \cap W_{2}$, then $\gamma_{1}=\gamma_{2}=0$. 
		\item \label{fact:Projec:Intersect:halfspace:ii} If $x \in C_{1}$, then $\gamma_{1}=\frac{\innp{x,u_{1}} - \eta_{1}}{\norm{u_{1}}^{2}} >0$ and $\gamma_{2} =0$.
		\item \label{fact:Projec:Intersect:halfspace:iii} If $x \in C_{2}$, then $\gamma_{1}=0$ and $\gamma_{2} =\frac{\innp{x,u_{2}} - \eta_{2}}{\norm{u_{2}}^{2}} >0$.
		\item \label{fact:Projec:Intersect:halfspace:iv} If $x \in C_{3}$, then $\gamma_{1} =\frac{\norm{u_{2}}^{2} (\innp{x,u_{1}} - \eta_{1}) - \innp{u_{1}, u_{2}} (\innp{x,u_{2}} -\eta_{2})  }{\norm{u_{1}}^{2} \norm{u_{2}}^{2} -|\innp{u_{1},u_{2}}|^{2}} >0$, $
		\gamma_{2}  =\frac{\norm{u_{1}}^{2} (\innp{x,u_{2}} - \eta_{2}) - \innp{u_{1}, u_{2}} (\innp{x,u_{1}} -\eta_{1}) }{ \norm{u_{1}}^{2} \norm{u_{2}}^{2} -|\innp{u_{1},u_{2}}|^{2} }>0$.
	\end{enumerate}
\end{fact}

\begin{lemma} \label{lem:PW1W2H1H2:IV}
	Let $x \in C_{3}$. Then $	\Pro_{W_{1} \cap W_{2}} x= \Pro_{H_{1} \cap H_{2}} x$.
\end{lemma}

\begin{proof}
	 As a consequence of  \cref{fact:Projec:Intersect:halfspace}\cref{fact:Projec:Intersect:halfspace:iv},  
	\begin{align}\label{eq:cor:PW1W2H1H2:IV:ProV1V1}
	\Pro_{W_{1} \cap W_{2}}x =x -\gamma_{1}u_{1} -\gamma_{2}u_{2}.
	\end{align}
	where $\gamma_{1} = \frac{\norm{u_{2}}^{2} (\innp{x,u_{1}} - \eta_{1}) - \innp{u_{1}, u_{2}} (\innp{x,u_{2}} -\eta_{2})  }{\norm{u_{1}}^{2} \norm{u_{2}}^{2} -|\innp{u_{1},u_{2}}|^{2}} >0$, and  $\gamma_{2} = \frac{\norm{u_{1}}^{2} (\innp{x,u_{2}} - \eta_{2}) - \innp{u_{1}, u_{2}} (\innp{x,u_{1}} -\eta_{1}) }{ \norm{u_{1}}^{2} \norm{u_{2}}^{2} -|\innp{u_{1},u_{2}}|^{2} } >0$.
	According to equations (29.36) and (29.37) in the proof of \cite[Proposition~29.23]{BC2017} which is the \cref{fact:Projec:Intersect:halfspace}, 
	\begin{subequations}
			\begin{align}  
		0&=\gamma_{1} \left( \innp{x- \gamma_{1}u_{1}- \gamma_{2}u_{2},u_{1}} -\eta_{1}\right) \stackrel{\cref{eq:cor:PW1W2H1H2:IV:ProV1V1}}{=}\gamma_{1} \left( \innp{\Pro_{W_{1} \cap W_{2}}x, u_{1}} -\eta_{1}\right),\label{eq:cor:PW1W2H1H2:IV:H1}\\
		0&=\gamma_{2} \left( \innp{x- \gamma_{1}u_{1}-\gamma_{2}u_{2},u_{2}} -\eta_{2}\right)\stackrel{\cref{eq:cor:PW1W2H1H2:IV:ProV1V1}}{=}\gamma_{2} \left( \innp{\Pro_{W_{1} \cap W_{2}}x, u_{2}} -\eta_{2}\right). \label{eq:cor:PW1W2H1H2:IV:H2}
		\end{align}
	\end{subequations}
Bearing in mind that $\gamma_{1}>0$ and $\gamma_{2} >0$, and that \cref{eq:W12},  we know that \cref{eq:cor:PW1W2H1H2:IV:H1} and \cref{eq:cor:PW1W2H1H2:IV:H2} imply that $\Pro_{W_{1} \cap W_{2}}x \in H_{1} \cap H_{2}$.
	Hence, apply \cref{fact:AsubseteqB:Projection} with $A=H_{1} \cap H_{2}$ and $B= W_{1} \cap W_{2}$ to obtain that $\Pro_{W_{1} \cap W_{2}}x=\Pro_{H_{1} \cap H_{2}}x $.
\end{proof}

\begin{remark} \label{rem:Ci:P}
	\begin{enumerate}
		\item \label{rem:Ci:P:C} In view of  \cref{eq:W12} and \cref{cor:PH2PH1}\cref{cor:PH2PH1:notin:EQ}, we know that $C_{1}=\{x \in \mathcal{H}~:~x \notin W_{1} \text{ and } \Pro_{H_{1}}x \in W_{2} \}$,  $C_{2}=\{x \in \mathcal{H}~:~x \notin W_{2} \text{ and } \Pro_{H_{2}}x \in W_{1} \}$, and $C_{3}=\{x \in \mathcal{H}~:~\Pro_{H_{1}}x \notin W_{2}  \text{ and } \Pro_{H_{2}}x \notin W_{1} \}$.
		\item \label{rem:Ci:P:Pformula} By \cref{fact:Projec:Intersect:halfspace}, \cref{fact:Projec:Hyperplane}, and \cref{lem:PW1W2H1H2:IV}, $(x \in C_{1})$ $\Pro_{W_{1} \cap W_{2}}x=\Pro_{H_{1}}x  $, $(x \in C_{2})$ $\Pro_{W_{1} \cap W_{2}}x=\Pro_{H_{2}}x  $, and $(x \in C_{3})$ $\Pro_{W_{1} \cap W_{2}}x=\Pro_{H_{1} \cap H_{2}}x $.
	\end{enumerate}
\end{remark}

\begin{lemma} \label{lemma:W1W2H1H2:SetCharacter}
	\begin{enumerate}
		\item  \label{lemma:W1W2H1H2:SetCharacter:<} Suppose that $\innp{u_{1},u_{2}} <0$.  Then the following statements hold:
		\begin{enumerate}
			\item \label{lemma:W1W2H1H2:SetCharacter:W1CINTW2} $C_{1} \subseteq  W^{c}_{1} \cap \inte(W_{2})$. 
			\item \label{lemma:W1W2H1H2:SetCharacter:W2CINTW1}$C_{2} \subseteq  W^{c}_{2} \cap \inte(W_{1})$. 
			\item \label{lemma:W1W2H1H2:SetCharacter:Remaining} $ \left( W^{c}_{1} \cap H_{2} \right) \cup \left( W^{c}_{2} \cap H_{1} \right) \cup \left( W^{c}_{1} \cap W^{c}_{2} \right) \subseteq  C_{3}$.
		\end{enumerate}
		\item \label{lemma:W1W2H1H2:SetCharacter:geq} Suppose that $\innp{u_{1},u_{2}} \geq 0$.   Then the following statements hold:
		\begin{enumerate}
			\item \label{lemma:W1W2H1H2:SetCharacter:geq:C1} $W^{c}_{1}\cap W_{2} \subseteq C_{1}$.
			\item \label{lemma:W1W2H1H2:SetCharacter:geq:C2} $W_{1}\cap W^{c}_{2} \subseteq C_{2}$.
			\item \label{lemma:W1W2H1H2:SetCharacter:geq:C3} $C_{3} \subseteq W^{c}_{1}\cap W^{c}_{2}$. In particular, if $ \innp{u_{1},u_{2}} =0$, then $C_{3} = W^{c}_{1}\cap W^{c}_{2}$. 
		\end{enumerate}
	\end{enumerate}
	
\end{lemma}

\begin{proof}
	Note that
	\begin{subequations} \label{eq:H}
		\begin{align}
		\mathcal{H} 
		& =  \left(W_{1} \cap W_{2} \right) \cup \left(W_{1} \cap W^{c}_{2} \right)  \cup \left(W^{c}_{1} \cap W_{2} \right) \cup  \left(W^{c}_{1} \cap W^{c}_{2} \right)  \label{eq:H:a}\\
		&= \left(W_{1} \cap W_{2} \right) \cup \left(\inte W_{1} \cap W^{c}_{2} \right) \cup \left(H_{1} \cap W^{c}_{2} \right) \cup \left(W^{c}_{1} \cap \inte W_{2} \right) \cup \left(W^{c}_{1} \cap H_{2} \right) \cup  \left(W^{c}_{1} \cap W^{c}_{2} \right).  \label{eq:H:b}
		\end{align}		
	\end{subequations}
	
	On the other hand,  	\cref{fact:Projec:Intersect:halfspace} implies that
	\begin{align} \label{eq:H:Again}
	\mathcal{H} = \left(W_{1} \cap W_{2} \right) \cup C_{1} \cup C_{2} \cup C_{3},
	\end{align}
	and the sets $W_{1} \cap W_{2} $, $C_{1}$, $C_{2}$ and $C_{3}$ are pairwise disjoint.
	
	\cref{lemma:W1W2H1H2:SetCharacter:W1CINTW2}: Let $x \in C_{1}$.  Clearly, $\innp{x,u_{1}} > \eta_{1} $  is equivalent to $x \in W^{c}_{1}$. Moreover,   $\innp{u_{1},u_{2}} <0$ yields $ \innp{u_{1}, u_{2}} (\innp{x,u_{1}} -\eta_{1}) <0$. Hence, $\norm{u_{1}}^{2} (\innp{x,u_{2}} - \eta_{2}) \leq \innp{u_{1}, u_{2}} (\innp{x,u_{1}} -\eta_{1})$ implies that $\innp{x,u_{2}} - \eta_{2} <0$, i.e., $x \in \inte (W_{2})$. Hence, $C_{1} \subseteq  W^{c}_{1} \cap \inte(W_{2})$.   
	
	\cref{lemma:W1W2H1H2:SetCharacter:W2CINTW1}: The proof is similar to the proof of \cref{lemma:W1W2H1H2:SetCharacter:W1CINTW2}.

	\cref{lemma:W1W2H1H2:SetCharacter:Remaining}: 
	This is from 	\cref{lemma:W1W2H1H2:SetCharacter:W1CINTW2}, \cref{lemma:W1W2H1H2:SetCharacter:W2CINTW1}, \cref{eq:H:b} and \cref{eq:H:Again}.

\cref{lemma:W1W2H1H2:SetCharacter:geq:C1}: Let $x \in W^{c}_{1}\cap W_{2}$. Then $\innp{x,u_{1} }>\eta_{1}$ and $ \innp{x,u_{2} }\leq \eta_{2}$ are from \cref{eq:W12}. Moreover,   $\innp{u_{1},u_{2}} \geq 0$ leads to $\norm{u_{1}}^{2} (\innp{x,u_{2}} - \eta_{2}) \leq 0 \leq \innp{u_{1}, u_{2}} (\innp{x,u_{1}} -\eta_{1}) $, which, due to  \cref{eq:C1}, implies $x \in C_{1}$. Hence, $W^{c}_{1}\cap W_{2} \subseteq  C_{1}$.

\cref{lemma:W1W2H1H2:SetCharacter:geq:C2}: By analogous proof of \cref{lemma:W1W2H1H2:SetCharacter:geq:C1}, we know that \cref{lemma:W1W2H1H2:SetCharacter:geq:C2} holds.

\cref{lemma:W1W2H1H2:SetCharacter:geq:C3}: Combine \cref{lemma:W1W2H1H2:SetCharacter:geq:C1}, \cref{lemma:W1W2H1H2:SetCharacter:geq:C2}, \cref{eq:H:a} and \cref{eq:H:Again} to obtain that $C_{3} \subseteq W^{c}_{1}\cap W^{c}_{2}$.  Suppose that $\innp{u_{1}, u_{2}}=0$. Let $x \in W^{c}_{1}\cap W^{c}_{2}$. Then by \cref{eq:W12}, $\norm{u_{1}}^{2} (\innp{x,u_{2}} - \eta_{2}) > 0=\innp{u_{1}, u_{2}} (\innp{x,u_{1}} -\eta_{1}) $ and $\norm{u_{2}}^{2} (\innp{x,u_{1}} - \eta_{1}) > 0=\innp{u_{1}, u_{2}} (\innp{x,u_{2}} -\eta_{2})$, which, by \cref{eq:C3}, implies that $x \in C_{3}$. Therefore, the last assertion in \cref{lemma:W1W2H1H2:SetCharacter:geq:C3} holds.  
\end{proof}

\begin{lemma} \label{lem:PH2LEQPH1H2}
	Suppose that   $\innp{u_{1},u_{2}} <0$. 
	Let $x \in C_{3}$.  The following hold:
	\begin{enumerate}
		\item \label{lem:PH2LEQPH1H2:W1} Assume that $x \in W_{1}$. Then 
		$
		\norm{\Pro_{H_{2}}x - \Pro_{H_{1} \cap H_{2}}x } \leq \norm{ \Pro_{H_{2}}\Pro_{H_{1}}x  - \Pro_{H_{1} \cap H_{2}}x  }.
		$
		\item \label{lem:PH2LEQPH1H2:W1C} Assume that $x \in W^{c}_{1}$. Then
		\begin{align*}
		\Pro_{W_{2}}\Pro_{W_{1}} x = \Pro_{W_{2}}\Pro_{H_{1}} x =\Pro_{H_{2}}\Pro_{H_{1}} x \in C_{3} \text{ and } \Pro_{W_{1} \cap W_{2}} \Pro_{H_{2}}\Pro_{H_{1}}x  = \Pro_{ H_{1} \cap H_{2}} \Pro_{H_{2}}\Pro_{H_{1}}x = \Pro_{ H_{1} \cap H_{2}}x.
	\end{align*}
	\end{enumerate}
	
\end{lemma}

\begin{proof}

	\cref{lem:PH2LEQPH1H2:W1}: 
	Using \cref{fact:Projec:Hyperplane}, we know that
		\begin{align} \label{eq:lem:PH2LEQPH1H2:H1H2}
	 \Pro_{H_{1}}x  = x+ \frac{\eta_{1} - \innp{x,u_{1}}}{\norm{u_{1}}^{2}}u_{1},~ \Pro_{H_{2}}x  = x+ \frac{\eta_{2} - \innp{x,u_{2}}}{\norm{u_{2}}^{2}}u_{2}, \text{ and }
\Pro_{H_{2}}\Pro_{H_{1}}x  =\Pro_{H_{1}}x + \frac{\eta_{2} - \innp{ \Pro_{H_{1}}x ,u_{2}}}{\norm{u_{2}}^{2}}u_{2},
		\end{align}
which implies that
	\begin{align} \label{eq:lem:PH2LEQPH1H2:NORM:H12}
	\norm{x - \Pro_{H_{1}}x }^{2}= \frac{(\eta_{1} - \innp{x,u_{1}} )^{2}}{\norm{u_{1}}^{2}}, \quad \text{and} \quad
	\norm{x - \Pro_{H_{2}}x }^{2} = \frac{(\eta_{2} - \innp{x,u_{2}} )^{2}}{\norm{u_{2}}^{2}}.  
	\end{align}
Apply the third and the first identities in \cref{eq:lem:PH2LEQPH1H2:H1H2} to the following first and the second equations, respectively, to obtain that 
		\begin{align}\label{eq:lem:PH2LEQPH1H2:NORM:H1:H2H1}
		\norm{\Pro_{H_{1}}x -\Pro_{H_{2}}\Pro_{H_{1}}x}^{2}  =\frac{(\eta_{2} - \innp{\Pro_{H_{1}}x,u_{2}} )^{2}}{\norm{u_{2}}^{2}}=  \frac{1}{ \norm{u_{1}}^{4} \norm{u_{2}}^{2}}  \left( \norm{u_{1}}^{2} ( \eta_{2} - \innp{x,u_{2}} )  - \innp{u_{1},u_{2}} (\eta_{1} - \innp{x,u_{1}}) \right)^{2}.
		\end{align}
	On the one hand, 
	\begin{subequations} 
		\begin{align}
		&\norm{\Pro_{H_{2}}x - \Pro_{H_{1} \cap H_{2}}x } \leq \norm{ \Pro_{H_{2}}\Pro_{H_{1}}x  - \Pro_{H_{1} \cap H_{2}}x  } \label{eq:cor:PH2LEQPH1H2:last}\\
		\Leftrightarrow & \norm{x - \Pro_{H_{1} \cap H_{2}}x}^{2} - \norm{ x - \Pro_{ H_{2}}x}^{2} \leq \norm{\Pro_{H_{1}}x -\Pro_{H_{1} \cap H_{2}}x}^{2} - \norm{\Pro_{H_{1}}x -\Pro_{H_{2}}\Pro_{H_{1}}x}^{2} \nonumber\\
		\Leftrightarrow & \norm{x - \Pro_{H_{1} \cap H_{2}}x}^{2} - \norm{ x - \Pro_{ H_{2}}x}^{2} \leq \norm{x -\Pro_{H_{1} \cap H_{2}}x}^{2}  - \norm{x - \Pro_{H_{1}}x }^{2} - \norm{\Pro_{H_{1}}x -\Pro_{H_{2}}\Pro_{H_{1}}x}^{2}\nonumber\\
		\Leftrightarrow & \norm{x - \Pro_{H_{1}}x }^{2} +  \norm{\Pro_{H_{1}}x -\Pro_{H_{2}}\Pro_{H_{1}}x}^{2}
		\leq \norm{ x - \Pro_{ H_{2}}x}^{2}\nonumber\\
		\Leftrightarrow & \frac{(\eta_{1} - \innp{x,u_{1}} )^{2}}{\norm{u_{1}}^{2}} +  \frac{1}{ \norm{u_{1}}^{4} \norm{u_{2}}^{2}}  \left( \norm{u_{1}}^{2} ( \eta_{2} - \innp{x,u_{2}} )  - \innp{u_{1},u_{2}} (\eta_{1} - \innp{x,u_{1}}) \right)^{2} \leq  \frac{(\eta_{2} - \innp{x,u_{2}} )^{2}}{\norm{u_{2}}^{2}} \nonumber\\
		\Leftrightarrow &  \norm{u_{1}}^{2} \norm{u_{2}}^{2} (\eta_{1} - \innp{x,u_{1}} )^{2} + \left( \norm{u_{1}}^{2} ( \eta_{2} - \innp{x,u_{2}} )  - \innp{u_{1},u_{2}} (\eta_{1} - \innp{x,u_{1}}) \right)^{2} \leq  \norm{u_{1}}^{4} (\eta_{2} - \innp{x,u_{2}} )^{2} \nonumber\\
		\Leftrightarrow &  \norm{u_{1}}^{2} \norm{u_{2}}^{2} (\eta_{1} - \innp{x,u_{1}} )^{2} + \innp{u_{1},u_{2}}^{2} (\eta_{1} - \innp{x,u_{1}})^{2} \leq 2 \norm{u_{1}}^{2} ( \eta_{2} - \innp{x,u_{2}} )  \innp{u_{1},u_{2}} (\eta_{1} - \innp{x,u_{1}})\nonumber\\
		\Leftrightarrow &    \norm{u_{1}}^{2} (\eta_{1} - \innp{x,u_{1}})  \left( \norm{u_{2}}^{2} (\eta_{1} - \innp{x,u_{1}} ) -  \innp{u_{1},u_{2}} ( \eta_{2} - \innp{x,u_{2}} ) \right) \nonumber\\
		& + \innp{u_{1},u_{2}} (\eta_{1} - \innp{x,u_{1}}) \left( \innp{u_{1},u_{2}} (\eta_{1} - \innp{x,u_{1}})  -  \norm{u_{1}}^{2} ( \eta_{2} - \innp{x,u_{2}} )   \right)  \leq 0, \label{eq:cor:PH2LEQPH1H2:last:Last}
		\end{align}
	\end{subequations} 
	where the first two equivalences are from \cite[Proposition~2.10]{BOyW2018Proper}, the fourth equivalence is from \cref{eq:lem:PH2LEQPH1H2:NORM:H12} and \cref{eq:lem:PH2LEQPH1H2:NORM:H1:H2H1},  and $ \left( \norm{u_{1}}^{2} ( \eta_{2} - \innp{x,u_{2}} )  - \innp{u_{1},u_{2}} (\eta_{1} - \innp{x,u_{1}}) \right)^{2}  = \norm{u_{1}}^{4} ( \eta_{2} - \innp{x,u_{2}} )^{2} - 2 \norm{u_{1}}^{2} ( \eta_{2} - \innp{x,u_{2}} )\innp{u_{1},u_{2}} (\eta_{1} - \innp{x,u_{1}}) + \innp{u_{1},u_{2}}^{2} (\eta_{1} - \innp{x,u_{1}})^{2} $ yields the sixth equivalence. Hence, it is equivalent to show \cref{eq:cor:PH2LEQPH1H2:last:Last}.
	
	Recall that $x \in C_{3}$. Then,   \cref{eq:C3} yields
	$\norm{u_{1}}^{2} (\innp{x,u_{2}} - \eta_{2}) > \innp{u_{1}, u_{2}} (\innp{x,u_{1}} -\eta_{1})$ and $ \norm{u_{2}}^{2} (\innp{x,u_{1}} - \eta_{1}) > \innp{u_{1}, u_{2}} (\innp{x,u_{2}} -\eta_{2})$. By \cref{eq:W12}, $x \in W_{1}$ means $\eta_{1} - \innp{x,u_{1}} \geq 0 $. Hence, by assumptions, $x \in C_{3}$,   $x \in W_{1}$ and $ \innp{u_{1},u_{2}}  <0 $, 
	\begin{subequations}
		\begin{align*}
		&\norm{u_{1}}^{2} (\eta_{1} - \innp{x,u_{1}})  \left( \norm{u_{2}}^{2} (\eta_{1} - \innp{x,u_{1}} ) -  \innp{u_{1},u_{2}} ( \eta_{2} - \innp{x,u_{2}} ) \right) \leq 0,\\
		&\innp{u_{1},u_{2}} (\eta_{1} - \innp{x,u_{1}}) \left( \innp{u_{1},u_{2}} (\eta_{1} - \innp{x,u_{1}})  -  \norm{u_{1}}^{2} ( \eta_{2} - \innp{x,u_{2}} )   \right)  \leq 0,
		\end{align*}
	\end{subequations}
	which imply that \cref{eq:cor:PH2LEQPH1H2:last:Last} holds.
Accordingly, \cref{lem:PH2LEQPH1H2:W1} holds.

	\cref{lem:PH2LEQPH1H2:W1C}: Recall again that $x \in C_{3}$ means
	\begin{align} \label{eq:cor:PH2LEQPH1H2:x1}
	\norm{u_{1}}^{2} (\innp{x,u_{2}} - \eta_{2}) > \innp{u_{1}, u_{2}} (\innp{x,u_{1}} -\eta_{1})~   \text{and}~
	\norm{u_{2}}^{2} (\innp{x,u_{1}} - \eta_{1}) > \innp{u_{1}, u_{2}} (\innp{x,u_{2}} -\eta_{2}).
	\end{align} 
	Combine the first inequality in \cref{eq:cor:PH2LEQPH1H2:x1} with the assumption $\innp{u_{1}, u_{2}} <0$ to know that
	\begin{align} \label{eq:cor:PH2LEQPH1H2:x1:u1u2}
	\innp{u_{1},u_{2}}\left(  \norm{u_{1}}^{2} \left( \eta_{2} -\innp{x,u_{2}}  \right) -  \innp{u_{1},u_{2}} \left( \eta_{1} - \innp{x,u_{1}} \right)\right) >0.
	\end{align}
	Clearly, $x \in C_{3}$ and  \cref{rem:Ci:P}\cref{rem:Ci:P:C} imply that $\Pro_{H_{1}}x \notin W_{2}$. Combine this with the assumption, $x \in W^{c}_{1}$  and  \cref{remark:FactWFactH},  to obtain that $	\Pro_{W_{2}}\Pro_{W_{1}} x =\Pro_{W_{2}}\Pro_{H_{1}} x =\Pro_{H_{2}}\Pro_{H_{1}} x$.
	Moreover, by \cref{cor:PH2PH1}\cref{cor:PH2PH1:H2H1},
		\begin{align*}
		\Pro_{H_{2}}\Pro_{H_{1}} x 
		= 
		x+ \frac{\eta_{1} - \innp{x,u_{1}}}{\norm{u_{1}}^{2}}u_{1} +\frac{1}{\norm{u_{2}}^{2} \norm{u_{1}}^{2}} \left( \norm{u_{1}}^{2} \left( \eta_{2} - \innp{x, u_{2}} \right) -   \innp{u_{1}, u_{2}}\left( \eta_{1} - \innp{x,u_{1}}\right) \right) u_{2},
		\end{align*} 
so
	\begin{align*}
	&\eta_{1} - \innp{ \Pro_{H_{2}}\Pro_{H_{1}}x, u_{1} } \\ 
=& \eta_{1} - \innp{x,u_{1}} -\left(\eta_{1} - \innp{x,u_{1}} \right) -\frac{1}{\norm{u_{2}}^{2} \norm{u_{1}}^{2}} \left( \norm{u_{1}}^{2} \left( \eta_{2} - \innp{x, u_{2}} \right) -   \innp{u_{1}, u_{2}}\left( \eta_{1} - \innp{x,u_{1}}\right) \right) \innp{u_{2},u_{1}}\\
	=& -\frac{1}{\norm{u_{1}}^{2}\norm{u_{2}}^{2}}\innp{u_{2},u_{1}}\left(  \norm{u_{1}}^{2} \left( \eta_{2} -\innp{x,u_{2}}  \right) -  \innp{u_{1},u_{2}} \left( \eta_{1} - \innp{x,u_{1}} \right)\right)   \stackrel{\cref{eq:cor:PH2LEQPH1H2:x1:u1u2}}{ <}0,
	\end{align*}
	which shows that $\Pro_{H_{2}}\Pro_{H_{1}}x \in W^{c}_{1}$. Hence, by \cref{lemma:W1W2H1H2:SetCharacter}\cref{lemma:W1W2H1H2:SetCharacter:Remaining}, $	\Pro_{H_{2}}\Pro_{H_{1}}x \in W^{c}_{1} \cap H_{2} \subseteq C_{3}$.
	Therefore, by \cref{lem:PW1W2H1H2:IV} and \cref{fact:ExchangeProj}, $	\Pro_{W_{1} \cap W_{2}} \Pro_{H_{2}}\Pro_{H_{1}}x  = \Pro_{ H_{1} \cap H_{2}} \Pro_{H_{2}}\Pro_{H_{1}}x = \Pro_{ H_{1} \cap H_{2}}x$.
\end{proof}

\begin{theorem} \label{thm:PROW1W2:BAM}
	Assume that $\innp{u_{1},u_{2}} <0$. Denote by $\gamma :=\frac{|\innp{u_{1}, u_{2}}| }{\norm{u_{1}} \norm{u_{2}}}$.
	 Then $ \gamma \in  \left[0,1\right[\,$ and $\Pro_{W_{2}}\Pro_{W_{1}}$ is a $\gamma$-BAM.
	Consequently, $(\forall x \in \mathcal{H})$ $(\forall k \in \mathbb{N})$  $\norm{(\Pro_{W_{2}}\Pro_{W_{1}})^{k}x -\Pro_{W_{1} \cap W_{2}}x} \leq \gamma^{k} \norm{x- \Pro_{W_{1} \cap W_{2}}x}$.
\end{theorem}	

\begin{proof}
	Recall that $u_{1}$ and $u_{2}$ are linearly independent. Then $\gamma \in \left[0,1\right[\,$ follows from  \cref{lem:H1H2Ineq}. 
As a consequence of  \cite[Corollary~4.5.2]{Cegielski}, $\Fix \Pro_{W_{2}}\Pro_{W_{1}}=W_{1} \cap W_{2}$ is a nonempty closed convex subset of $\mathcal{H}$. Hence, by \cref{def:BAM} and   \cref{fact:BAM:Properties}, it remains to show that for every $x \in \mathcal{H}$, 
	\begin{enumerate}
		\item \label{thm:PROW1W2:BAM:product} $\Pro_{W_{1} \cap W_{2}} \Pro_{W_{2}}\Pro_{W_{1}}x =\Pro_{W_{1} \cap W_{2}}x$, and
		\item \label{thm:PROW1W2:BAM:Indequa}   $ \norm{ \Pro_{W_{2}}\Pro_{W_{1}}x -  \Pro_{W_{1} \cap W_{2}}x} \leq \gamma  \norm{ x -  \Pro_{W_{1} \cap W_{2}}x}$.
	\end{enumerate}
	 
	Let $x \in \mathcal{H}$.  Note that if   $\Pro_{W_{2}}\Pro_{W_{1}}x=\Pro_{W_{1} \cap W_{2}}x$, then clearly \cref{thm:PROW1W2:BAM:product} and \cref{thm:PROW1W2:BAM:Indequa}  hold. 
	
	Taking  \cref{fact:Projec:Intersect:halfspace} into account, we have exactly the following four cases. 
	
	\textbf{Case 1:}  $x \in W_{1} \cap W_{2}$. This is trivial.
	
	\textbf{Case 2:}  $x \in C_{1}$. Then   \cref{fact:Projec:Intersect:halfspace}\cref{fact:Projec:Intersect:halfspace:ii} and \cref{fact:Projec:halfspaces} imply that 
	\begin{align} \label{eq:thm:PROW1W2:BAM:II:Prow1w2}
	\Pro_{W_{1} \cap W_{2}}x=x+ \frac{\eta_{1} - \innp{x,u_{1}}}{\norm{u_{1}}^{2}}u_{1}= \Pro_{H_{1}}x \in H_{1} \cap W_{2}. 
	\end{align} 
	According to \cref{lemma:W1W2H1H2:SetCharacter}\cref{lemma:W1W2H1H2:SetCharacter:W1CINTW2}, $x \in C_{1} \subseteq W^{c}_{1} \cap \inte(W_{2})$, which, by \cref{remark:FactWFactH} and \cref{eq:thm:PROW1W2:BAM:II:Prow1w2},  deduces that
	\begin{align*}
	\Pro_{W_{2}}\Pro_{W_{1}}x =\Pro_{W_{2}}\Pro_{H_{1}}x = \Pro_{H_{1}}x=	\Pro_{W_{1} \cap W_{2}}x.
	\end{align*}

	\textbf{Case 3:}  $x \in C_{2}$. Then enforcing  \cref{fact:Projec:Intersect:halfspace}\cref{fact:Projec:Intersect:halfspace:iii} and \cref{fact:Projec:halfspaces}, we know that 
	\begin{align} \label{eq:thm:PROW1W2:BAM:III:Prow1w2}
	\Pro_{W_{1} \cap W_{2}}x=x+ \frac{\eta_{2} - \innp{x,u_{2}}}{\norm{u_{2}}^{2}}u_{2} = \Pro_{H_{2}}x.
	\end{align} 
In view of \cref{lemma:W1W2H1H2:SetCharacter}\cref{lemma:W1W2H1H2:SetCharacter:W2CINTW1},  $	x \in C_{2} \subseteq  W^{c}_{2} \cap \inte(W_{1})$.
Hence, by  \cref{remark:FactWFactH} and \cref{eq:thm:PROW1W2:BAM:III:Prow1w2},
	\begin{align} \label{eq:thm:PROW1W2:BAM:Case3touse}
 \Pro_{W_{2}}\Pro_{W_{1}}x  = \Pro_{W_{2}}x= \Pro_{H_{2}}x= \Pro_{W_{1} \cap W_{2}}x.
	\end{align}

	\textbf{Case 4:}   $x \in C_{3}$. 
Then   \cref{lem:PW1W2H1H2:IV} leads to
	\begin{align} \label{eq:thm:PROW1W2:BAM:Case4W1W2}
	\Pro_{W_{1} \cap W_{2}}x = \Pro_{H_{1} \cap H_{2}}x.
	\end{align}
	
	Note that $C_{3} \subseteq \mathcal{H} = \left(W_{1} \cap W_{2} \right) \cup \left(W_{1} \cap W^{c}_{2} \right)  \cup W^{c}_{1}  $,   \cref{fact:Projec:Intersect:halfspace}, and $C_{3} \cap (W_{1} \cap W_{2} ) =\varnothing$. We have exactly the following two subcases.

	Case 4.1:  $x \in  W_{1} \cap W^{c}_{2}$. Combine $x \in C_{3}$ with  \cref{remark:FactWFactH}, \cref{rem:Ci:P}\cref{rem:Ci:P:C}, and \cref{lemma:W1W2H1H2:SetCharacter}\cref{lemma:W1W2H1H2:SetCharacter:Remaining} to see that
	\begin{align} \label{eq:thm:PROW1W2:BAM:Case4:4.1:W2W1}
	\Pro_{W_{2}}\Pro_{W_{1}}x=\Pro_{W_{2}}x= \Pro_{H_{2}}x\in H_{2} \cap W^{c}_{1} \subseteq C_{3}.
	\end{align} 
	 Using \cref{lem:PW1W2H1H2:IV} and \cref{fact:ExchangeProj}, we have that 
	\begin{align}  \label{eq:thm:PROW1W2:BAM:Case4:4.1:W1W2}
	\Pro_{W_{1} \cap W_{2}} \Pro_{W_{2}}\Pro_{W_{1}}x=\Pro_{H_{1} \cap H_{2}} \Pro_{W_{2}}\Pro_{W_{1}}x \stackrel{\cref{eq:thm:PROW1W2:BAM:Case4:4.1:W2W1}}{=} \Pro_{ H_{1} \cap H_{2}} \Pro_{H_{2}}x  = \Pro_{ H_{1} \cap H_{2}}x\stackrel{\cref{eq:thm:PROW1W2:BAM:Case4W1W2}}{=} \Pro_{W_{1} \cap W_{2}}x.
	\end{align} 
Moreover, $ \norm{ \Pro_{W_{2}}\Pro_{W_{1}}x -  \Pro_{W_{1} \cap W_{2}}x} \stackrel{\cref{eq:thm:PROW1W2:BAM:Case4:4.1:W2W1}}{=} \norm{ \Pro_{H_{2}}x -  \Pro_{W_{1} \cap W_{2}}x} \stackrel{\cref{eq:thm:PROW1W2:BAM:Case4W1W2}}{=} \norm{\Pro_{H_{2}}x  -\Pro_{ H_{1} \cap H_{2}} x } \leq \norm{\Pro_{H_{2}}\Pro_{H_{1}}x  -\Pro_{ H_{1} \cap H_{2}} x }  \leq \gamma \norm{x -\Pro_{H_{1} \cap H_{2}}x } \stackrel{\cref{eq:thm:PROW1W2:BAM:Case4W1W2}}{=} \gamma \norm{x -\Pro_{W_{1} \cap W_{2}}x }$,
where the last two inequalities are from \cref{lem:PH2LEQPH1H2}\cref{lem:PH2LEQPH1H2:W1} and \cref{lem:H1H2Ineq} respectively.
Hence, \cref{thm:PROW1W2:BAM:product} and
	\cref{thm:PROW1W2:BAM:Indequa} hold. 
	
	Case 4.2:  $x \in  W^{c}_{1} $. Because $x \in C_{3}$, by  \cref{rem:Ci:P}\cref{rem:Ci:P:C},  $\Pro_{H_{1}}x \notin W_{2}$.  
	Hence, by \cref{remark:FactWFactH} and	\cref{lem:PH2LEQPH1H2}\cref{lem:PH2LEQPH1H2:W1C}, 
	\begin{align} \label{thm:PROW1W2:BAM:Case4:4.2:touse}
	\Pro_{W_{2}}\Pro_{W_{1}} x = \Pro_{W_{2}}\Pro_{H_{1}} x =\Pro_{H_{2}}\Pro_{H_{1}} x \quad \text{and} \quad \Pro_{W_{1} \cap W_{2}} \Pro_{W_{2}}\Pro_{W_{1}} x    = \Pro_{ H_{1} \cap H_{2}}x\stackrel{\cref{eq:thm:PROW1W2:BAM:Case4W1W2}}{=}\Pro_{W_{1} \cap W_{2}}x.
	\end{align}
Moreover, $\norm{ \Pro_{W_{2}}\Pro_{W_{1}}x -  \Pro_{W_{1} \cap W_{2}}x}  \stackrel{\cref{thm:PROW1W2:BAM:Case4:4.2:touse}}{=} \norm{\Pro_{H_{2}}\Pro_{H_{1}}x  -\Pro_{ W_{1} \cap W_{2}} x } \stackrel{\cref{eq:thm:PROW1W2:BAM:Case4W1W2}}{=} \norm{\Pro_{H_{2}}\Pro_{H_{1}}x  -\Pro_{ H_{1} \cap H_{2}} x }  \leq \gamma \norm{x -\Pro_{H_{1} \cap H_{2}}x}$ $ \stackrel{\cref{eq:thm:PROW1W2:BAM:Case4W1W2}}{=} \gamma \norm{x -\Pro_{W_{1} \cap W_{2}}x }$, where the inequality is from  \cref{lem:H1H2Ineq}.
	Therefore, \cref{thm:PROW1W2:BAM:product} and 	\cref{thm:PROW1W2:BAM:Indequa} are true.
\end{proof}

\begin{proposition} \label{prop:u1u2geq0}
	Suppose that $\innp{u_{1},u_{2}} \geq 0$. Let $x \in \mathcal{H}$. Then the following statements hold:
	\begin{enumerate}
		\item \label{prop:u1u2geq0:W1W2} Suppose that $x \in W_{1} \cap W_{2}$. Then $\Pro_{W_{2}}\Pro_{W_{1}}x=x=\Pro_{W_{1} \cap W_{2}} x $.
		\item  \label{prop:u1u2geq0:Wc1W2} Suppose that $x \in W^{c}_{1} \cap W_{2}$. Then $\Pro_{W_{2}}\Pro_{W_{1}}x=\Pro_{H_{1}}x=\Pro_{W_{1} \cap W_{2}} x $.
		\item  \label{prop:u1u2geq0:W1Wc2} Suppose that $x \in W_{1} \cap W^{c}_{2}$. Then $\Pro_{W_{2}}\Pro_{W_{1}}x=\Pro_{H_{2}}x=\Pro_{W_{1} \cap W_{2}} x $.
	\end{enumerate}
\end{proposition}

\begin{proof}
	\cref{prop:u1u2geq0:W1W2}: This is trivial. 
	
	\cref{prop:u1u2geq0:Wc1W2}: According to \cref{lemma:W1W2H1H2:SetCharacter}\cref{lemma:W1W2H1H2:SetCharacter:geq:C1},  $x \in W^{c}_{1} \cap W_{2} \subseteq C_{1}$. Then    \cref{fact:Projec:Intersect:halfspace}\cref{fact:Projec:Intersect:halfspace:ii} and \cref{fact:Projec:Hyperplane} deduce 
	\begin{align} \label{eq:prop:u1u2geq0:W1capW2}
	\Pro_{W_{1} \cap W_{2}} x=\Pro_{H_{1}}x.
	\end{align}
In view of \cref{eq:W12}, $x \in W^{c}_{1} \cap W_{2} $ means $\innp{x,u_{1}} >\eta_{1} $ and $\innp{x,u_{2}} \leq \eta_{2}$, which combining with  $\innp{u_{1},u_{2}} \geq 0$,  imply that $\eta_{2} -\innp{x,u_{2}} \geq 0$ and $ - \innp{u_{1},u_{2}} ( \eta_{1} -\innp{x, u_{1} } ) \geq 0$. Hence,
	\begin{align} \label{eq:prop:u1u2geq0:Wc1W2:pre}
	\norm{u_{1}}^{2} ( \eta_{2} -\innp{x,u_{2}}) - \innp{u_{1},u_{2}} ( \eta_{1} -\innp{x, u_{1} } ) \geq 0,
	\end{align}
	which, connecting with  \cref{cor:PH2PH1}\cref{cor:PH2PH1:notin:EQ}, yields that $\Pro_{H_{1}}x \in W_{2} $. Hence, 
	\begin{align*}
	 \Pro_{W_{2}}\Pro_{W_{1}}x=\Pro_{W_{2}}\Pro_{H_{1}}x=\Pro_{H_{1}}x\stackrel{\cref{eq:prop:u1u2geq0:W1capW2}}{=}\Pro_{W_{1} \cap W_{2}} x. 
	\end{align*}

	 \cref{prop:u1u2geq0:W1Wc2}: Due to \cref{lemma:W1W2H1H2:SetCharacter}\cref{lemma:W1W2H1H2:SetCharacter:geq:C2},  $x \in W_{1} \cap W^{c}_{2} \subseteq C_{2}$. Then by  \cref{fact:Projec:Intersect:halfspace}\cref{fact:Projec:Intersect:halfspace:iii} and \cref{fact:Projec:Hyperplane}, 
	 \begin{align} \label{eq:prop:u1u2geq0:W1capW2:H2}
	 \Pro_{W_{1} \cap W_{2}} x=\Pro_{H_{2}}x.
	 \end{align}
	 Hence,  $x \in W_{1} \cap W^{c}_{2} $ and \cref{remark:FactWFactH} imply that
	 \begin{align*}
	 \Pro_{W_{2}}\Pro_{W_{1}}x=\Pro_{W_{2}}x=\Pro_{H_{2}}x\stackrel{\cref{eq:prop:u1u2geq0:W1capW2:H2}}{=}\Pro_{W_{1} \cap W_{2}} x. 
	 \end{align*}
\end{proof}

\begin{proposition} \label{prop:>0:inW1W2}
	Suppose that $\innp{u_{1}, u_{2}} > 0$. Then $(\forall x \in \mathcal{H})$ $\Pro_{W_{2}}\Pro_{W_{1}}x \in W_{1} \cap W_{2}$. In particular, if $ x \in W_{1} \cup W_{2}$, then  $\Pro_{W_{2}}\Pro_{W_{1}}x =\Pro_{W_{1} \cap W_{2}}x $; if $ x \in W^{c}_{1} \cap W^{c}_{2}$ with $ \Pro_{H_{1}}x \in W_{2} $, then  $\Pro_{W_{2}}\Pro_{W_{1}}x=\Pro_{H_{1}}x  =\Pro_{W_{1} \cap W_{2}}x $; if $ x \in W^{c}_{1} \cap W^{c}_{2}$ with $ \Pro_{H_{1}}x \notin W_{2} $, then $\Pro_{W_{2}}\Pro_{W_{1}}x \in W_{1} \cap W_{2} \smallsetminus \{ \Pro_{W_{1} \cap W_{2}}x  \}$.
\end{proposition}

\begin{proof}
	Let $x\in \mathcal{H}$. If $x \in W_{1} \cup W_{2} $, then  
	$\Pro_{W_{2}}\Pro_{W_{1}}x =\Pro_{W_{1} \cap W_{2}}x \in W_{1} \cap W_{2}$ follows from \cref{prop:u1u2geq0}. 
	
	Suppose that $x \in W^{c}_{1} \cap W^{c}_{2} $. 
	Then  \cref{remark:FactWFactH} leads to  $\Pro_{W_{2}}\Pro_{W_{1}}x =\Pro_{W_{2}}\Pro_{H_{1}}x $.   If $\Pro_{H_{1}}x \in W_{2} $, then by \cref{rem:Ci:P}\cref{rem:Ci:P:C}, $x \in C_{1}$. Hence,  by \cref{rem:Ci:P}\cref{rem:Ci:P:Pformula}, $\Pro_{W_{2}}\Pro_{W_{1}}x =\Pro_{W_{2}}\Pro_{H_{1}}x=\Pro_{H_{1}}x =\Pro_{W_{1} \cap W_{2}}x$.
	Assume that $\Pro_{H_{1}}x \notin W_{2} $. 
Then employing \cref{cor:PH2PH1}\cref{cor:PH2PH1:inW1}, we have that
\begin{align} \label{eq:prop:>0:inW1W2}
\Pro_{W_{2}}\Pro_{W_{1}}x=\Pro_{H_{2}}\Pro_{H_{1}}x \in \inte W_{1} \cap H_{2} \subseteq W_{1} \cap W_{2}.
\end{align} 
On the other hand, combine  \cref{rem:Ci:P}\cref{rem:Ci:P:C}$\&$\cref{rem:Ci:P:Pformula} with $x \in W^{c}_{1} \cap W^{c}_{2} $ and $\Pro_{H_{1}}x \notin W_{2} $ to see that
 $x \in C_{2} \cup C_{3}$ and that either $\Pro_{W_{1} \cap W_{2}}x=\Pro_{H_{1} \cap H_{2}}x$ or $\Pro_{W_{1} \cap W_{2}}x=\Pro_{ H_{2}}x$. Clearly, $\Pro_{H_{1} \cap H_{2}}x \notin \inte W_{1}$, so by \cref{eq:prop:>0:inW1W2}, $\Pro_{W_{2}}\Pro_{W_{1}}x \neq   \Pro_{H_{1} \cap H_{2}}x$. Moreover, by \cref{fact:Projec:Hyperplane} and by the first identity in  \cref{cor:PH2PH1}\cref{cor:PH2PH1:H2H1}, if $ \Pro_{H_{2}}\Pro_{H_{1}}x =\Pro_{ H_{2}}x$, then $0 = \frac{\eta_{1} -\innp{x,u_{1}}}{\norm{u_{1}}^{2}} u_{1}   -\frac{ (\eta_{1} -\innp{x,u_{1}}) \innp{u_{1},u_{2}}}{ \norm{u_{1}}^{2} \norm{u_{2}}^{2}} u_{2}  $, which contradicts with the assumption that $u_{1}$ and $u_{2}$ are linearly independent. Therefore, by \cref{eq:prop:>0:inW1W2}, $\Pro_{W_{2}}\Pro_{W_{1}}x \neq   \Pro_{ H_{2}}x$. 
\end{proof}
\begin{lemma} \label{lem:PW1W2:BAM:u1u2eq0:equiva}
	Suppose that $\innp{u_{1}, u_{2}} = 0$.  Let $x \in \mathcal{H}$. Let $i \in \{1,2\}$ and $j \in \{1,2\} \smallsetminus \{i\}$.   Then $x \in W_{i}$ if and only if $P_{H_{j}}x \in W_{i}$.
\end{lemma}
\begin{proof}
 Combine \cref{eq:W12} with   $\innp{u_{1}, u_{2}} = 0$ and \cref{fact:Projec:Hyperplane} to obtain that $x \in W_{i}  \Leftrightarrow \innp{x, u_{i}} \leq \eta_{i} 
	\Leftrightarrow  \Innp{x+ \frac{\eta_{j} - \innp{x,u_{j}}}{\norm{u_{j}}^{2}}u_{j}, u_{i}} \leq \eta_{i} 
	\Leftrightarrow \innp{\Pro_{H_{j}}x, u_{i}} \leq \eta_{i}  
	\Leftrightarrow 	\Pro_{H_{j}}x  \in W_{i}$.
\end{proof}

\begin{theorem} \label{theom:PW2PW1:BAM:0}
	Suppose   $\innp{u_{1}, u_{2}}=0$. Then $\Pro_{W_{2}}\Pro_{W_{1}} =\Pro_{W_{2} \cap W_{1}}$.
\end{theorem}
\begin{proof}
According to \cref{prop:u1u2geq0}, it suffices to show that $(\forall x \in W^{c}_{1} \cap W^{c}_{2})$ $\Pro_{W_{2}}\Pro_{W_{1}}x =\Pro_{W_{1} \cap W_{2}}x $. Let $x \in W^{c}_{1} \cap W^{c}_{2} $. Since $\innp{u_{1}, u_{2}}=0$, thus, by \cref{lemma:W1W2H1H2:SetCharacter}\cref{lemma:W1W2H1H2:SetCharacter:geq:C3}, $ W^{c}_{1} \cap W^{c}_{2}=C_{3}$. Hence,  \cref{lem:PW1W2H1H2:IV} leads to $\Pro_{W_{1} \cap W_{2}} x=\Pro_{H_{1} \cap H_{2}} x $. Moreover, by \cref{lem:PW1W2:BAM:u1u2eq0:equiva}, $x \notin W_{2}$ implies that $\Pro_{H_{1}}x \notin W_{2}$. So,   \cref{cor:PH2PH1}\cref{cor:PH2PH1:H2H1:=0} and $x \in W^{c}_{1} \cap W^{c}_{2}$ deduce that $\Pro_{W_{2}} \Pro_{W_{1}} x = \Pro_{W_{2}} \Pro_{H_{1}} x=\Pro_{H_{2}} \Pro_{H_{1}} x =\Pro_{H_{1} \cap H_{2}} x $.
Altogether,  $\Pro_{W_{2}}\Pro_{W_{1}}x  =\Pro_{H_{1} \cap H_{2}} x  =\Pro_{W_{1} \cap W_{2}}x $.
\end{proof}

We conclude the main results obtained in this section below.
\begin{theorem} \label{theorem:Halfspaces}
	Recall that $u_{1}$ and $u_{2}$ are in $\mathcal{H}$, that $\eta_{1}$ and $\eta_{2}$ are  in $\mathbb{R}$,  that $W_{1} := \{x \in \mathcal{H} ~:~ \innp{x,u_{1}} \leq \eta_{1} \}$, and $ W_{2} := \{x \in \mathcal{H} ~:~ \innp{x,u_{2}} \leq \eta_{2} \}$.
	Then  the following statements hold:
	\begin{enumerate}
		\item Suppose that $ u_{1}, u_{2} $ are linearly dependent and that $W_{1} \cap W_{2} \neq \varnothing$. Then $\Pro_{W_{2}}\Pro_{W_{1}} =\Pro_{W_{1} \cap W_{2}}$.  $($See \cref{theorem:W1capW2:u1u2LD}.$)$
		\item  Suppose that $ u_{1}, u_{2} $ are linearly independent.
		\begin{enumerate}
			\item If $\innp{u_{1},u_{2}} =0$, then $\Pro_{W_{2}}\Pro_{W_{1}} =\Pro_{W_{1} \cap W_{2}}$. $($See \cref{theom:PW2PW1:BAM:0}.$)$
			\item If $\innp{u_{1},u_{2}} <0$, then $\gamma :=\frac{|\innp{u_{1}, u_{2}}| }{\norm{u_{1}} \norm{u_{2}}} \in  \left[0,1\right[\,$, and $(\forall x \in \mathcal{H})$  $\norm{ (\Pro_{W_{2}}\Pro_{W_{1}})^{k}x - \Pro_{W_{1} \cap W_{2}}x} \leq \gamma^{k} \norm{x - \Pro_{W_{1} \cap W_{2}}x}$.  $($See \cref{thm:PROW1W2:BAM}.$)$
			\item   If $\innp{u_{1},u_{2}} >0$, then $(\forall x \in \mathcal{H})$ $\Pro_{W_{2}}\Pro_{W_{1}}x \in W_{1} \cap W_{2}$. In particular,  if $ x \in W_{1} \cup W_{2}$, then  $\Pro_{W_{2}}\Pro_{W_{1}}x =\Pro_{W_{1} \cap W_{2}}x $; if $ x \in W^{c}_{1} \cap W^{c}_{2}$ with $ \Pro_{H_{1}}x \in W_{2} $, then  $\Pro_{W_{2}}\Pro_{W_{1}}x=\Pro_{H_{1}}x  =\Pro_{W_{1} \cap W_{2}}x $; if $ x \in W^{c}_{1} \cap W^{c}_{2}$ with $ \Pro_{H_{1}}x \notin W_{2} $, then $\Pro_{W_{2}}\Pro_{W_{1}}x \in W_{1} \cap W_{2} \smallsetminus \{ \Pro_{W_{1} \cap W_{2}}x  \}$.  $($See \cref{prop:>0:inW1W2}.$)$ 
		\end{enumerate}
	\end{enumerate}
\end{theorem}

\section{Projection onto intersection of  hyperplane and halfspace} \label{sec:Projection:HyperpnaeHalfspace}

\subsection*{Karush-Kuhn-Tucker  conditions}
In order to deduce the formula of projection onto intersection of hyperplane and halfspace, we need the following   Karush-Kuhn-Tucker conditions  to characterize the optimal solution of  convex optimization with finitely many equality and inequality constraints in Hilbert spaces.

Before showing the main result \cref{them:KKT:EQINEQ:Hilbert} in this subsection, we see first  the following result.

\begin{lemma} \label{lemma:Riesz-FrechetRep}
	Let $F:\mathcal{H} \to \mathbb{R}$ be affine and continuous. Then there exists a unique vector $u \in \mathcal{H}$ such that 
	\begin{align*}
	(\forall x \in \mathcal{H}) \quad F(x)=\innp{x,u} + F(0).
	\end{align*}
\end{lemma}
\begin{proof}
Define 	$(\forall x \in \mathcal{H}) $  $T(x):=F(x)-F(0)$.  Let $x \in \mathcal{H}$, $y \in \mathcal{H}$, $\alpha \in \mathbb{R}$, and $\beta \in \mathbb{R}$. Because $F$ is affine, we have that  $ T(\alpha x+\beta y) = F(\alpha x+\beta y)-F(0)=\frac{1}{2} F(2\alpha x) +\frac{1}{2} F(2\beta y) -F(0)=\frac{1}{2} (F(2\alpha x) + F(0) ) +\frac{1}{2} (F(2\beta x) + F(0) )-2F(0)=F(\alpha x+(1-\alpha)0) +F(\beta y +(1-\beta) 0) -2F(0) =\alpha (F(x) -F(0)) +\beta (F(y) -F(0)) =\alpha T(x) +\beta T(y) $, which implies that $T$ is linear. 

Moreover, because $F$ is continuous implies that $T$ is continuous, we know that $T\in \mathcal{B} (\mathcal{H}, \mathbb{R})$.  Now, apply \cite[Fact~2.24]{BC2017} with $f=T$ to obtain that there exists a unique vector $u \in \mathcal{H}$ such that 	$(\forall x \in \mathcal{H}) $ $ T(x)=\innp{x,u}$, which, by the definition of $T$, yields that $	(\forall x \in \mathcal{H}) $ $ F(x)=\innp{x,u} + F(0)$.
\end{proof}

Because, 
differentiable function must be continuous,
by \cref{lemma:Riesz-FrechetRep}, the condition \enquote{$h_{j}$ is affine} in \cite[Page~244]{BV2004}  is equivalent to \enquote{$h_{j}(x) =\innp{u_{j}, x} - \eta_{j}$ where $u_{j}  \in \mathcal{H}$ and $\eta_{j}  \in \mathbb{R} $}.
Therefore, the following KKT conditions are generalization of the version shown in \cite[Page~244]{BV2004} from finite-dimensional spaces to Hilbert spaces and from differentiable functions to subdifferentiable functions.  Actually, writing the affine function $h_{j}$ as the form $h_{j}(x) =\innp{u_{j}, x} - \eta_{j}$ plays a critical role in the proof of \cref{them:KKT:EQINEQ:Hilbert}.

 The main idea of the proof of \cref{them:KKT:EQINEQ:Hilbert} is from \cite[Proposition~27.21]{BC2017} that characterizes the optimal solution of  convex optimization with inequality constraints in Hilbert spaces, but
 \cref{them:KKT:EQINEQ:Hilbert} is not a direct result from \cite[Proposition~27.21]{BC2017}, because if 
 there was equality constraint  $g(x)=0$ in \cite[Proposition~27.21]{BC2017}, then $(\lev_{< 0} g) \cap (\lev_{< 0} -g) =\varnothing$ implies that the Slater condition (27.50) in \cite[Proposition~27.21]{BC2017} fails.
 
\begin{theorem} \label{them:KKT:EQINEQ:Hilbert}
	Let $s$ and $t$ be in $\mathbb{N} \smallsetminus \{0\}$, set $\I:=\{1, \ldots, s\}$ and $\J :=\{1, \ldots, t \}$, and let $\bar{x} \in \mathcal{H}$.  Suppose  that   $f$ and $(g_{i})_{i \in \I}$ are functions in $\Gamma_{0} (\mathcal{H})$. Set $(\forall j \in \J)$ $(\forall x \in \mathcal{H})$ $h_{j}(x) :=\innp{u_{j}, x} - \eta_{j}$ where $u_{j}  \in \mathcal{H}$ and $\eta_{j}  \in \mathbb{R} $. Assume that
	\begin{subequations} \label{eq:them:KKT:EQINEQ:Hilbert:gi:hi}
		\begin{align}  
		& (\forall i \in \I) \quad \lev_{\leq 0} g_{i} \subseteq \inte \dom g_{i},  \label{eq:them:KKT:EQINEQ:Hilbert:gi} \\
		& \left( \cap_{i \in \I} \lev_{< 0} g_{i} \right) \cap \left( \cap_{j \in \J} \ker h_{j} \right) \neq \varnothing, \quad \text{and} \label{eq:them:KKT:EQINEQ:Hilbert:noempty}\\
		& 0 \in \sri \left(   \left( \cap_{i \in \I} \lev_{\leq 0} g_{i}\right) \cap  \left( \cap_{j \in \J} \ker h_{j} \right)   -\dom f \right). \label{eq:them:KKT:EQINEQ:Hilbert:gi:hi:sri}
		\end{align}
	\end{subequations}
 Consider the problem
	\begin{align} \label{eq:problem:KKT}
	\minimize ~&f(x)\\
	\subject ~ & g_{i}(x) \leq 0, ~i \in \I \nonumber\\
	& h_{j}(x) =0,~ j \in \J. \nonumber
	\end{align}
	Then $\bar{x}$ is a solution to \cref{eq:problem:KKT} if and only if 
	\begin{align} \label{eq:them:KKT:EQINEQ:Hilbert:EQ:conditions}
	(\exists (\bar{\lambda}_{i})_{i \in \I} \in \mathbb{R}^{s}_{+}) ~(\exists (\bar{v}_{i})_{i \in \I} \in \times_{i \in \I} \partial g_{i} (\bar{x}) )~(\exists (\bar{\beta}_{j})_{j \in \J} \in \mathbb{R}^{t}) 
	\begin{cases}
	&-\sum_{ i \in \I} \bar{\lambda}_{i} \bar{v}_{i} -\sum_{j \in \J} \bar{\beta}_{j}u_{j}  \in  \partial f (\bar{x}), \\
	&(\forall i \in \I)  \quad  g_{i}(\bar{x}) \leq 0, \bar{\lambda}_{i}  g_{i}(\bar{x}) = 0,\\
	&(\forall j \in \J)   \quad    h_{j}(\bar{x}) = 0,
	\end{cases}
	\end{align}
in which case $ (\bar{\lambda}_{i})_{i \in \I}  \times (\bar{\beta}_{j})_{j \in \J}$ are Lagrange multipliers associated with $\bar{x}$, and $\bar{x}$ solves the problem 
\begin{align} \label{eq:Problem:fgh}
\min_{x \in \mathcal{H}} f(x) + \sum_{ i \in \I} \bar{\lambda}_{i} g_{i} + \sum_{j \in \J} \bar{\beta}_{j} h_{j}.
\end{align}
	
	Moreover, if  $f$ and $(g_{i})_{i \in \I}$ are  G\^ateaux differentiable at $\bar{x}$, then \cref{eq:them:KKT:EQINEQ:Hilbert:EQ:conditions} becomes that 
		\begin{align}\label{eq:them:KKT:EQINEQ:Hilbert}
		\left( \exists (\bar{\lambda}_{i})_{i \in \I} \in \mathbb{R}^{s}_{+} \right) ~\left(\exists (\bar{\beta}_{j})_{j \in \J} \in \mathbb{R}^{t} \right)~
		\begin{cases}
		&\nabla f(\bar{x})  + \sum_{i \in \I} \bar{\lambda}_{i} \nabla g_{i}(\bar{x})    +\sum_{j \in \J} \bar{\beta}_{j} u_{j} =0 \\
		&(\forall i \in \I)  \quad  g_{i}(\bar{x}) \leq 0, \bar{\lambda}_{i}  g_{i}(\bar{x}) = 0, \text{ and}\\
		&(\forall j \in \J)   \quad    h_{j}(\bar{x}) = 0.
		\end{cases}
		\end{align}
\end{theorem}

\begin{proof}
	We split the proof into the following five steps.
	
	\textbf{Step~1:} By \cref{eq:them:KKT:EQINEQ:Hilbert:noempty}, $\left( \cap_{j \in \J} \ker h_{j} \right) \neq \varnothing$. Take $\bar{z} \in \cap_{j \in J} \ker h_{j}$. Then for every $j \in \J$, by definition of $h_{j}$,  
	\begin{align} \label{eq:them:KKT:EQINEQ:Hilbert:hjujx}
	h_{j}(x +\bar{z}) =\innp{u_{j}, x}.
	\end{align}
	 Define 
	 \begin{subequations} \label{eq:them:KKT:EQINEQ:Hilbert:T:tilde}
	 	\begin{align}
	 &(\forall x \in \mathcal{H}) \quad T(x) :=x +\bar{z}, \label{eq:them:KKT:EQINEQ:Hilbert:T}\\
	 	\tilde{f} :=f \circ T, \quad (\forall i \in \I) &~\tilde{g_{i}} :=g_{i} \circ T \quad \text{and} \quad (\forall j \in \J) ~ \tilde{h}_{j} := h_{j}  \circ T =\innp{u_{j}, \cdot}. \label{eq:them:KKT:EQINEQ:Hilberttilde}
	 	\end{align}
	 \end{subequations}
 Then  we have  that $\tilde{f}$ and  $(\tilde{g}_{i})_{i \in \I}$  are  functions in $\Gamma_{0} (\mathcal{H})$  with $\dom \tilde{f} =\dom f -\bar{z}$, and that
 \begin{align} \label{eq:partial:fgh:eq:partial}
 \partial f (\bar{x}) = \partial \tilde{f} (\bar{x} -\bar{z}), ~ (\forall i \in \I) ~\partial g_{i} (\bar{x}) = \partial \tilde{g}_{i} (\bar{x} -\bar{z}), \text{ and } (\forall j \in \J) ~\nabla h_{j} (\bar{x}) = \nabla \tilde{h}_{j} (\bar{x}-\bar{z}) =u_{j}.
 \end{align}
	Because $(\forall i \in \I)$ $\lev_{\leq 0} \tilde{g}_{i} = \lev_{\leq 0} g_{i} -\bar{z} $,  $\inte \dom \tilde{g}_{i} = \inte \dom g_{i} -\bar{z} $, $\cap_{i \in \I} \lev_{< 0} \tilde{g}_{i}  = \cap_{i \in \I} \lev_{< 0} g_{i} -\bar{z}$ and $ \cap_{j \in \J} \ker \tilde{h}_{j}  =\cap_{j \in \J} \ker h_{j}  -\bar{z}$, by \cref{eq:them:KKT:EQINEQ:Hilbert:gi:hi}, we have that
	\begin{subequations} \label{eq:them:KKT:EQINEQ:Hilbert:gi:hi:tilde}
		\begin{align}  
		& (\forall i \in \I) \quad \lev_{\leq 0} \tilde{g}_{i} \subseteq \inte \dom \tilde{g}_{i}, \label{eq:them:KKT:EQINEQ:Hilbert:gi:tilde} \\
		&\left( \cap_{i \in \I} \lev_{< 0} \tilde{g}_{i} \right) \cap \left( \cap_{j \in \J} \ker \tilde{h}_{j} \right) \neq \varnothing, \text{ and} \label{eq:them:KKT:EQINEQ:Hilbert:noempty:tilde}\\
		& 0 \in \sri \left(   \left( \cap_{i \in \I} \lev_{\leq 0} \tilde{g}_{i}\right) \cap  \left( \cap_{j \in \J} \ker \tilde{h}_{j} \right)   -\dom \tilde{f} \right). \label{eq:them:KKT:EQINEQ:Hilbert:sri:tilde}
		\end{align}
	\end{subequations}
	
	Let $y \in \mathcal{H}$. Substitute $x$ in \cref{eq:problem:KKT} by $y+\bar{z}$ to obtain that 
	\begin{align} \label{eq:problem:KKT:xbarz}
	\minimize ~&\tilde{f}(y)\\
	\subject ~ & \tilde{g}_{i}(y) \leq 0, ~i \in \I \nonumber\\
	& \tilde{h}_{j}(y) =0,~ j \in \J. \nonumber
	\end{align}
	Hence, 
	\begin{align} \label{eq:them:KKT:EQINEQ:Hilbert:EQ1}
	\bar{x}	\text{ is a solution to \cref{eq:problem:KKT}} \Leftrightarrow  \bar{x} -\bar{z}	\text{ is a solution to \cref{eq:problem:KKT:xbarz}}.
	\end{align}

	\textbf{Step~2:} In this part, we characterize the solutions to \cref{eq:problem:KKT:xbarz}. Set 
	\begin{align} \label{eq:them:KKT:EQINEQ:Hilbert:C}
	C:=( \cap_{i \in \I} W_{i} ) \cap ( \cap_{j \in \J} L_{j} ),
	\end{align}
	where $ (\forall i \in \I)$  $W_{i} :=\lev_{\leq 0} \tilde{g}_{i} = \{ y \in \mathcal{H} ~:~ \tilde{g}_{i}(y) \leq 0 \}  $ and $(\forall j \in \J)$  $L_{j} :=\ker \tilde{h}_{j} \stackrel{\cref{eq:them:KKT:EQINEQ:Hilberttilde}}{=} \ker u_{j}$.

	Recall that $C$ is a  closed and convex subset of $\mathcal{H}$ and $\tilde{f} \in \Gamma_{0} (\mathcal{H})$. Then \cref{eq:them:KKT:EQINEQ:Hilbert:sri:tilde} and \cite[Proposition~27.8]{BC2017} imply 
	\begin{align} \label{eq:them:KKT:EQINEQ:Hilbert:EQ2:partial}
	\bar{x} -\bar{z}	\text{ is a solution to \cref{eq:problem:KKT:xbarz}} \Leftrightarrow  \left(\exists w \in \N_{C}(\bar{x} -\bar{z})\right),  -w \in \partial \tilde{f}(\bar{x} -\bar{z}).
	\end{align}

	\textbf{Step~3:} Denote by $\bar{y} :=\bar{x} -\bar{z}$. We characterize the set $\N_{C}(\bar{y})$ in this part. As a consequence of \cite[Definition~6.38]{BC2017}, if $\bar{y} \notin C$, then $\N_{C}(\bar{y}) =\varnothing$. Suppose that $\bar{y} \in C $. Set 
	\begin{align*}
	\I_{+} :=\{i \in \I ~:~ \tilde{g}_{i}(\bar{y}) =0 \} \quad \text{and} \quad \I_{-} :=\{i \in \I ~:~ \tilde{g}_{i}(\bar{y}) < 0 \}.
	\end{align*}
	Because $\bar{y} \in C  \subseteq \cap_{i \in \I} W_{i}$,  we know that 
	\begin{align} \label{eq:them:KKT:EQINEQ:Hilbert:I}
	\I =\I_{+} \cup \I_{-}.
	\end{align} 
	Let $i \in \I$.  Then employing \cref{eq:them:KKT:EQINEQ:Hilbert:gi:tilde} in the last inclusion below, we know that 
	\begin{align} \label{eq:ybar:interdomtildegi}
	\bar{y} \in C \subseteq W_{i} = \lev_{\leq 0} \tilde{g}_{i} \subseteq \inte \dom \tilde{g}_{i}.
	\end{align}
Note that  $\tilde{g}_{i}$ is convex implies that $\dom \tilde{g}_{i}$ is convex. Then, according to \cite[Propostion~6.45]{BC2017}, \cref{eq:ybar:interdomtildegi} implies that 
	\begin{align}  \label{eq:them:KKT:EQINEQ:Hilbert:Ngi}
	(\forall i \in \I)	 \quad \N_{\dom \tilde{g}_{i}} \bar{y} =\{0\}.
	\end{align}	
	Hence, using \cite[Proposition~6.2(i)]{BC2017} in the second identity below, we have that 
	\begin{align} \label{eq:them:KKT:EQINEQ:Hilbert:Ntildegiy}
\N_{\dom \tilde{g}_{i}} \bar{y} \cup (\cone \partial \tilde{g}_{i}(\bar{y}) ) \stackrel{\cref{eq:them:KKT:EQINEQ:Hilbert:Ngi}}{=} \{0\} \cup (\cone \partial \tilde{g}_{i}(\bar{y}) ) = \{0\} \cup (\mathbb{R}_{++} \partial \tilde{g}_{i}(\bar{y}) )=\mathbb{R}_{+} \partial \tilde{g}_{i}(\bar{y}).
	\end{align}
	 Moreover, since
	\cref{eq:them:KKT:EQINEQ:Hilbert:noempty:tilde} implies that $(\forall i \in \I) $ $\lev_{< 0} \tilde{g}_{i} \neq \varnothing$, thus combining \cite[Lemma~27.20]{BC2017} with \cref{eq:ybar:interdomtildegi}, \cref{eq:them:KKT:EQINEQ:Hilbert:Ngi} and \cref{eq:them:KKT:EQINEQ:Hilbert:Ntildegiy}, we obtain that
	\begin{align}  \label{eq:them:KKT:EQINEQ:Hilbert:NWi}
	(\forall i \in \I) \quad  \N_{W_{i}} (\bar{y}) =\begin{cases}
	\mathbb{R}_{+} \partial \tilde{g}_{i}(\bar{y}), \quad & \text{if } i \in \I_{+}; \\
	\{0\}, \quad & \text{if } i \in \I_{-}.
	\end{cases}
	\end{align}
	Because $(\forall j \in \J)$ $L_{j} =\ker u_{j}$ is a linear subspace, $L_{j} -L_{j}  =\ker u_{j}$, $\bar{y} \in C \subseteq \cap_{j \in \J} L_{j} $, using   \cite[Example~6.43]{BC2017},
	 we have that
	\begin{align} \label{eq:them:KKT:EQINEQ:Hilbert:NHj}
	(\forall j \in \J) \quad \N_{L_{j}} (\bar{y}) =(L_{j} -L_{j} )^{\perp}=( \ker u_{j})^{\perp}= \spn \{ u_{j} \}.
	\end{align}
	Recall that  $\{  \tilde{g}_{i} \}_{i \in \I} \subseteq \Gamma_{0} (\mathcal{H})$. As a consequence of \cite[Corollary~8.39(ii)]{BC2017}, $	(\forall i \in \I) $   $\cont \tilde{g}_{i}=\inte \dom \tilde{g}_{i}$. Combine this with \cref{eq:them:KKT:EQINEQ:Hilbert:gi:tilde} to know that $(\forall i \in \I) $ $\tilde{g}_{i}$ is continuous on $ \lev_{< 0} \tilde{g}_{i} \subseteq \lev_{\leq 0} \tilde{g}_{i} \subseteq \inte \dom \tilde{g}_{i}$. Moreover, because \cref{eq:them:KKT:EQINEQ:Hilbert:noempty:tilde} implies that $(\forall i \in \I)$ $\lev_{<0} \tilde{g}_{i} \neq \varnothing$, by \cite[Corollary~8.47(i)]{BC2017},  $(\forall i \in \I)$ $\inte W_{i}= \lev_{<0} \tilde{g}_{i} $, which yields that  
	\begin{align} \label{eq:capLjcapintWi:neqnothing}
	(\cap_{j \in \J} L_{j}) \cap (\cap_{i \in \I} \inte W_{i}) =(\cap_{j \in \J} L_{j} )\cap (\cap_{i \in \I}  \lev_{<0} \tilde{g}_{i} ) \stackrel{\cref{eq:them:KKT:EQINEQ:Hilbert:noempty:tilde}}{\neq} \varnothing. 
	\end{align}
	  Apply \cite[Example~16.50(iv)]{BC2017} with $m=s+1$, $(\forall i \in \I)$ $f_{i} = \iota_{W_{i}}$ and $f_{s+1} = \iota_{\cap^{t}_{j =1} L_{j}}$ to obtain that 
	\begin{align} \label{eq:them:KKT:EQINEQ:Hilbert:partialSum1}
	\partial \left(  \sum_{ i \in \I}   \iota_{W_{i} } + \iota_{ \cap_{j \in \J} L_{j} } \right) (\bar{y})   = \left( \big( \sum_{ i \in \I}  \partial \iota_{  W_{i} } \big) + \partial \iota_{ \cap_{j \in \J} L_{j} } \right) (\bar{y}) .
	\end{align}
	In addition, note that $(\forall j \in \J)$ $\sum^{j}_{k=1} L^{\perp}_{k} =\sum^{j}_{k=1} (\ker u_{k})^{\perp} = \sum^{j}_{k=1}  \spn \{u_{k} \} $ is a closed finite-dimensional linear subspace of $\mathcal{H}$. Taking  \cite[Lemma~4.3(ii)]{BOyW2020BAM} into account, we see that
	\begin{align} \label{eq:j:cap:closed:L}
	(\forall j \in \J \smallsetminus \{1\}) \quad  L_{j} + \cap^{j-1}_{k=1} L_{k} \text{  is closed}.
	\end{align} 
	Notice that $(\forall j \in \J)$ $\dom \iota_{ L_{j} } = L_{j} =\ker u_{j}$ is a linear subspace. Use   \cref{eq:j:cap:closed:L}, and 
 apply  \cite[Example~16.50(iii)]{BC2017} with $m=t$, $( \forall k \in \{1, \ldots, m \} )$ $f_{k} =\iota_{L_{k}}$ to obtain 
	\begin{align} \label{eq:them:KKT:EQINEQ:Hilbert:partialSum:iota:Lk}
	\partial ( \sum_{j \in \J} \iota_{ L_{j} } ) (\bar{y})  = \sum_{j \in \J} \partial \iota_{ L_{j} }  (\bar{y}). 
	\end{align}  
	Now 
	\begin{align}
	\N_{C} (\bar{y} )&~=~ \partial \iota_{C} (\bar{y})  \quad \text{(by \cite[Example~16.13]{BC2017}  )}  \nonumber\\
	&\stackrel{\cref{eq:them:KKT:EQINEQ:Hilbert:C}}{=} \partial \iota_{( \cap_{i \in \I} W_{i} ) \cap ( \cap_{j \in \J} L_{j} )} (\bar{y}) \nonumber\\
	&~=~ \partial \left( ( \sum_{i \in \I}\iota_{ W_{i} } )+ \iota_{ \cap_{j \in \J} L_{j} } \right) (\bar{y})  \quad \text{(by definition of indicator function)}\nonumber\\
	& \stackrel{\cref{eq:them:KKT:EQINEQ:Hilbert:partialSum1}}{=} \left( \big(\sum_{i \in \I} \partial \iota_{  W_{i} }  \big)+ \partial \iota_{ \cap_{j \in \J} L_{j} } \right) (\bar{y}) \nonumber\\
	& \stackrel{\cref{eq:them:KKT:EQINEQ:Hilbert:I}}{=}  \sum_{i \in \I_{+}} \partial \iota_{  W_{i} }(\bar{y})  +\sum_{i \in \I_{-}} \partial \iota_{  W_{i} }(\bar{y})  + \partial \iota_{ \cap_{j \in \J} L_{j} } (\bar{y}) \nonumber\\
	& ~=~  \sum_{i \in \I_{+}} \partial \iota_{  W_{i} }(\bar{y})  +\sum_{i \in \I_{-}} \partial \iota_{  W_{i} }(\bar{y})  + \partial ( \sum_{j \in \J} \iota_{ L_{j} } ) (\bar{y}) \quad \text{(by definition of indicator function)} \nonumber\\
	& \stackrel{\cref{eq:them:KKT:EQINEQ:Hilbert:partialSum:iota:Lk}}{=}  \sum_{i \in \I_{+}} \partial \iota_{ W_{i} }(\bar{y})  +\sum_{i \in \I_{-}} \partial \iota_{  W_{i} }(\bar{y})  + \sum_{j \in \J} \partial \iota_{ L_{j} }  (\bar{y})  \nonumber\\
	&~=~  \sum_{i \in \I_{+}}  \N_{ W_{i} }(\bar{y})  +\sum_{i \in \I_{-}} \N_{  W_{i} }(\bar{y})  + \sum_{j \in \J}  \N_{ L_{j} }  (\bar{y})   \quad \text{(by \cite[Example~16.13]{BC2017} )} \nonumber\\
	&~ =~ \sum_{i \in \I_{+}}  \mathbb{R}_{+} \partial \tilde{g}_{i}(\bar{y}) + \sum_{j \in \J}  \spn \{ u_{j} \},  \label{eq:NC=sum}  
	\end{align}
	where the last identity follows from \cref{eq:them:KKT:EQINEQ:Hilbert:NWi} and \cref{eq:them:KKT:EQINEQ:Hilbert:NHj}.
	Combine \cref{eq:NC=sum} with \cref{eq:them:KKT:EQINEQ:Hilbert:EQ2:partial} to obtain that $\bar{y} $ is a solution to \cref{eq:problem:KKT:xbarz}  if and only if
	\begin{align} \label{eq:KKTconditions:tilde}
	(\exists (\bar{\lambda}_{i})_{i \in \I} \in \mathbb{R}^{s}_{+}) ~(\exists(\bar{v}_{i})_{i \in \I} \in \times_{i \in \I} \partial \tilde{g}_{i} (\bar{y}) )~(\exists (\bar{\beta}_{j})_{j \in \J} \in \mathbb{R}^{t}) 
	\begin{cases}
	&-\sum_{ i \in \I} \bar{\lambda}_{i} \bar{v}_{i} -\sum_{j \in \J} \bar{\beta}_{j}u_{j}  \in  \partial \tilde{f} (\bar{y}) \\
	&(\forall i \in \I)  \quad  \tilde{g}_{i}(\bar{y}) \leq 0, \bar{\lambda}_{i}  \tilde{g}_{i}(\bar{y}) = 0,\\
	&(\forall j \in \J)   \quad    \tilde{h}_{j}(\bar{y}) = 0,
	\end{cases}
	\end{align}
	Note that  we have  only $\I_{+}$  in \cref{eq:NC=sum}, but   $\I$ in \cref{eq:KKTconditions:tilde}.  In fact, 
	for every $ i \in \I $, if $i \in \I_{+}$, then by definition of  $\I_{+}$, $\tilde{g}_{i}(\bar{y})=0$, so $\bar{\lambda}_{i} \tilde{g}_{i} (\bar{y}) =0$. Otherwise, if $i \in \I_{-}$, then to satisfy \cref{eq:NC=sum} and $-\sum_{ i \in \I}\lambda_{i} v_{i} -\sum_{j \in \J} \beta_{j}u_{j}  \in  \partial \tilde{f} (\bar{y})$, we set  $\lambda_{i} =0$, which implies $\lambda_{i} \tilde{g}_{i} (\bar{y}) =0$ as well. 
	
	Using	\cref{eq:them:KKT:EQINEQ:Hilbert:T:tilde} and \cref{eq:partial:fgh:eq:partial}, we know that \cref{eq:KKTconditions:tilde} is equivalent to the desired \cref{eq:them:KKT:EQINEQ:Hilbert:EQ:conditions}.

	Therefore, by   \cref{eq:them:KKT:EQINEQ:Hilbert:EQ1},  $\bar{x} $ is a solution to \cref{eq:problem:KKT} if and only if \cref{eq:them:KKT:EQINEQ:Hilbert:EQ:conditions} holds.

	\textbf{Step~4:} Suppose that $\bar{x}$ is a solution of \cref{eq:problem:KKT}. Then by  Step~3 above,  \cref{eq:them:KKT:EQINEQ:Hilbert:EQ:conditions} holds. Now, for the 
	$  (\bar{\lambda}_{i})_{i \in \I} \in \mathbb{R}^{s}_{+}$,  $ (\bar{v}_{i})_{i \in \I} \in \times_{i \in \I} \partial g_{i} (\bar{x})$, and $  (\bar{\beta}_{j})_{j \in \J} \in \mathbb{R}^{t}$ in  \cref{eq:them:KKT:EQINEQ:Hilbert:EQ:conditions},  combine \cref{eq:them:KKT:EQINEQ:Hilbert:EQ:conditions} and the definition of subdifferential to obtain that 
	\begin{align*}
0 \in \partial f (\bar{x})+\sum_{ i \in \I} \bar{\lambda}_{i} \bar{v}_{i} + \sum_{j \in \J} \bar{\beta}_{j}u_{j} \subseteq \partial f (\bar{x})+\sum_{ i \in \I} \bar{\lambda}_{i} \partial g_{i} (\bar{x})+ \sum_{j \in \J} \bar{\beta}_{j} \partial h_{j} (\bar{x}) \subseteq \partial \left( f +  \sum_{ i \in \I} \bar{\lambda}_{i}  g_{i} + \sum_{j \in \J} \bar{\beta}_{j} h_{j} \right) (\bar{x}),
	\end{align*}
	which, by \cite[Theorem~16.3]{BC2017}, implies that $\bar{x}$ solves \cref{eq:Problem:fgh}. 
	
	In addition, because it is clear that $\left( \mathbb{R}^{s}_{-} \times \{0\}^{t} \right) ^{\ominus} =\mathbb{R}^{s}_{+} \times \mathbb{R}^{t}$, using  \cref{eq:them:KKT:EQINEQ:Hilbert:EQ:conditions} and \cite[Remark~19.26]{BC2017} and applying \cite[19.25(v)]{BC2017} with $\mathcal{K}=\mathbb{R}^{s+t}$, $K=\mathbb{R}^{s}_{-} \times \{0\}^{t}$, and $R: x  \mapsto (g_{i}(x))_{i \in \I} \times (h_{j}(x))_{j \in \J}$, we obtain that $ (\bar{\lambda}_{i})_{i \in \I}  \times (\bar{\beta}_{j})_{j \in \J}$ are Lagrange multipliers associated with $\bar{x}$.
	
	\textbf{Step~5:}  Suppose that  $f$ and $(g_{i})_{i \in \I}$ are  G\^ateaux differentiable at $\bar{x}$.	Then, using the assumptions that  $f$ and $(g_{i})_{i \in \I}$ are functions in $\Gamma_{0} (\mathcal{H})$, and  \cite[Proposition~17.31(i)]{BC2017}, we obtain that 
	\begin{align}\label{eq:partial:gi:nabla}
\partial f (\bar{x}) =\{ \nabla f (\bar{x}) \} 	\text{ and } (\forall i \in \I) ~\partial g_{i}(\bar{x}) =\{ \nabla g_{i} (\bar{x}) \}.
	\end{align}
	
	Therefore, \cref{eq:them:KKT:EQINEQ:Hilbert:EQ:conditions} is equivalent to \cref{eq:them:KKT:EQINEQ:Hilbert} under the G\^ateaux differentiable assumptions.
		
	Altogether, the proof is complete.
\end{proof}

\begin{remark}  \label{rem:KKT}
	\begin{enumerate}
		\item \label{rem:KKT:fulldomain}	For special cases satisfying the condition \cref{eq:them:KKT:EQINEQ:Hilbert:gi:hi:sri} in \cref{them:KKT:EQINEQ:Hilbert}, we refer the interested readers to \cite[Proposition~6.19]{BC2017}. In particular, if $\dom f =\mathcal{H}$, then applying \cite[Proposition~6.19(vii)]{BC2017} with $L=\Id$, $D= \left( \cap_{i \in \I} \lev_{\leq 0} g_{i}\right) \cap  \left( \cap_{j \in \J} \ker h_{j} \right)  $, and $C=\dom f=\mathcal{H}$, we know that \cref{eq:them:KKT:EQINEQ:Hilbert:noempty} implies \cref{eq:them:KKT:EQINEQ:Hilbert:gi:hi:sri}. Hence, if $\dom f =\mathcal{H}$, \cref{eq:them:KKT:EQINEQ:Hilbert:gi:hi:sri} is unnecessary. 
		\item If  $J=\varnothing$ in \cref{them:KKT:EQINEQ:Hilbert}, then as it is shown in the proof of \cite[Proposition~27.21]{BC2017}, in view of  \cite[Propositions~17.50 and 6.19(vii)]{BC2017}, the conditions (27.50) in \cite[Proposition~27.21]{BC2017} imply our conditions  \cref{eq:them:KKT:EQINEQ:Hilbert:gi:hi}. Hence,  we know that if $J=\varnothing$,  then \cref{them:KKT:EQINEQ:Hilbert} reduces to  \cite[Proposition~27.21]{BC2017}.
	\end{enumerate}
\end{remark}

\subsection*{Explicit formula of projection onto intersection of hyperplane and halfspace}
Recall that $u_{1}$ and $u_{2}$ are in $\mathcal{H}$ and $\eta_{1}$ and $\eta_{2}$ are in $\mathbb{R}$, and that $H_{1} := \{x \in \mathcal{H} ~:~ \innp{x,u_{1}}= \eta_{1} \}$, and  $W_{2} := \{x \in \mathcal{H} ~:~ \innp{x,u_{2}} \leq \eta_{2} \}$, $H_{2} := \{x \in \mathcal{H} ~:~ \innp{x,u_{2}} = \eta_{2} \}$.

In this subsection, we  shall provide an explicit formula for $\Pro_{H_{1} \cap W_{2}}$.

We can also prove the following result by using \cite[Theorem~3.1]{DDHT2020}.
\begin{theorem} \label{theor:u1u2LD:PBAM}
	Suppose that $u_{1}$ and $u_{2}$ are linearly dependent and  that $H_{1} \cap W_{2} \neq \varnothing$. Then   the following hold:
	\begin{enumerate}
		\item \label{item:theor:u1u2LD:PBAM:u1=0} Assume that $u_{1} =0$. Then $H_{1} =\mathcal{H}$, $ H_{1} \cap W_{2} =W_{2}$ and $\Pro_{ H_{1} \cap W_{2}  } =\Pro_{ W_{2}}$.
		\item \label{item:theor:u1u2LD:PBAM:u1neq0} Assume that $u_{1} \neq 0$. Then $H_{1} \cap W_{2} =H_{1}$ and $\Pro_{ H_{1} \cap W_{2}  } =\Pro_{ H_{1}}$.
	\end{enumerate}
\end{theorem}

\begin{proof}
	\cref{item:theor:u1u2LD:PBAM:u1=0}: Clearly, $u_{1} =0$ and  $H_{1} \cap W_{2} \neq \varnothing$ imply that $H_{1}  \neq \varnothing$,  $\eta_{1} =0$ and $H_{1}=\mathcal{H}$. Hence, \cref{item:theor:u1u2LD:PBAM:u1=0} holds.

	\cref{item:theor:u1u2LD:PBAM:u1neq0}:  Using \cref{lem:basic:u1u2:LinerDep}\cref{item:lem:basic:u1u2:LinerDep:LD:u1neq0}, $u_{1} \neq 0$, and the linear dependence of  $u_{1}$ and $u_{2}$, we have that 
	\begin{align} \label{eq:theor:u1u2LD:BAM:case2:u1u2}
	\text{either } u_{2} = \frac{\norm{u_{2}}}{\norm{u_{1}}} u_{1} \quad  \text{ or } \quad u_{2} = -\frac{\norm{u_{2}}}{\norm{u_{1}}} u_{1}.
	\end{align}  
	
Recall that $H_{1} \cap W_{2} \neq \varnothing$. Take $\bar{z} \in H_{1} \cap W_{2}$, that is, 
		\begin{align}\label{eq:theor:u1u2LD:BAM:case2}
		\innp{\bar{z}, u_{1}} =\eta_{1},  \quad \text{and} \quad 
		\innp{\bar{z}, u_{2}} \leq \eta_{2}. 
		\end{align}
Then $(\forall x \in H_{1})$  $\innp{x- \bar{z} ,u_{1}}=\eta_{1}-\eta_{1}=0$. Moreover, for every $x \in H_{1}$,
	\begin{align*}
	\innp{x,u_{2}} -\eta_{2} = \innp{x- \bar{z} ,u_{2}}  +\innp{\bar{z}, u_{2}}  -\eta_{2}  \stackrel{\cref{eq:theor:u1u2LD:BAM:case2:u1u2} }{=} \pm \frac{\norm{u_{2}}}{\norm{u_{1}}} \innp{x- \bar{z} ,u_{1}} +\innp{\bar{z}, u_{2}}  -\eta_{2} 
	 =  \innp{\bar{z}, u_{2}}  -\eta_{2} 
	  \stackrel{\cref{eq:theor:u1u2LD:BAM:case2}}{\leq} 0, 
	\end{align*}
	which implies that $x \in W_{2} $. So, 
	$
	H_{1} \subseteq W_{2}.
	$
	Hence, $	H_{1} \cap W_{2} =H_{1} $ and $\Pro_{H_{1} \cap W_{2} }=\Pro_{H_{1}}$.
\end{proof}
The following easy result is  necessary to prove the  \cref{them:Formula:H1capW2} below.
\begin{lemma} \label{lem:H1interW2:nonempty}
	Suppose that $u_{1}$ and $u_{2}$ are linearly independent.
	Then $H_{1} \cap H_{2} \neq \varnothing$ and $H_{1} \cap \inte W_{2} \neq \varnothing$.
\end{lemma}

\begin{proof}
$H_{1} \cap H_{2} \neq \varnothing$ follows from \cref{lem:capMi}. 
	Because $(\forall i \in \{1,2\})$ $(\spn \{u_{i}\} )^{\perp} = \ker u_{i}$,  and  $u_{1}$ and $u_{2}$ are linearly independent, we know that $  \spn \{u_{2}\}  \not\subseteq  \spn \{u_{1}\}$, and so $ \ker u_{1}   =  (\spn \{u_{1}\} )^{\perp}  \not \subseteq  (\spn \{u_{2}\} )^{\perp}  =\ker u_{2} $.  Take $\bar{z} \in H_{1} \cap H_{2}$, and  $ \bar{y} \in \ker u_{1} \smallsetminus \ker u_{2}$.  Now,
	\begin{align*}
	&\Innp{ \bar{z} - \innp{\bar{y}, u_{2}} \bar{y}, u_{1} } = \innp{\bar{z}, u_{1}} -\innp{\bar{y}, u_{2}} \innp{\bar{y}, u_{1}} =\eta_{1},  \\
	& \Innp{ \bar{z} - \innp{\bar{y}, u_{2}} \bar{y}, u_{2} } = \innp{\bar{z}, u_{2}} -\innp{\bar{y}, u_{2}}^{2}  = \eta_{1}  -\innp{\bar{y}, u_{2}}^{2}  < \eta_{2},
	\end{align*}
	which  imply  that $\bar{z} - \innp{\bar{y}, u_{2}} \bar{y} \in H_{1} \cap \inte W_{2} \neq \varnothing $.
\end{proof}

The main idea of   the following result is from \cite[Proposition~29.23]{BC2017}, but this time we use the KKT conditions proved in \cref{them:KKT:EQINEQ:Hilbert}. 
\begin{theorem} \label{them:Formula:H1capW2}
	Suppose that  $u_{1}$ and $u_{2}$ are linearly independent.
	Let $x \in \mathcal{H}$. 
	Then $H_{1} \cap W_{2} \neq \varnothing$ and 
	\begin{align*}
	\Pro_{H_{1} \cap W_{2}} x =x - \xi_{1} u_{1} -\xi_{2} u_{2},
	\end{align*}
	where exactly one of the following holds:
	\begin{enumerate}
		\item \label{item:them:Formula:H1capW2:>0} $(\innp{x,u_{2}} - \eta_{2}) \norm{u_{1}}^{2} - (\innp{x, u_{1}} -\eta_{1}) \innp{u_{1}, u_{2}} >0$.  Then $\xi_{1} = \frac{(\innp{x,u_{1}} - \eta_{1}) \norm{u_{2}}^{2} - (\innp{x, u_{2}} -\eta_{2}) \innp{u_{1}, u_{2}}}{\norm{u_{1}}^{2} \norm{u_{2}}^{2} - |\innp{u_{1}, u_{2}}|^{2} }$, and $\xi_{2} = \frac{(\innp{x,u_{2}} - \eta_{2}) \norm{u_{1}}^{2} - (\innp{x, u_{1}} -\eta_{1}) \innp{u_{1}, u_{2}} }{\norm{u_{1}}^{2} \norm{u_{2}}^{2} - |\innp{u_{1}, u_{2}}|^{2} }>0.$
		Moreover, $\Pro_{H_{1} \cap W_{2}}x = \Pro_{H_{1} \cap H_{2}}x	$.
		\item  \label{item:them:Formula:H1capW2:leq0}  $(\innp{x,u_{2}} - \eta_{2}) \norm{u_{1}}^{2} - (\innp{x, u_{1}} -\eta_{1}) \innp{u_{1}, u_{2}}  \leq 0$. Then  $	\xi_{1} =\frac{ \innp{x, u_{1}} -\eta_{1} }{\norm{u_{1}}^{2}}$ and $  \xi_{2} = 0.$
		Moreover, $\Pro_{H_{1} \cap W_{2}}x = \Pro_{H_{1}}x	$.
	\end{enumerate}
\end{theorem}

\begin{proof}
	By the definition of projection,  $\Pro_{H_{1} \cap W_{2}}x$ is the unique solution of the problem 
	\begin{align*}
	\minimize ~ & f(y) = \frac{1}{2} \norm{y -x}^{2}\\
	\subject ~ & g (y) =\innp{y, u_{2}} -\eta_{2} \leq 0\\
	& h(y) = \innp{y, u_{1}} - \eta_{1} =0.
	\end{align*}
	Now $\dom f=\mathcal{H}$, $\dom g =\mathcal{H}$, $f$ and $g$ are differentiable functions in $\Gamma_{0}(\mathcal{H})$, and 
	$ \lev_{\leq 0} g =W_{2} \subseteq \mathcal{H} = \inte \dom g $. By \cref{lem:H1interW2:nonempty},  $\lev_{<0} g \cap \ker h = \inte W_{2} \cap H_{1} \neq \varnothing$. 
	Moreover, because $\dom f =\mathcal{H}$, by \cref{rem:KKT}\cref{rem:KKT:fulldomain}, $ 0 \in \sri \left( (\lev_{\leq 0}g \cap \ker h)  -\dom f\right)$.
	Note that  $(\forall y \in \mathcal{H})$ $\nabla f(y) =y-x$, and $\nabla g (y)= u_{2}$.
	Hence, apply \cref{them:KKT:EQINEQ:Hilbert} with $s=1$, $t=1$, $f(y)=\frac{1}{2} \norm{y -x}^{2} $, $g_{1}=g$ and $h_{1}=h $ to obtain that there exist $\xi_{2} \in \mathbb{R}_{+}$ and $\xi_{1} \in \mathbb{R}$ such that
	\begin{subequations} \label{them:Formula:H1capW2:formu}
		\begin{align}
		&\Pro_{H_{1} \cap W_{2}}x -x + \xi_{2} u_{2} +\xi_{1} u_{1} =0, \label{them:Formula:H1capW2:formu:deriva}\\
		& \innp{u_{2} , \Pro_{H_{1} \cap W_{2}}x } - \eta_{2} \leq 0, \quad \xi_{2} (\innp{u_{2} , \Pro_{H_{1} \cap W_{2}}x } - \eta_{2}  ) =0, \label{them:Formula:H1capW2:formu:g}\\
		& \innp{u_{1} , \Pro_{H_{1} \cap W_{2}}x } - \eta_{1} = 0. \label{them:Formula:H1capW2:formu:h}
		\end{align}
	\end{subequations}
Hence,  we have that
	\begin{subequations}
		\begin{align}
		&	\Pro_{H_{1} \cap W_{2}}x = x -\xi_{1} u_{1}  - \xi_{2} u_{2}, \quad (\text{by \cref{them:Formula:H1capW2:formu:deriva}})  \label{eq:them:Formula:H1capW2:P}\\
		& \innp{u_{2} ,  x -\xi_{1} u_{1}  - \xi_{2} u_{2}  } - \eta_{2} \leq 0, \quad (\text{by \cref{eq:them:Formula:H1capW2:P} and \cref{them:Formula:H1capW2:formu:g}})  \label{eq:them:Formula:H1capW2:formu:g}\\
		&  \innp{u_{1} , x -\xi_{1} u_{1}  - \xi_{2} u_{2}  } - \eta_{1} = 0,  \quad (\text{by \cref{eq:them:Formula:H1capW2:P} and \cref{them:Formula:H1capW2:formu:h}}) \label{eq:them:Formula:H1capW2:formu:h}\\
		&  \xi_{2} (\innp{u_{2} ,x -\xi_{1} u_{1}  - \xi_{2} u_{2} } - \eta_{2}  ) =0.  \quad (\text{by \cref{eq:them:Formula:H1capW2:P} and \cref{them:Formula:H1capW2:formu:g}})  \label{eq:them:Formula:H1capW2:formu:xig}
		\end{align}
	\end{subequations}
Now, we have exactly the following two cases. 
	
	\textbf{Case~1:} $\xi_{2} >0$. Then   \cref{eq:them:Formula:H1capW2:formu:xig} implies 
	\begin{align} \label{eq:them:Formula:H1capW2:case1:W2}
	\innp{u_{2} ,x -\xi_{1} u_{1}  - \xi_{2} u_{2} } - \eta_{2}   =0.
	\end{align}
	Combine \cref{eq:them:Formula:H1capW2:case1:W2} with \cref{eq:them:Formula:H1capW2:formu:h} to obtain that
	\begin{align} \label{eq:u1u2Linearsystem}
	\begin{pmatrix}
	\innp{u_{1}, u_{1}} &  \innp{u_{1}, u_{2}} \\
	\innp{u_{2}, u_{1}} & \innp{u_{2}, u_{2}} 
	\end{pmatrix}
	\begin{pmatrix}
	\xi_{1}\\
	\xi_{2}
	\end{pmatrix}
	=  \begin{pmatrix}
	\innp{u_{1}, x} -	\eta_{1}\\
	\innp{u_{2}, x} -	\eta_{2}
	\end{pmatrix}.
	\end{align} 
\cref{fact:Gram:inver} and the linear independence of  $u_{1}$ and $u_{2}$  imply that the Gram matrix $G(u_{1}, u_{2})$ defined as \cref{eq:GramMatrix} is invertible. Solve the system \cref{eq:u1u2Linearsystem} of linear equations to obtain that 
	\begin{align} \label{eq:them:Formula:H1capW2:case1:xi}
	\xi_{1} = \frac{(\innp{x,u_{1}} - \eta_{1}) \norm{u_{2}}^{2} - (\innp{x, u_{2}} -\eta_{2}) \innp{u_{1}, u_{2}}}{\norm{u_{1}}^{2} \norm{u_{2}}^{2} - |\innp{u_{1}, u_{2}}|^{2} } \text{ and }\xi_{2} = \frac{(\innp{x,u_{2}} - \eta_{2}) \norm{u_{1}}^{2} - (\innp{x, u_{1}} -\eta_{1}) \innp{u_{1}, u_{2}} }{\norm{u_{1}}^{2} \norm{u_{2}}^{2} - |\innp{u_{1}, u_{2}}|^{2} }.
	\end{align}
	Use  $u_{1}$ and $u_{2}$ are linearly independent again and apply \cref{lem:basic:u1u2:LinerDep}\cref{item:lem:basic:u1u2:LinerDep:LID} to know that
	\begin{align*}
	\xi_{2} >0 \Leftrightarrow (\innp{x,u_{2}} - \eta_{2}) \norm{u_{1}}^{2} - (\innp{x, u_{1}} -\eta_{1}) \innp{u_{1}, u_{2}} >0.
	\end{align*}
	Hence, \cref{item:them:Formula:H1capW2:>0}  is exactly the Case 1. Combine   \cref{eq:them:Formula:H1capW2:P} with \cref{eq:them:Formula:H1capW2:case1:xi} to know that  
the first part of	\cref{item:them:Formula:H1capW2:>0}  is true. 

In addition, 
\begin{align*}
\innp{\Pro_{H_{1} \cap W_{2}}x, u_{2}} -\eta_{2} 
\stackrel{\cref{eq:them:Formula:H1capW2:P}}{=} &\innp{x - \xi_{1} u_{1} -\xi_{2} u_{2}, u_{2}} -\eta_{2}\\
~=~ & \innp{x, u_{2}} -\eta_{2} - \xi_{1} \innp{u_{1}, u_{2}} -\xi_{2} \norm{u_{2}}^{2}\\
\stackrel{\cref{eq:them:Formula:H1capW2:case1:xi}}{=} & ( \innp{x, u_{2}} -\eta_{2} ) - \frac{(\innp{x,u_{1}} - \eta_{1}) \norm{u_{2}}^{2} - (\innp{x, u_{2}} -\eta_{2}) \innp{u_{1}, u_{2}}}{\norm{u_{1}}^{2} \norm{u_{2}}^{2} - |\innp{u_{1}, u_{2}}|^{2} }\innp{u_{1}, u_{2}} \\
\quad \quad &
-  \frac{(\innp{x,u_{2}} - \eta_{2}) \norm{u_{1}}^{2} - (\innp{x, u_{1}} -\eta_{1}) \innp{u_{1}, u_{2}} }{\norm{u_{1}}^{2} \norm{u_{2}}^{2} - |\innp{u_{1}, u_{2}}|^{2} } \norm{u_{2}}^{2}\\
~=~ & ( \innp{x, u_{2}} -\eta_{2} ) -  \frac{(\innp{x,u_{1}} - \eta_{1}) (\norm{u_{2}}^{2} \innp{u_{1} ,u_{2}} -\innp{u_{1},u_{2}} \norm{u_{2}}^{2} )}{\norm{u_{1}}^{2} \norm{u_{2}}^{2} - |\innp{u_{1}, u_{2}}|^{2} } \\
 \quad \quad & -  \frac{(\innp{x,u_{2}} - \eta_{2}) (\norm{u_{1}}^{2} \norm{u_{2}}^{2} -\innp{u_{1}, u_{2}}^{2}) }{\norm{u_{1}}^{2} \norm{u_{2}}^{2} - |\innp{u_{1}, u_{2}}|^{2} } \\
~=~& (\innp{x,u_{2}} - \eta_{2}) - (\innp{x,u_{2}} - \eta_{2}) = 0,
\end{align*}
which implies that $\Pro_{H_{1} \cap W_{2}}x \in H_{2} $.  So, $ \Pro_{H_{1} \cap W_{2}}x \in H_{1} \cap H_{2}$. Hence, apply \cref{fact:AsubseteqB:Projection}  with $A=H_{1} \cap H_{2} $ and $B= H_{1} \cap W_{2}$ to obtain $ \Pro_{H_{1} \cap W_{2}}x = \Pro_{H_{1} \cap H_{2}}x $.
	
	\textbf{Case~2:}  $\xi_{2} =0$. Then by \cref{eq:them:Formula:H1capW2:formu:h}, we know that 
	\begin{align} \label{eq:them:Formula:H1capW2:case2:xi1}
	\xi_{1} = \frac{\innp{u_{1}, x} - \eta_{1}}{\norm{u_{1}}^{2}}.
	\end{align}
	Taking  \cref{eq:them:Formula:H1capW2:P} and \cref{fact:Projec:Hyperplane} into account, we know that
	\begin{align}   \label{eq:them:Formula:H1capW2:case2:eq:xi2=0}
	\xi_{2} =0 \Leftrightarrow \Pro_{H_{1} \cap W_{2}}x = x -\xi_{1} u_{1} =x + \frac{ \eta_{1} -\innp{u_{1}, x} }{\norm{u_{1}}^{2}} u_{1} =\Pro_{H_{1}} x.
	\end{align}
	Apply \cref{fact:AsubseteqB:Projection}  with $A= H_{1} \cap W_{2}$ and $B= H_{1}$ and use \cref{cor:PH2PH1}\cref{cor:PH2PH1:notin:EQ} to obtain that
		\begin{align}  \label{eq:them:Formula:H1capW2:case2:eq:PH1capW2}
		\Pro_{H_{1} \cap W_{2}}x =  \Pro_{H_{1}} x \Leftrightarrow \Pro_{H_{1}} x \in W_{2}  
		  \Leftrightarrow (\innp{x,u_{2}} - \eta_{2}) \norm{u_{1}}^{2} - (\innp{x, u_{1}} -\eta_{1}) \innp{u_{1}, u_{2}}  \leq 0.
		\end{align}
	Combine \cref{eq:them:Formula:H1capW2:case2:eq:xi2=0} with \cref{eq:them:Formula:H1capW2:case2:eq:PH1capW2} to obtain that
	\begin{align}  \label{eq:them:Formula:H1capW2:case2:equiva}
	\xi_{2} =0 \Leftrightarrow (\innp{x,u_{2}} - \eta_{2}) \norm{u_{1}}^{2} - (\innp{x, u_{1}} -\eta_{1}) \innp{u_{1}, u_{2}}  \leq 0.
	\end{align}
	Hence,  \cref{item:them:Formula:H1capW2:leq0}   is exactly the Case 2. 
	Moreover, \cref{eq:them:Formula:H1capW2:case2:xi1} and \cref{eq:them:Formula:H1capW2:case2:eq:xi2=0} deduce \cref{item:them:Formula:H1capW2:leq0}.	
	Altogether, the proof is complete.
\end{proof}

\begin{lemma} \label{lemma:PH1capW2:PH1H2}
	Suppose that $u_{1}$ and $u_{2}$ are linearly independent. Let $x \in \mathcal{H}$. 	 Then the following hold:
	\begin{enumerate}
		\item  \label{lemma:PH1capW2:PH1H2:P} If $ \Pro_{H_{1} }x  \notin W_{2}$, then $\Pro_{H_{1} \cap W_{2}}x = \Pro_{H_{1} \cap H_{2}}x$; otherwise, $\Pro_{H_{1} \cap W_{2}}x = \Pro_{H_{1} }x$.
		\item  \label{lemma:PH1capW2:PH1H2:notinW2} Assume that $ \Pro_{H_{1}} \Pro_{W_{2}}x \notin   W_{2}$. Then $ \Pro_{H_{1} \cap W_{2}}x=\Pro_{H_{1} \cap H_{2}}x$.
	\end{enumerate} 
\end{lemma}

\begin{proof}
	\cref{lemma:PH1capW2:PH1H2:P}: This is clear from \cref{them:Formula:H1capW2} and \cref{cor:PH2PH1}\cref{cor:PH2PH1:notin:EQ}.
	
	\cref{lemma:PH1capW2:PH1H2:notinW2}:
	By \cref{lemma:PH1capW2:PH1H2:P}, it suffices to show that $ \Pro_{H_{1} }x  \notin W_{2} $. 
	If $x \in W_{2}$, then $ \Pro_{H_{1}} x =\Pro_{H_{1}} \Pro_{W_{2}}x \notin  W_{2}$. 
	
	Suppose $x \notin W_{2}$. 
	Assume to the contrary that $ \Pro_{H_{1} }x \in W_{2} $. Then, by \cref{fact:Projec:Hyperplane}, 
	\begin{align*}
	0 \leq \eta_{2} - \innp{\Pro_{H_{1} }x, u_{2} } \Leftrightarrow 0 \leq \eta_{2} -\Innp{x +\frac{\eta_{1} -\innp{x,u_{1}}}{\norm{u_{1}}^{2}} u_{1}, u_{2}} 
	\Leftrightarrow (\eta_{1} -\innp{x,u_{1}}) \innp{u_{1},u_{2}}  \leq \norm{u_{1}}^{2} (\eta_{2} -\innp{x,u_{2}}),		
	\end{align*}
	which, by  \cref{cor:PH2PH1}\cref{cor:PH2PH1:H2H1} with swapping $H_{1}$ and $H_{2}$, and  by $x \notin W_{2}$ and the Cauchy-Schwarz inequality, implies  that
	\begin{align*}
	\eta_{2} - \innp{\Pro_{H_{1} }\Pro_{H_{2} }x, u_{2} } x & =- \frac{1}{ \norm{u_{1}}^{2}  \norm{u_{2}}^{2} } \left( (\eta_{1} -\innp{x,u_{1}}) \norm{u_{2}}^{2} - (\eta_{2} -\innp{x,u_{2}}) \innp{u_{1},u_{2}}  \right) \innp{u_{1},u_{2}}\\
	& \geq - \frac{1}{ \norm{u_{1}}^{2}  \norm{u_{2}}^{2} } \left( \norm{u_{1}}^{2}\norm{u_{2}}^{2} (\eta_{2} -\innp{x,u_{2}}) - (\eta_{2} -\innp{x,u_{2}}) \innp{u_{1},u_{2}}^{2} \right)\\
	&=  (\eta_{2} -\innp{x,u_{2}}) \left( \frac{\innp{u_{1},u_{2}}^{2} }{ \norm{u_{1}}^{2}  \norm{u_{2}}^{2} } -1\right) \geq 0, 
	\end{align*} 
which implies that $ \Pro_{H_{1} }\Pro_{H_{2} }x \in W_{2}$ and contradicts the assumption.  
\end{proof}

\section{Compositions of projections onto hyperplane and halfspace}
Similarly to the \cref{eq:fact:Dykstra:halfspaces},  given finitely many hyperplanes and halfspaces, by \cite[Theorem~9.24]{D2012},  the Boyle-Dykstra Theorem, we are able to use only the projections onto these individual  hyperplanes and halfspaces to generate the sequence according to  the Dykstra's algorithm for finding the
 projection onto the intersection of these hyperplanes and halfspaces. 
 
 In this section, for simplicity, we consider only one hyperplane and one halfspace.  We shall systematically investigate the relations of $ \Pro_{ H_{1} \cap W_{2}  }$ and $\Pro_{W_{2}}\Pro_{H_{1}}$, and of  $ \Pro_{ H_{1} \cap W_{2}  }$ and $\Pro_{H_{1}}\Pro_{W_{2}}$. 
 \label{sec:com:hyperplane:halfspace}

Recall that $u_{1}$ and $u_{2}$ are in $\mathcal{H}$, that  $\eta_{1}$ and $\eta_{2}$ are in $\mathbb{R}$, and that  $H_{1} := \{x \in \mathcal{H} ~:~ \innp{x,u_{1}}= \eta_{1} \}$,  $W_{2} := \{x \in \mathcal{H} ~:~ \innp{x,u_{2}} \leq \eta_{2} \}$,   and $H_{2} := \{x \in \mathcal{H} ~:~ \innp{x,u_{2}} = \eta_{2} \}$.

\subsection*{$u_{1}$ and $u_{2}$ are linearly dependent}

\begin{theorem} \label{theor:u1u2LD:BAM}
	Suppose that $u_{1}$ and $u_{2}$ are linearly dependent and  that $H_{1} \cap W_{2} \neq \varnothing$. Then $\Pro_{W_{2}}\Pro_{H_{1}} =\Pro_{H_{1} \cap W_{2}} =\Pro_{H_{1}} \Pro_{W_{2}}$.  In particular, the following statements hold:
	\begin{enumerate}
		\item \label{item:theor:u1u2LD:BAM:u1=0} Assume that $u_{1} =0$. Then $\Pro_{W_{2}}\Pro_{H_{1}} =\Pro_{ W_{2}} =\Pro_{H_{1}} \Pro_{W_{2}}$.  
		\item \label{item:theor:u1u2LD:BAM:u1neq0} Assume that $u_{1} \neq 0$. Then $\Pro_{W_{2}}\Pro_{H_{1}} =\Pro_{H_{1} } =\Pro_{H_{1}} \Pro_{W_{2}}$. 
	\end{enumerate}
\end{theorem}

\begin{proof}
	\cref{item:theor:u1u2LD:BAM:u1=0}:  According to \cref{theor:u1u2LD:PBAM}\cref{item:theor:u1u2LD:PBAM:u1=0}, $H_{1}=\mathcal{H}$ and $H_{1} \cap W_{2} =W_{2}$. Then clearly 	\cref{item:theor:u1u2LD:BAM:u1=0} holds.

	\cref{item:theor:u1u2LD:BAM:u1neq0}: In view of \cref{theor:u1u2LD:PBAM}\cref{item:theor:u1u2LD:PBAM:u1neq0}, $H_{1} \cap W_{2} =H_{1}$. Then  clearly  $	\Pro_{W_{2}}\Pro_{H_{1}}=\Pro_{H_{1}} =\Pro_{H_{1} \cap W_{2}}$.
 	
	On the other hand, if $u_{2}=0$, then because $H_{1} \cap W_{2} \neq \varnothing $ implies that $W_{2} \neq \varnothing$, we have that $\eta_{2} \geq 0$ and $W_{2}=\mathcal{H}$. So, 	$\Pro_{H_{1}}\Pro_{W_{2}}=\Pro_{H_{1}}\Pro_{\mathcal{H}}=\Pro_{H_{1}}$.
	Suppose that $u_{2} \neq 0$. Let $x \in \mathcal{H}$. If $x \in W_{2}$, then $\Pro_{H_{1}}\Pro_{W_{2}}x=\Pro_{H_{1}}x$. Assume that $x \notin W_{2}$. Then use \cref{remark:FactWFactH} and apply \cref{cor:PH2PH1}\cref{cor:PH2PH1:H2} with swapping $H_{1}$ and $H_{2}$ to  yield $\Pro_{H_{1}}\Pro_{W_{2}}x   =\Pro_{H_{1}}\Pro_{H_{2}}x =\Pro_{H_{1}}x$.
	Hence, $ \Pro_{H_{1}}\Pro_{W_{2}}=\Pro_{H_{1}}$.
	
	Altogether, we have that  $\Pro_{W_{2}}\Pro_{H_{1}}= \Pro_{H_{1} \cap W_{2}} =\Pro_{H_{1}} =\Pro_{H_{1}}\Pro_{W_{2}}$. 
\end{proof}

\subsection*{$u_{1}$ and $u_{2}$ are linearly independent}
In the whole subsection, we set
\begin{align} \label{eq:C}
C:=\left\{ x \in \mathcal{H} ~:~ (\innp{x,u_{2}} - \eta_{2}) \norm{u_{1}}^{2} - (\innp{x, u_{1}} -\eta_{1}) \innp{u_{1}, u_{2}} >0 \right\}.
\end{align}

\begin{proposition} \label{prop:Candu1u20}
	Suppose that  $u_{1}$ and $u_{2}$ are linearly independent. Then the following hold:
	\begin{enumerate}
		\item \label{prop:Candu1u20:C} Let $x \in \mathcal{H}$. Then $x \in C$ if and only if $\Pro_{H_{1}}x \notin W_{2}$.
		\item \label{prop:Candu1u20:u1u2} Suppose that $\innp{u_{1},u_{2}}=0$. Then $C =W^{c}_{2}$ and $\Pro_{W_{2}}\Pro_{H_{1}} =\Pro_{H_{1} \cap W_{2}} =\Pro_{H_{1}} \Pro_{W_{2}}$.
		\item \label{prop:Candu1u20:C:completment} Let $x \in C^{c}$. Then  $ \Pro_{W_{2}}\Pro_{H_{1}} x=\Pro_{H_{1}}x=\Pro_{H_{1} \cap W_{2}} x$. 
		\item \label{prop:Candu1u20:C:C} Suppose that $\innp{u_{1},u_{2}}\neq 0$.  Let $x \in C$. Then $\Pro_{W_{2}}\Pro_{H_{1}} x=\Pro_{H_{2}}\Pro_{H_{1}} x \in C$. Moreover, $\Pro_{H_{1} \cap W_{2}}\Pro_{W_{2}}\Pro_{H_{1}}x  =\Pro_{H_{1} \cap H_{2}}x=\Pro_{H_{1} \cap W_{2}}x$.
	\end{enumerate}
\end{proposition}

\begin{proof}
\cref{prop:Candu1u20:C}: This is clear from \cref{cor:PH2PH1}\cref{cor:PH2PH1:notin:EQ} and \cref{eq:C}.

\cref{prop:Candu1u20:u1u2}: It is easy to see that $C=W^{c}_{2}$ follows from $\innp{u_{1},u_{2} }=0$,  \cref{eq:C} and the definition of $W_{2}$. Let $x \in \mathcal{H}$. 
Then we have exactly the following two cases:

 Case~1: $x \in C$. Then   \cref{prop:Candu1u20:C} and \cref{remark:FactWFactH} imply $ \Pro_{W_{2}}\Pro_{H_{1}}x =\Pro_{H_{2}}\Pro_{H_{1}}x $. Hence, by $\innp{u_{1},u_{2}}=0$, \cref{them:Formula:H1capW2}\cref{item:them:Formula:H1capW2:>0} and \cref{cor:PH2PH1}\cref{cor:PH2PH1:H2H1:=0},
\begin{align*}
\Pro_{H_{1} \cap W_{2}} x 
=  x + \frac{ \eta_{1} -\innp{u_{1}, x} }{\norm{u_{1}}^{2}} u_{1} +\frac{ \eta_{2} -\innp{u_{2}, x} }{\norm{u_{2}}^{2}} u_{2} =\Pro_{H_{2}}\Pro_{H_{1}}x=\Pro_{W_{2}}\Pro_{H_{1}}x.
\end{align*}
Recall that $x \in C= W^{c}_{2}$. Use \cref{remark:FactWFactH} and \cref{them:Formula:H1capW2}\cref{item:them:Formula:H1capW2:>0}, and apply \cref{cor:PH2PH1}\cref{cor:PH2PH1:H2H1:=0} with swapping $H_{1}$ and $H_{2}$     to obtain that 
\begin{align*}
\Pro_{H_{1}}\Pro_{W_{2}}x= \Pro_{H_{1}}\Pro_{H_{2}}x=x + \frac{ \eta_{1} -\innp{u_{1}, x} }{\norm{u_{1}}^{2}} u_{1} +\frac{ \eta_{2} -\innp{u_{2}, x} }{\norm{u_{2}}^{2}} u_{2}= \Pro_{H_{1} \cap W_{2}} x.
\end{align*}

Case~2: $x \in C^{c}$. Then by \cref{them:Formula:H1capW2}\cref{item:them:Formula:H1capW2:leq0} and \cref{lemma:PH1capW2:PH1H2}\cref{lemma:PH1capW2:PH1H2:P}, $\Pro_{H_{1} \cap W_{2}} x =\Pro_{H_{1}}x$ and   $ \Pro_{H_{1}}x \in W_{2}$. Hence, $\Pro_{W_{2}}\Pro_{H_{1}}x= \Pro_{H_{1}}x=\Pro_{H_{1} \cap W_{2}} x $.
Note that $x \in C^{c} =(W^{c}_{2})^{c}=W_{2}$.  Hence, $ \Pro_{H_{1}}\Pro_{W_{2}}x=\Pro_{H_{1}}x=\Pro_{H_{1} \cap W_{2}} x$.

\cref{prop:Candu1u20:C:completment}:   \cref{them:Formula:H1capW2}\cref{item:them:Formula:H1capW2:leq0} and \cref{lemma:PH1capW2:PH1H2}\cref{lemma:PH1capW2:PH1H2:P} yield  $\Pro_{H_{1} \cap W_{2}}x =\Pro_{H_{1}} x$ and  $\Pro_{H_{1}} x \in   W_{2}$. Hence,  \cref{prop:Candu1u20:C:completment} holds.

\cref{prop:Candu1u20:C:C}: Clearly, \cref{prop:Candu1u20:C} and \cref{remark:FactWFactH} lead to $\Pro_{W_{2}}\Pro_{H_{1}}x = \Pro_{H_{2}}\Pro_{H_{1}}x$.
Because  $\Pro_{H_{1}}x \in H_{1} \cap W^{c}_{2}$, apply \cref{lemma:x:H1:W2:Px}\cref{lemma:x:H1:W2:Px:W2} with $x = \Pro_{H_{1}}x$ to see that $ \Pro_{H_{1}}\Pro_{W_{2}}\Pro_{H_{1}}x  =\Pro_{H_{1}}\Pro_{H_{2}}\Pro_{H_{1}}x \notin W_{2}$. Hence,  apply \cref{prop:Candu1u20:C} with $x = \Pro_{W_{2}}\Pro_{H_{1}}x$ to know that $\Pro_{W_{2}}\Pro_{H_{1}}x \in C$.  Combine this with $x \in C$ and  \cref{them:Formula:H1capW2} to 
have that
$ \Pro_{H_{1} \cap W_{2}} x=\Pro_{H_{1} \cap H_{2}} x$ and $ \Pro_{H_{1} \cap W_{2}} \Pro_{W_{2}}\Pro_{H_{1}}x =\Pro_{H_{1} \cap H_{2}}\Pro_{W_{2}}\Pro_{H_{1}}x  $. Recall that $\Pro_{W_{2}}\Pro_{H_{1}}x = \Pro_{H_{2}}\Pro_{H_{1}}x$. Combine the last three identities with \cref{fact:ExchangeProj} used in the   third equation below to obtain that
\begin{align*}
\Pro_{H_{1} \cap W_{2}}\Pro_{W_{2}}\Pro_{H_{1}}x =\Pro_{H_{1} \cap H_{2}}\Pro_{W_{2}}\Pro_{H_{1}}x =\Pro_{H_{1} \cap H_{2}}\Pro_{H_{2}}\Pro_{H_{1}}x=\Pro_{H_{1} \cap H_{2}}x  =\Pro_{H_{1} \cap W_{2}}x.
\end{align*}
\end{proof}

\begin{theorem} \label{theom:PW2PH1:BAM}
	 Suppose that  $u_{1}$ and $u_{2}$ are linearly independent.
	Set 
	$\gamma:=\frac{|\innp{u_{1}, u_{2}}| }{\norm{u_{1}} \norm{u_{2}}}$.  Then $\gamma \in \left[0,1\right[\,$  and  $\Pro_{W_{2}}\Pro_{H_{1}}$ is a $\gamma$-BAM. 
		Consequently, 
		\begin{align*}
		(\forall x \in \mathcal{H}) (\forall k \in \mathbb{N}) \quad \norm{ (\Pro_{W_{2}}\Pro_{H_{1}})^{k}x - \Pro_{H_{1} \cap W_{2}}x} \leq \gamma^{k} \norm{x -\Pro_{H_{1} \cap W_{2}}x }.
		\end{align*}
\end{theorem}

\begin{proof}
	According to \cref{lem:H1H2Ineq},  $\gamma \in \left[0,1\right[\,$. 
	 Let $x \in \mathcal{H}$. 
	 By  \cite[Corollary~4.5.2]{Cegielski}, $\Fix \Pro_{W_{2}}\Pro_{H_{1}}=H_{1} \cap W_{2}$ is a nonempty closed convex subset of $\mathcal{H}$. Hence, using \cref{def:BAM} and \cref{fact:BAM:Properties}, we only need to prove the following two statements:
	\begin{enumerate}
		\item \label{item:theom:PW2PH1:BAM:eq} $\Pro_{H_{1} \cap W_{2}}\Pro_{W_{2}}\Pro_{H_{1}} x=\Pro_{H_{1} \cap W_{2}}x$.
		\item \label{item:theom:PW2PH1:BAM:ineq}  $\norm{\Pro_{W_{2}}\Pro_{H_{1}} x -\Pro_{H_{1} \cap W_{2}}x} \leq \gamma \norm{x- \Pro_{H_{1} \cap W_{2}}x}$.
	\end{enumerate}
If $x \in C^{c}$ or if $\innp{u_{1},u_{2}} =0$, then by \cref{prop:Candu1u20}\cref{prop:Candu1u20:C:completment}$\&$\cref{prop:Candu1u20:u1u2}, $ \Pro_{W_{2}}\Pro_{H_{1}} x=\Pro_{H_{1} \cap W_{2}} x$, which   implies \cref{item:theom:PW2PH1:BAM:eq} and \cref{item:theom:PW2PH1:BAM:ineq}.

Suppose that $x \in C$ and that $\innp{u_{1},u_{2}} \neq 0$. Then   \cref{prop:Candu1u20}\cref{prop:Candu1u20:C:C} implies that  \cref{item:theom:PW2PH1:BAM:eq} holds and that $\Pro_{W_{2}} \Pro_{H_{1}} x =\Pro_{H_{2}} \Pro_{H_{1}} x$ and 	$\Pro_{H_{1} \cap W_{2}}x =  \Pro_{H_{1} \cap H_{2}}x$.
Combine these identities   with  \cref{lem:H1H2Ineq} to obtain that
	\begin{align*}  
	\norm{\Pro_{W_{2}} \Pro_{H_{1}} x - \Pro_{H_{1} \cap W_{2}}x  } =	\norm{\Pro_{H_{2}} \Pro_{H_{1}} x - \Pro_{H_{1} \cap H_{2}}x  } \leq \gamma \norm{x - \Pro_{H_{1} \cap H_{2}}x }= \gamma \norm{x - \Pro_{H_{1} \cap W_{2}}x },
	\end{align*} 
	which yields
	\cref{item:theom:PW2PH1:BAM:ineq}. Altogether, the proof is complete.
	\end{proof}

	\begin{proposition} \label{prop:H1W2}
		 Suppose that  $u_{1}$ and $u_{2}$ are linearly independent. Denote by $\gamma :=\frac{|\innp{u_{1}, u_{2}}| }{\norm{u_{1}} \norm{u_{2}}}$. Then exactly one of the following holds:
		 \begin{enumerate}
		 	\item \label{prop:H1W2:i} $  \Pro_{H_{1}} \Pro_{W_{2}}x \in  W_{2}$. Then $  \Pro_{H_{1}} \Pro_{W_{2}}x \in  H_{1} \cap W_{2}$.
		 	\item \label{prop:H1W2:ii} $  \Pro_{H_{1}} \Pro_{W_{2}}x \notin   W_{2}$. Then $(\forall k \in \mathbb{N} )$ $\Pro_{W_{2}} (\Pro_{H_{1}} \Pro_{W_{2}})^{k}x   \notin H_{1}$ and $  (\Pro_{H_{1}}\Pro_{W_{2}})^{k+1}x  \notin W_{2}$. Moreover, for every  $k \in \mathbb{N}$,
		 	\begin{subequations}
		 		\begin{align}
		 		(\forall x \in W_{2}) \quad &\norm{ \Pro_{W_{2}} (\Pro_{H_{1}}\Pro_{W_{2}} )^{k} x- \Pro_{H_{1} \cap W_{2}}x } \leq \gamma^{k} \norm{x -\Pro_{H_{1} \cap W_{2}}x}, \label{eq:prop:H1W2:W2}\\
		 			(\forall x \in W^{c}_{2}) \quad  &\norm{  (\Pro_{H_{1}}\Pro_{W_{2}} )^{k} - \Pro_{H_{1} \cap W_{2}}x }  \leq \gamma^{k} \norm{x -\Pro_{H_{1} \cap W_{2}}x }.\label{eq:prop:H1W2:W2c}
		 			\end{align}
		 		\end{subequations}
		 \end{enumerate}
	\end{proposition}

\begin{proof}
	\cref{prop:H1W2:i}: This is trivial. 
	
	\cref{prop:H1W2:ii}:
Apply \cref{lem:H1H2Ineq} with swapping $H_{1}$ and $H_{2}$  to obtain that
	$\gamma \in \left[0,1\right[\,$ and  for every $x \in \mathcal{H}$ and $k \in \mathbb{N}$,
	\begin{subequations}
			\begin{align}
	 &\norm{ (\Pro_{H_{2}}  \Pro_{H_{1}} )^{k} x-  \Pro_{H_{1} \cap H_{2}}x } \leq \gamma^{k} \norm{x -\Pro_{H_{1} \cap H_{2}}x }, \label{eq:prop:H1W2:12}\\
 &\norm{ (\Pro_{H_{1}}  \Pro_{H_{2}} )^{k} x-  \Pro_{H_{1} \cap H_{2}}x } \leq \gamma^{k} \norm{x -\Pro_{H_{1} \cap H_{2}}x }.\label{eq:prop:H1W2:21}
		\end{align}
	\end{subequations}

According to \cref{lemma:PH1capW2:PH1H2}\cref{lemma:PH1capW2:PH1H2:notinW2},  $ \Pro_{H_{1} \cap W_{2}}x=\Pro_{H_{1} \cap H_{2}}x$. Set $y:=\Pro_{H_{1}} \Pro_{W_{2}}x \in H_{1} \cap  W^{c}_{2}$. 
Apply \cref{lemma:x:H1:W2:Px}\cref{lemma:x:H1:W2:Px:k}  to the point $y$ to obtain that for every $k \in \mathbb{N}$, ($k=0$ is trivial)
\begin{align}
& (\Pro_{W_{2}}\Pro_{H_{1}})^{k}y =	(\Pro_{H_{2}}\Pro_{H_{1}})^{k}y \notin H_{1}, \text{~i.e.,~} \Pro_{W_{2}}(\Pro_{H_{1}}\Pro_{W_{2}})^{k}x=(\Pro_{H_{2}}\Pro_{H_{1}})^{k}  \Pro_{W_{2}}x \notin H_{1} \text{ and} \label{eq:prop:H1W2:notinH1}\\
& \Pro_{H_{1}}(\Pro_{W_{2}}\Pro_{H_{1}})^{k}y =\Pro_{H_{1}}(\Pro_{H_{2}}\Pro_{H_{1}})^{k}y\notin W_{2},  \text{ i.e., } (\Pro_{H_{1}}\Pro_{W_{2}})^{k+1}  x=\Pro_{H_{1}}(\Pro_{H_{2}}\Pro_{H_{1}})^{k} \Pro_{W_{2}}x \notin W_{2}. \label{eq:prop:H1W2:notinW2}
\end{align}

Moreover, if $x \in W_{2}$, then by \cref{eq:prop:H1W2:notinH1}, $ \Pro_{W_{2}}(\Pro_{H_{1}}\Pro_{W_{2}})^{k}x=(\Pro_{H_{2}}\Pro_{H_{1}})^{k}  x$.  Hence,   \cref{eq:prop:H1W2:12} and $ \Pro_{H_{1} \cap W_{2}}x=\Pro_{H_{1} \cap H_{2}}x$ yield  \cref{eq:prop:H1W2:W2}.

If $x \notin W_{2}$, then by \cref{eq:prop:H1W2:notinW2}, $(\Pro_{H_{1}}\Pro_{W_{2}})^{k+1}  x=\Pro_{H_{1}}(\Pro_{H_{2}}\Pro_{H_{1}})^{k} \Pro_{H_{2}}x= (\Pro_{H_{1}}\Pro_{H_{2}})^{k+1}x $. Hence,   \cref{eq:prop:H1W2:21} and $ \Pro_{H_{1} \cap W_{2}}x=\Pro_{H_{1} \cap H_{2}}x$ imply that \cref{eq:prop:H1W2:W2c} is true.
	
	Altogether, the proof is complete.
\end{proof}

To conclude this section, we summarize the results obtained in this section in the following theorem. 
\begin{theorem} \label{theom:HyperplaneHalfspace:BAM}
Recall that $u_{1}$ and $u_{2}$ are in $\mathcal{H}$, that $\eta_{1}$ and $\eta_{2}$ are in $\mathbb{R}$, and that $H_{1} := \{x \in \mathcal{H} ~:~ \innp{x,u_{1}}= \eta_{1} \}$, $
W_{2} := \{x \in \mathcal{H} ~:~ \innp{x,u_{2}} \leq \eta_{2} \}$, $H_{2} := \{x \in \mathcal{H} ~:~ \innp{x,u_{2}} = \eta_{2} \}$.
	Then exactly one of the following statements hold.
	\begin{enumerate}
		\item $u_{1}$ and $u_{2}$ are linearly dependent. Then $\Pro_{W_{2}}\Pro_{H_{1}} =\Pro_{H_{1} \cap W_{2}} =\Pro_{H_{1}} \Pro_{W_{2}}$. $($See \cref{theor:u1u2LD:BAM}$).$
		\item $u_{1}$ and $u_{2}$ are linearly independent.  Then the following hold:
		\begin{enumerate}
			\item Suppose that $\innp{u_{1},u_{2}}=0$. Then  $\Pro_{W_{2}}\Pro_{H_{1}} =\Pro_{H_{1} \cap W_{2}} =\Pro_{H_{1}} \Pro_{W_{2}}$.$($See \cref{prop:Candu1u20}\cref{prop:Candu1u20:u1u2}$).$
			\item Suppose that $\innp{u_{1},u_{2}} \neq 0$. Denote by $\gamma :=\frac{|\innp{u_{1}, u_{2}}| }{\norm{u_{1}} \norm{u_{2}}}$. Then $\gamma \in \left[0,1\right[\,$. Moreover,
			\begin{enumerate}
				\item $(\forall x \in \mathcal{H})$ $(\forall k \in \mathbb{N}) $ $\norm{ (\Pro_{W_{2}}\Pro_{H_{1}})^{k}x - \Pro_{H_{1} \cap W_{2}}x} \leq \gamma^{k} \norm{x -\Pro_{H_{1} \cap W_{2}}x }$. $($See \cref{theom:PW2PH1:BAM}$).$
				\item  Let $x \in \mathcal{H}$ and $k \in \mathbb{N}$.  If $  \Pro_{H_{1}} \Pro_{W_{2}}x \notin H_{1} \cap W_{2}$, then 	$(\forall x \in W_{2})$   $\norm{ \Pro_{W_{2}} (\Pro_{H_{1}}\Pro_{W_{2}} )^{k} x-  \Pro_{H_{1} \cap W_{2}}x   } \leq \gamma^{k} \norm{x -\Pro_{H_{1} \cap W_{2}}x  }$, and $ (\forall x \in W^{c}_{2}) $   $\norm{  (\Pro_{H_{1}}\Pro_{W_{2}} )^{k} - \Pro_{H_{1} \cap W_{2}}x }  \leq \gamma^{k} \norm{x -\Pro_{H_{1} \cap W_{2}}x  }$.				
				  $($See \cref{prop:H1W2}$).$
			\end{enumerate}
		\end{enumerate}
	\end{enumerate}
\end{theorem}

\section{Conclusion and future work}

We provided an explicit formula of the projection onto intersection of finitely many hyperplanes.
We also shown KKT conditions for characterizing the optimal solution of convex optimization with finitely many inequality and equality constraints in Hilbert spaces. Moreover, we constructed  formulae of projections onto the intersection of hyperplane and halfspace. 
In addition, we systematically investigated the relations between: $\Pro_{W_{2}}\Pro_{W_{1}}$ and $ \Pro_{W_{1} \cap W_{2}}$,  $ \Pro_{H_{1} \cap W_{2}}$ and   $\Pro_{W_{2}}\Pro_{H_{1}}$, and  $ \Pro_{H_{1} \cap W_{2}}$  and $\Pro_{H_{1}} \Pro_{W_{2}}$. 

According to the  explicit formulae for  $\Pro_{ W_{1} \cap W_{2}}$ and $\Pro_{ H_{1} \cap W_{2}  }$ (see, \cref{fact:W1capW2:u1u2LD}, \cref{fact:Projec:Intersect:halfspace}, \cref{theor:u1u2LD:PBAM}, and \cref{them:Formula:H1capW2}), given a point $x \in \mathcal{H}$, the formulae of $\Pro_{ W_{1} \cap W_{2}}x$ and $\Pro_{ H_{1} \cap W_{2}  }x$  depend on the linear dependence relation of $u_{1}$ and $u_{2}$, and on the \enquote{region} where the $x$ is located in.  Hence, in many proofs of  this work, we mainly argued by cases and considered our questions on two halfspaces, or on one halfspace and one hyperplane. It is easy to see that if we wanted to use the current logic of proofs to extend our results from two sets to finitely many sets, then the number of cases to argue would increase exponentially. In the future, we shall try to find some techniques or tricks to make the extension work easy and the statements of the results simple.

\section*{Acknowledgements}
The author thanks the anonymous referee for his or her valuable comments which significantly improved this manuscript.

\addcontentsline{toc}{section}{References}

\bibliographystyle{abbrv}

\end{document}